\documentclass[11pt]{article}  
\usepackage[textsize=small]{todonotes}

\usepackage{graphicx}
           
\usepackage{color}              
\usepackage{amsmath,amssymb,bbm,enumerate,amsthm}

\usepackage{tikz}

\definecolor{MyDarkBlue}{rgb}{0,0.08,0.50}
\definecolor{BrickRed}{rgb}{0.65,0.08,0}

\usepackage{hyperref}
\hypersetup{
  colorlinks=true,       
  linkcolor=MyDarkBlue,          
  citecolor=BrickRed,       
  filecolor=red,      
  urlcolor=cyan           
}

\theoremstyle{plain}
\newtheorem{innercustomthm}{Theorem}
\newenvironment{customthm}[1]
  {\renewcommand\theinnercustomthm{#1}\innercustomthm}
  {\endinnercustomthm}
  
\newtheorem{innercustomcorollary}{Corollary}
\newenvironment{customcorollary}[1]{\renewcommand\theinnercustomcorollary{#1}\innercustomcorollary}{\endinnercustomcorollary}

\newtheorem{COND}{Condition}

\newtheorem{LEM}{Lemma}[section]
\newtheorem{PRP}[LEM]{Proposition}
\newtheorem{CON}[LEM]{Conjecture}

\newtheorem{REM}[LEM]{Remark}
\newtheorem{REMS}[LEM]{Remarks}

\newtheorem{COR}[LEM]{Corollary}
\newtheorem{DEF}[LEM]{Definition}

\renewcommand{\P}{\mathbb{P}}
\newcommand{\E}{\mathbb{E}}
\newcommand{\Z}{\mathbb{Z}}
\newcommand{\R}{\mathbb{R}}
\newcommand{\C}{\mathbb{C}}
\newcommand{\T}{\mathbb{T}}
\newcommand{\ra}{\rightarrow}
\newcommand{\nn}{\nonumber}

\newcommand{\mc}[1]{\mathcal {#1}}

\newcommand{\sss}{\scriptscriptstyle}
\newcommand{\sn}{\sss {(n)}}
\newcommand\1{\mathbbm{1}}
\newcommand{\indic}[1]{\1_{\{#1\}}}
\newcommand {\cweak}{\overset{w}{\longrightarrow}}

\newcommand {\cas}{\overset{a.s.}{\longrightarrow}}

\newcommand{\supp}{\mathrm{supp}}

\newcommand{\dd}{\mathrm{d}}
\newcommand{\ee}{\mathrm{e}}

\numberwithin{equation}{section}

\newcommand{\floor}[1]{\lfloor #1 \rfloor}
\newcommand{\ceil}[1]{\lceil #1 \rceil}
\newcommand{\N}{\mathbb{N}}
\newcommand{\blank}[1]{}

\newcommand{\ttt}[2]{\tilde{t}^{\sss\{#1\}}_{#2}}
\newcommand{\ppp}[2]{\tilde{\pi}^{\sss\{#1\}}_{#2}}
\newcommand{\tplus}[2]{\tilde{t}^{\sss\{#1\}}_{#2,\sss{+}}}
\newcommand{\pplus}[2]{\tilde{\pi}^{\sss\{#1\}}_{#2 ,\sss{+}}}
\newcommand{\vep}{\varepsilon}
\newcommand{\st}{\,{\sss \times}\,}
\newcommand{\ms}{\hat{m}}

\newcommand{\Ed}[1]{{\bf \color{MyDarkBlue} #1}}
\newcommand{\bs}[1]{\boldsymbol{#1}}
\newcommand{\bbt}{t^*}

\newcommand{\sumtwo}[2]{\sum_{\substack{#1 \\ #2}}}

\newcommand{\sm}[3]{\frac{#1}{#2^{#3}}}

\newcommand{\Omit}[1]{#1}
\renewcommand{\Omit}[1]{}

\makeatletter
\renewcommand\@makefnmark{\hbox{\@textsuperscript{\normalfont\color{magenta}\@thefnmark}}}
\makeatother

\newcommand{\mT}{\mathcal{T}}

\newcommand{\bT}{\bs{\mT}}
\newcommand{\ara}{\overset{\bs{a}}{\ra}}
\newcommand{\aran}{\overset{\bs{a,n}}{\ra}}
\newcommand{\arano}{\overset{\bs{a,2^{n_0}}}{\ra}}
\newcommand{\BB}[2]{\bar{#1}^{#2}_\infty}
\newcommand{\xint}{\BB{X}{}}
\newcommand{\xnint}{\BB{X}{\sn}}
\newcommand{\PS}{\P^{\sss\{s\}}}

\newcommand{\ARXIV}[1]{#1}
\newcommand{\SUBMIT}[1]{#1}
\renewcommand{\SUBMIT}[1]{}

\begin{document}

\author{
Mark Holmes\footnote{School of Mathematics and Statistics,
The University of Melbourne. {\tt holmes.m@unimelb.edu.au}}
\and Edwin Perkins \footnote{Department of Mathematics, The University of British Columbia.  {\tt perkins@math.ubc.ca}}
}

\title{On the range of lattice models\\
 in high dimensions \ARXIV{- extended version}}

\maketitle

\begin{abstract}
In this paper we investigate the scaling limit of the {\em range} (the set of visited vertices) for a class of critical lattice models, starting from a single initial particle at the origin.  We give conditions on the random sets and an associated ``ancestral relation" under which, conditional on longterm survival, the rescaled ranges converge weakly to the range of super-Brownian motion as random sets. These hypotheses also give 
 precise asymptotics for the limiting behaviour of exiting a large ball, that is for the {\em extrinsic one-arm probabililty}.  
  We show that these conditions are satisfied by the voter model in dimensions $d\ge2$ and critical sufficiently spread out lattice trees in dimensions $d>8$.  The latter result also has important consequences for the behaviour of random walks on lattice trees in high dimensions.

We conjecture that our conditions  are also satisfied by other models (at criticality above the critical dimension) such as sufficiently spread out oriented percolation and contact processes  in dimensions $d>4$.   \ARXIV{This version of the paper contains details not present in the submitted version \cite{submitted_version}.}
\end{abstract}
\tableofcontents
\section{Introduction}
\label{sec-setting}
Super-Brownian motion is a measure-valued process arising as a universal scaling limit for a variety of critical lattice models above the critical dimension in statistical physics and mathematical biology.  Examples include oriented percolation (\cite{HofSla03b}), lattice trees (\cite{H08}), models for competing species such as voter models (\cite{CDP00},\cite{BCLG01}), models for spread of disease such as  contact processes (\cite{HofSak10}), and percolation (\cite{HS00}),
where the full result in the latter context is the subject of ongoing research (e.g., \cite{HHHM17}).
The nature of the convergence in all these contexts is that of convergence of associated 
empirical processes to super-Brownian motion. Moreover, here often only convergence of the finite-dimensional distributions is known.  Extending this to convergence on path space for lattice trees was recently carried out in \cite{HHP17} with great effort.  Convergence of the actual random sets of occupied sites to the range of super-Brownian motion is one of the most natural questions but has not been accomplished in any of these settings (convergence at a fixed time was done for the voter model in \cite{BCLG01}, and for the simple setting of branching random walk it is implicit in \cite{DIP89}). 
We provide a unified solution to this problem in the form of quite general conditions under which the rescaled ranges of a single occupancy particle model on the integer lattice (in discrete or continuous time) 
converge to the range of super-Brownian motion.  The conditions include convergence of 
the associated integrated measure-valued processes to integrated super-Brownian motion, but a feature
of our results is that convergence of finite-dimensional distributions suffices (see Lemma~\ref{lem:int_fdd} below).  We verify the conditions for the voter model 
in two or more dimensions, and for critical lattice trees in more than eight dimensions. We conjecture that our general conditions also hold for the critical contact process and critical oriented percolation, both sufficiently spread out, in more than four dimensions. In fact, we \SUBMIT{can easily }verify all but one of the conditions for the latter, thus reducing the problem for oriented percolation  to a natural bound on the sixth spatial moment on the mean measure.   Our general lattice models include a random ``ancestral relation" which in the case of random graphs such as lattice trees or oriented percolation is a fundamental part of the model, but for particle models such as the voter model or contact process, arises naturally from the graphical construction of such models.

We begin by briefly defining the two models to which our results will be applied.  These models depend on a random walk step kernel, $D:\Z^d\ra \R_+$, with finite range and covariance matrix $\sigma^2 I_{d\times d}$ for some $\sigma^2>0$, and such that $D(-x)=D(x)$ (with $D(o)=0$, where $o=(0,\dots,0)$ is the origin).  By finite range we mean there is an $L>0$ such that $D(x)=0$ if $\|x\|>L$ where $\|x\|$ is the $L_\infty$ norm of $x$. 
The nearest neighbour case is where $D(x)=(2d)^{-1}\indic{|x|=1}$, where $|x|$ is the Euclidean norm of $x$.
\subsubsection*{{\em The voter model}}
The {\em voter model} on $\Z^d$ (introduced in \cite{CS73} and \cite{HL75}) is a  spin-flip system, and so in particular, a continuous time Feller process $(\xi_t)_{t\ge 0}$ with state space $\{0,1\}^{\Z^d}$ and flip rates as follows.  With rate one each vertex, say at $x$, imposes its type ($0$ or $1$) on a randomly chosen vertex $y$ with probability $D(y-x)$. Let $\xi_t(x)\in \{0,1\}$ denote the type of $x\in \Z^d$ at time $t\ge 0$, and let $\mT_t:=\{x \in \Z^d:\xi_t(x)=1\}$.  In the notation of \cite{Li85} the flip rate at site $x$ in state $\xi$ is
\[c(x,\xi)=\sum_{y:\xi(y)\neq \xi(x)}D(y-x).\]
   If $\E[|\mT_0|]<\infty$, then $|\mT_t|$ is a non-negative martingale and the extinction time $S^{\sss(1)}=\inf\{t\ge 0:|\mT_t|=0\}$ is almost surely finite (see Lemma~\ref{markovmart}(b) below). Here $|A|$ denotes the cardinality of a finite set $A$.   We will usually assume that the process starts with a single site with type 1 at time 0, located at the origin $o$, i.e.
\begin{equation}
\P(\mT_0=\{o\})=1\label{single1}.\nn
\end{equation}

\subsubsection*{{\em Lattice trees}}
A lattice tree $T$ on $\Z^d$, is a finite connected simple graph in $\Z^d$ with no cycles. In particular, it consists of a  set of lattice bonds, $E(T)$ (unordered pairs of points in $\Z^d$), together with the corresponding set of end-vertices, $V(T)$,  in $\Z^d$. By connected we mean that for any distinct $v_1,v_2\in V(T)$ there is an $m \in \N$ and a function $w:\{0,\dots,m\}\to V(T)$ so that $w(0)=v_1$, $w(m)=v_2$, and for all $1\le k\le m$, $\{w(k-1),w(k)\}\in E(T)$. We call $w$ a path in $T$ of length $m$ from $v_1$ to $v_2$. Given any two vertices $v_1,v_2$ in the tree, the lack of cycles means there is a unique path (length $0$ if $v_1=v_2$) of bonds connecting $v_1$ and $v_2$.  The number of such bonds, $d_T(v_1,v_2)$, is the tree distance between $v_1$ and $v_2$. It is a metric on the set of vertices, called the tree metric.   
For some $L>0$ our lattice trees will be in the countable space $\T_L(o)$ of lattice trees on $\Z^d$ whose vertex set contains the origin $o$ 
and for which every bond has $L_\infty$-norm at most $L$.   More generally we let $\T_L(x)$ denote the space of lattice trees on $\Z^d$ whose vertex set contains  $x\in \Z^d$ (and bonds as above). 
 $L$ will be taken sufficiently large for our main results.  We now describe a way of choosing a ``random'' lattice tree $\mc{T}$ in $\T_L:=\T_L(o)$. 

Let $d>8$ and let $D(\cdot)$ be the uniform distribution on a finite box $[-L,L]^d\setminus o$.  For a lattice tree $T\in\T_L(o)$ define 
\begin{equation}\label{wtdefn}W_{z,D}(T)=z^{|T|}\prod_{e\in E(T)}D(e), 
\end{equation}
where for $e=(y,x)$, $D(e):=D(x-y)$, and $|T|$ is the number of edges in $T$.  For any $z>0$ such that $\rho_z:=\sum_{T\in\T_L(o)}W_{z,D}(T)<\infty$ we can define a probability on $\T_L(o)$ by
$\P_z(\mT=T)=\rho_z^{-1}W_{z,D}(T)$.
It turns out (see e.g.~\cite{H08,HH13}) that there exists a critical value $z_{\sss D}$ such that $\rho_{z_{\sss D}}<\infty$, $\rho_z=\infty$ for $z>z_{\sss D}$,  and $\E_{z_{\sss D}}[|\mT|]=\infty$.  Hereafter we write $W(\cdot)$ for the critical weighting $W_{z_{\sss D},D}(\cdot)$, write $\rho:=\rho_{z_{\sss D}}$ and $\P=\P_{z_{\sss D}}$, and we select a random tree $\mT$ according to this critical weighting. We will also define $W(T)$ for $T\in\T_L(x)$ by \eqref{wtdefn}. 

For  $T\in\T_L(o)$ and $m \in \Z_+$,  let $T_{m}$ denote the set of vertices in $T$ of tree distance $m$ from $o$ and, in particular, $\mT_m$ is the corresponding set of vertices for our random tree $\mT$.   Note that $\P(\mT_0=\{o\})=1$.

\subsubsection*{{\em General models and ancestral relations}}

Although the voter model and lattice trees will be our prototypes in continuous and discrete time, respectively, our goal is to establish general conditions for convergence of the ranges of a wide class of rescaled lattice models (including oriented percolation and the contact process, and perhaps also percolation).  We introduce our general framework in this section. The time index $I$ will either be $\Z_+$ (discrete time) or $[0,\infty)$ (continuous time). We use the notation $I_t=\{s \in I:s\le t\}$.  As we will be dealing with
random compact sets,  we let $\mc{K}$ denote the set of compact subsets of $\R^d$. We
equip it with the Hausdorff metric $d_0$ (and note that $(\mc{K},d_0)$ is Polish) defined by $d_0(\varnothing,K)=d_0(K,\varnothing)=1$ for $K\ne \varnothing$, while for $K,K'\ne \varnothing$
\begin{align}
\nn d_0(K,K')&=d_1(K,K')\wedge 1, \quad \text{ where }\\
\nn d_1(K,K')&:=\Delta_1(K,K')+\Delta_1(K',K), \\
\label{Delta'def}\Delta_1(K,K')&:=\inf\big\{\delta>0:K\subset \{x:d(x,K')\le \delta\}\big\},\quad \text{ and }\\
\nn d(x,K)&:=\inf\{|x-y|:y \in K\}.
\end{align}
Although $d_1$ is the usual Hausdorff metric, it is easy to check $d_0$ is also a complete metric.

As our models of interest will be single occupancy models, we assume throughout that
\begin{align}
\label{cond1}
&\bT=(\mT_t)_{t\in I} \text{ is a stochastic process taking values in the finite}\\
\nn&\text{subsets of }\Z^d\text{ such that }\mT_0=\{o\},
\text{and in continuous time the}\\
\nn&\text{sample paths are cadlag }\mc{K}-\text{valued.}
\end{align}
\noindent {\bf Notation.} For a metric space $M$, $\mc{D}([0,\infty),M)$ will denote the space of cadlag $M$-valued paths with the Skorokhod topology and $C_b(M)$ is the space of bounded $\R$-valued continuous functions on $M$. $C^2_b(\R^d)$ is the set of bounded continuous functions whose first and second order partials are also in $C_b(\R^d)$.  Integration of $f$ with respect to a measure $\mu$ is often denoted by $\mu(f)$.
\medskip

Cadlag paths are bounded on bounded intervals and so this implies
\begin{equation}
\label{finiterange}
\text{for any $t\in I$, }\cup_{s\in I_t}\mT_s\text{ is a finite subset of }\Z^d.
\end{equation}
We will write
\[(t,x)\in\vec \mT\text{ if and only if }x\in \mT_t,\text{ where }(t,x)\in I\times\Z^d.\]
$(\mc{F}_t)_{t\in I}$ will denote a filtration with respect to which $(\mT_t)_{t \in I}$ is adapted. In practice it may be larger than the filtration generated by $\bT$. 

A random ancestral relation, $(s,y)\ara(t,x)$, on $I\times\Z^d$ will be fundamental to our analysis. If it holds we say that $(s,y)$ is an ancestor of $(t,x)$, and it will imply $s\le t$. We write 
\[(s_1,y_1)\ara\dots\ara(s_N,y_N)\text{ iff }(s_i,y_i)\ara(s_{i+1},y_{i+1})\text{ for }i=1,\dots,N-1,\]
and define (for $s,t\in I$, $x,y\in\Z^d$),
\begin{align}\label{edefn}e_{s,t}(y,x)=\begin{cases}\1((s,y)\ara(t,x))&\text{ if }s< t\\
\1(x=y\in\mT_t)&\text{ if }s\ge t, 
\end{cases}
\text{ and }\hat e_t(y,x)(s)=e_{s,t}(y,x).
\end{align}
We will assume $\ara$ satisfies the following conditions where (AR)(i)-(iii) will hold off a single null set which we usually ignore:
\begin{align}
\text{(AR)}(i)\nn&\quad \text{For all $(s,y),(t,x)\in I\times\Z^d$}:\\
\label{eqara}&\quad   (s,y)\ara(s,x)\text{ iff }x=y\in\mT_s,\\ 
\label{arbasica}&\quad(s,y)\ara(t,x) \text{ implies }
s\le t, y\in\mT_s,\text{ and }x\in\mT_t, \\
\label{arbasic0}&\quad (0,o)\ara(t,x)\text{ iff }x\in\mT_t.\\
\nn(ii)&\quad \text{For any }0\le s_1<s_2<s_3\text{ in }I \text{ and }y_1,y_2,y_3\in\Z^d:\\
\label{transitive}& \quad (s_1,y_1)\ara(s_2,y_2)\ara(s_3,y_3)
\text{ implies}\ (s_1,y_1)\ara(s_3,y_3).\\
\label{convtr}&\quad\text{Conversely if }(s_1,y_1)\ara(s_3,y_3)\text{ then }\exists y_{2}\in\mT_{s_2} \text{ s.t. }\\
\nn&\quad (s_1,y_1)\ara(s_2,y_2)\ara(s_3,y_3).\\
\label{cadlage}(iii)&\quad \text{If } I=[0,\infty),\text{ then for every $x,y\in\Z^d$:} \\
\nn&\quad \hat e_t(y,x)\in \mc{D}([0,\infty),\R)=:\mc{D}_{\R}, \text{ for every $t \in I$, and }\\
\nn&\quad t\mapsto\hat e_t(y,x)\in \mc{D}([0,\infty),\mc{D}_{\R}).\\
\label{wadapt}(iv)&\quad e_{s,t}(y,x)\text{ is }\mc{F}_t-\text{measurable for all }s, t\text{ in }I \text{ and }x,y\in \Z^d.
\end{align}
We call $\ara$ an ancestral relation iff (AR) holds. In this case we call $(\bT,\ara)$ an ancestral system.
\begin{REMS}\label{APremark} 
\emph{(1) It is immediate from \eqref{arbasic0} and \eqref{convtr} (the latter with $(s_1,s_2,s_3)=(0,s,t)$) that
\begin{equation}\label{death}
\mT_s=\varnothing\Rightarrow \mT_t=\varnothing\ \forall t\ge s.
\end{equation}
\noindent (2) In practice it is often easiest to verify (AR)(iii) by showing $s\mapsto e_{s,t}(y,x)$ is cadlag for each $t,x,y$, and that 
\begin{align}\nn&\text{For each }t\ge 0,x,y\in\Z^d\text{ there is a }\delta>0\text{ s.t. }\\
\label{estepr}&\hat e_u(y,x)=\hat e_t(y,x),\ \forall u\in[t,t+\delta)\ \text{ and }\\
\label{estepl} &\hat e_u(y,x)=\hat e_{u'}(y,x),\ \forall u,u'\in(t-\delta,t)\cap[0,\infty).
\end{align}
}
\end{REMS}

{\bf We will always assume \eqref{cond1} and (AR) when dealing with our abstract models.}
\medskip

In the discrete time case we can extend $\mT_t$ and $\mc{F}_t$ to $t\in[0,\infty)$ by 
\begin{equation}
\nn
\mT_t=\mT_{\lfloor t\rfloor},\quad \vec \mT=\{(t,x):x\in \mT_t,t\ge 0\},\quad \mc{F}_t=\mc{F}_{\lfloor t\rfloor},
\end{equation}
and define $(s,y)\ara(t,x)$ for all $t\ge s\ge 0$ by 
\begin{equation}\label{winterp}
(s,y)\ara (t,x)\text{ iff }(\floor{s},y)\ara(\floor{t},x).
\end{equation}
It is easy to check then that in the discrete case (AR)(i)-(iv) 
 hold,  where now $s,s_i,t$ are allowed
to take values in $[0,\infty)$. Moreover \eqref{estepr} and \eqref{estepl} hold.  
\medskip

\noindent{\bf Note: We allow $n$ to denote a real parameter in $[1,\infty)$.}

\medskip
\noindent For $A\subset \R^d$ and $a\in \R$ define $aA=\{ax:x \in A\}$.  
To rescale our model for $n\in [1,\infty)$ we set
\[\mT^{\sn}_t=\mT_{nt}/\sqrt n,\text{ for }t\ge 0,\]
and for $s,t\ge 0$, $x,y\in\Z^d/\sqrt n$
\begin{equation}
\nn
\text{we write }(s,y)\aran(t,x)\text{ iff }(ns,\sqrt n y)\ara(nt,\sqrt n x).
\end{equation}
We also define for $n\in [1,\infty)$, $s,t\ge 0$ and $x,y\in\Z^d/\sqrt{n}$,
\begin{equation}
\hat e^{\sn}_t(y,x)(s)=e^{\sn}_{s,t}(y,x)=e_{ns,nt}(\sqrt n y,\sqrt n x),\label{endef}
\end{equation}
and note that $\hat e^{\sn}_t(y,x)(s)=\hat e_{nt}(\sqrt n y,\sqrt n x)(ns)$.

Here is a simple consequence of (AR)(ii) which will be used frequently.
\begin{LEM}\label{anseq}
W.p.1 if $n\in [1,\infty)$, $M\in\N$, $0\le s_0<s_1<\dots<s_M$, and $(s_0,y_0)\aran(s_M,y_M)$, then there are $y_1\in\mT^{\sn}_{s_1},\dots, y_{M-1}\in \mT^{\sn}_{s_{M-1}}$ 
s.t. $(s_{i-1},y_{i-1})\aran(s_i,y_i)$ for $i=1,\dots,M$.
\end{LEM}
\begin{proof} Fix $\omega$ s.t. (AR)(i)-(iii) hold.  By scaling we may assume $n=1$.  By \eqref{convtr} there is a $y_1\in\mT_{s_1}$ s.t. $(s_0,y_0)\ara(s_1,y_1)\ara(s_M,y_M)$.  Repeat this argument $M-2$ times to construct the required sequence.
\end{proof}
\begin{DEF} 
\label{def:ancestral_path}
An {\em ancestral path} to $(t,x)\in \vec{\mT}$ is a cadlag path $w=(w_s)_{s \ge 0}$ for which $(s,w_s)\ara (s',w_{s'})$ for every $0\le s\le s'\le t$, and $w_s=x$ for all $s\ge t$.
The random collection of all ancestral paths to points in $\vec{\mT}$ is denoted by $\mc{W}$ and is called the system of ancestral paths for $(\bT,\ara)$.

If $n\in[1,\infty)$ and $w$ is an ancestral path to $(nt,\sqrt n x)$, we define the rescaled ancestral path $w^{\sn}$ by $w^{\sn}_s=w_{ns}/\sqrt n$, and call $w^{\sn}$ an ancestral path to $(t,x)\in\mT^{\sn}_t$. 
\end{DEF}

\begin{REM} \label{rem:aps}
It is easy to check that if $I=\Z_+$ then \eqref{winterp} and \eqref{eqara} imply that for any ancestral path $w\in \mc{W}$, $w_s=w_{\floor{s}}$ for all $s\ge 0$.
For this reason we will often restrict our ancestral paths to $s\in \Z_+$. 
\end{REM}
\begin{PRP}\label{prop:wexist} With probability, $1$ for any $(t,x)\in\vec{\mT}$, $\mc{W}$ includes at least one ancestral path to $(t,x)$.
\end{PRP}
The elementary proof is given in Section~\ref{sec:onconds} below. 

Let us briefly consider \eqref{cond1} and (AR) for our prototypes, lattice trees and the voter model, 
introduced earlier. For lattice trees \eqref{cond1} is immediate. 
Since a lattice tree $T\in\T_L$ is a tree, for any $x\in T_{m}$ there is a unique ``ancestral" path $w(m,x)=(w_k(m,x))_{k\le m}$ of length $m$ in the tree from $o$ to $x$. Moreover $w_k(m,x)\in T_k$ for all $0\le k\le m$. Define $(k,y)\ara(m,x)$ iff 
$x\in \mT_m$, $0\le k\le m$, and $w_k(m,x)=y$. Here we allow $m=0$. 
AR(i)  and AR(ii) are then elementary to verify.  
 It remains to verify AR(iv) which is deferred to Section~\ref{sec:LT} where the definition of $\mc{F}_t$ is also given.

For the voter model there will also be a unique ancestral path $w(t,x)$ for each $(t,x)\in \vec{\mT}$ (see Lemma~\ref{lem:vw_tree}).  This path is obtained 
by tracing back the opinion $1$ at $x$ at time $t$ to
its source at the origin at time $0$.  Formally the ancestral paths are defined by reversing the dual system of coalescing random walks, obtained from the graphical construction of the voter model.  We then define $(s,y)\ara(t,x)$ iff $0\le s\le t$, $x\in\mT_t$ and $w_s(t,x)=y$. This standard
construction is described in Section~\ref{sec:voter} where (AR) and \eqref{cond1} are then verified (see Lemmas~\ref{markovmart} and \ref{lem:vmara}).

\subsubsection*{{\em Survival probability and measure-valued processes}}
The survival time of our scaled model $\mT^{\sn}$ is
$$S^{\sn}=\inf\{t\ge 0:\mT^{\sn}_t=\varnothing\},$$
so that $\mT^{\sn}_t=\varnothing$ for all $t\ge S^{\sn}$ by \eqref{death} and \eqref{cond1}.
The unscaled survival time is $S^{\sss(1)}$, and for $t>0$, the unscaled survival probability is defined as 
\begin{equation}
\nn
\theta(t):=\P(S^{\sss(1)}>t).\end{equation}
Our main results require a number of conditions on ${\bs{\mT}}$, the first of which concerns the asymptotics of  the survival probability.

\medskip
\noindent{\bf Notation.} Write $f(t)\sim g(t)$ as $t\uparrow \infty$ iff $\lim_{t\to\infty}f(t)/g(t)=1$.  Similarly for $f(t)\sim g(t)$ as $t\downarrow 0$.

\begin{COND}
\label{cond:surv}
There is a constant $s_D>0$ and a non-decreasing function $m:[0,\infty)\to(0,\infty)$ such that  $m(t)\uparrow \infty$, as $t\uparrow\infty$
\begin{align}
\label{survival}&\theta(t)\sim \frac{s_D}{m(t)}\text{ as }t\uparrow\infty,
\end{align}
\begin{equation}\label{mregvar}
\text{for any }s>0,\ \lim_{t\to\infty}m(st)/m(t)=s,
\end{equation}
and a constant $c_{\ref{mdef}}\ge 1$ such that
\begin{equation}
\label{mdef}\frac{m(sn)}{m(n)}\le c_{\ref{mdef}}s, \, \forall s,n\ge 1,\quad
\frac{m(n)}{m(sn)}\le c_{\ref{mdef}}\frac{1}{s}, \, \text{ for }1\le s\le n.
\end{equation}
\end{COND}
The monotonicity properties of $m$ and $\theta$ and \eqref{survival} easily show that 
\begin{equation}\label{survivbnds}
0<\underline s_D=\inf_{t\ge 0}m(t)\theta(t)\le \sup_{t\ge 0}m(t)\theta(t)=\overline s_D<\infty,
\end{equation}
Note also that the first inequality in \eqref{mdef} with $n=1$ implies that 
\begin{equation}\label{mboundi}
m(s)\le c_{\ref{mboundi}}s, \text{ for all } s\ge 1.
\end{equation}

For lattice trees with $d>8$ we set
\begin{equation}\label{mLTs}
m(t)=m^{LT}(t)=A^2V(t\vee 1),
\end{equation}
where $A,V>0$ are constants that depend on $D$; $A$ gives the asymptotic expected number of particles alive at time $n$ and $V$ is called the vertex factor (see \cite{HH13}).  

For the voter model in two or more dimensions we set
\begin{equation}
\nn
m(t)=m^{VM}(t)=\begin{cases} t\vee 1 &\text{ if }d>2\\
\frac{t\vee e}{\log(t\vee e)}&\text{ if }d=2,
\end{cases}\end{equation}
and
\begin{equation}
\nn
0<\beta_d=\begin{cases}P_o(S_n\neq o\ \ \forall n\in\N)&\text{ if }d>2\\
2\pi\sigma^2 &\text{ if } d=2.
\end{cases}
\end{equation}
In the above under $P_o$, $\{S_n\}$ is a discrete-time random walk with step distribution $D$, started at $o$, and we recall that $\sigma^2I_{d\times d}$ is the covariance matrix of $D$.
\begin{PRP}\label{cond1check}
(a) Condition~\ref{cond:surv} holds for critical sufficiently spread out lattice trees in dimension $d>8$ with $s_D=2A$.\\
(b) Condition~\ref{cond:surv} holds for the voter model in dimension $d>1$ with $s_D=\beta_d^{-1}$.
\end{PRP}
\begin{proof} Conditions \eqref{mregvar} and \eqref{mdef} are obvious in both cases. \\
(a) \eqref{survival} is a special case of Theorem~1.4 of \cite{HH13}.\\
(b) Theorem~1' of \cite{BG80} (or (1.5) of \cite{BCLG01}) gives \eqref{survival} for the voter model.
\end{proof}

We can reinterpret the state of our rescaled models in terms of an empirical measure 
given by 
\begin{equation}
X^{\sn}_t=\frac{1}{m(n)}\sum_{x\in \mT^{\sn}_t}\delta_x=\frac{1}{m(n)}\sum_{x\in \mT_{nt}}\delta_{x/\sqrt n}.\nn
\end{equation}
So $X_t^{\sn}$ takes values in the Polish space $\mc{M}_F(\R^d)$ of finite measures on $\R^d$ equipped
with the topology of weak convergence. It follows from \eqref{cond1} that the measure-valued
process
 $\bs{X}^{\sn}=(X_t^{\sn})_{t\ge 0}$ is in the Polish space $\mc{D}:=\mc{D}([0,\infty),\mc{M}_F(\R^d))$.
 We define the survival map $\mc{S}:\mc{D} \ra [0,\infty)$ for $\bs{\nu}=(\nu_t)_{t\ge 0}\in \mc{D}$ by 
\begin{equation}
\mc{S}(\bs{\nu})=\inf\{t>0:\nu_t(\R^d)=0\},\nn
\end{equation}
so that our survival times satisfy
$S^{\sn}=\mc{S}(\bs{X}^{\sn})$, 
and $X^{\sn}_t=0$ for all $t\ge S^{\sn}$ by \eqref{death}.

\subsubsection*{{\em Weak convergence and super-Brownian motion}}
An adapted a.s.~continuous $\mc{M}_F(\R^d)$-valued process, $\bs{X}=(X_s)_{s\ge 0}$, on a complete filtered probability space $(\Omega,\mc{F},\mc{F}_t,\P_{X_0})$ is said to be a super-Brownian motion (SBM) with branching rate $\gamma>0$ and diffusion parameter $\sigma_0^2>0$ (or a $(\gamma,\sigma_0^2)$-SBM) starting at $X_0\in\mc{M}_F(\R^d)$ iff it solves the following martingale problem:
\begin{align}
\nn
&\forall \phi\in C_b^2(\R^d), M_t(\phi)=X_t(\phi)-X_0(\phi)-\int_0^t X_s(\sigma_0^2\Delta\phi/2)\,ds\\
\nonumber&\text{is a continuous }\mc{F}_t-\text{martingale starting at $0$, and with square} \\
\nonumber&\text{function }\langle M(\phi)\rangle_t=\int_0^tX_s(\gamma\phi^2)ds.
\end{align}
See Section II.5 of \cite{Per02} for the well-posedness of the above martingale problem.  Let $S=\mc{S}(\bs{X})$.
Associated with such a SBM is a $\sigma$-finite measure, $\N_o=\N_o^{\gamma,\sigma_0^2}$, on the space of continuous measure-valued paths satisfying $\nu_0=0$, $0<S<\infty$ and $\nu_s=0$ for all $s\ge S$; let $\Omega^{Ex}_C$ denote the space of such paths. 
$\N_o$ is called the canonical measure for super-Brownian motion. 
The connection between $\N_o$ and super-Brownian motion
is that if $\Xi$ is a Poisson point process on the space $\Omega_C^{Ex}$ with intensity $\N_o$, then
\begin{equation}\label{sbmppp}X_t=\int\nu_t\,d\Xi(\bs{\nu}), t>0;\quad X_0=\delta_0
\end{equation}
defines a SBM starting at $\delta_0$.  
It is known that
\begin{equation}\label{sbmext}
\N_o(S>s)=\frac{2}{\gamma s}<\infty\text{ for all $s>0$}.
\end{equation}
Intuitively $\N_o$ governs the evolution of the descendants of a single ancestor at time zero, starting from the origin.  For the above and more information on the canonical measure of super-Brownian motion see,  e.g., Section II.7 of  \cite{Per02}.  
We will sometimes work with the unconditioned measures ($n\in[1,\infty)$)
\[\mu_n(\cdot)=m(n)\P(\cdot).\]   
 Note that \eqref{survival} and \eqref{mregvar} of Condition~\ref{cond:surv}  together imply 
 \begin{equation}\label{cond1'}
 \text{for each }s>0,\quad 
\lim_{n\to\infty}\frac{2}{ s_D}\mu_n(S^{\sn}>s)=\frac{2}{ s}.
\end{equation}
Combining \eqref{survivbnds} with \eqref{mdef} and taking limits from the left, we arrive at
\begin{equation}\label{sprobbnds}
 \frac{\underline s_D}{c_{\ref{mdef}}(s\vee 1)}\le \mu_n(S^{\sn}>s)\le \mu_n(S^{\sn}\ge s)\le \frac{c_{\ref{mdef}}\overline s_D}{s\wedge n},\quad\forall s\ge 0. 
\end{equation}
 
Suppressing dependence on $\gamma,\sigma_0^2$, for $s>0$ we define probabilities by 
\begin{equation}\label{Pnsdef}\P_n^{s}(\cdot)=\P\big( \cdot \big | S^{\sn}>s)
\end{equation}
 and 
 \begin{equation}\label{Nosdef}
 \N_o^s(\cdot)=\N_o(\cdot \,|\, S>s).
 \end{equation}
 We also let $v>0$ be the spatial scaling parameter arising in Theorem~1.2 of \cite{HHP17} (see especially (1.6) of that reference).   
The definition of $v$ is non-trivial (it is defined in terms of so-called lace expansion coefficients), but it satisfies
 \[v=[1+O(L^{-d/2})]\frac{1}{d}\sum_x |x|^2D(x)=O(L^2).\]
 
 \begin{PRP}\label{prop:wkcvgcevoterLT} Consider either critical lattice trees with $d>8$ and $L$ sufficiently large, or
 the voter model with $d>1$. Then for any $s>0$, 
 \begin{equation}
 \nn
\P_n^{s}(\bs{X}^{\sn}\in\cdot) \cweak \N_o^s, \quad  (\text{weak convergence on }\mc{D}),
\end{equation}
where $\N_o^s$ has parameters $(\gamma,\sigma_0^2)=(1,v)$ for lattice trees, and $(\gamma,\sigma_0^2)=(2\beta_d,\sigma^2)$ for the voter model.
\end{PRP}
\begin{proof} For lattice trees this is an immediate consequence of Theorem~1.2 of \cite{HHP17} (and an elementary rescaling) and \eqref{cond1'}. The reader should note, however, that the definition of $\mu_n$ in \cite{HHP17} and that given above differ by a constant factor of $A$.

For the voter model this is Theorem~4(b) of \cite{BCLG01}.
\end{proof}
  It is conjectured that super-Brownian motion is also the scaling limit of critical oriented percolation and the contact process above $4$ dimensions, and critical percolation above 6 dimensions.  For oriented percolation and the contact process, convergence of {\it the finite-dimensional distributions} (f.d.d.'s)  has been established \cite{HofSla03b,HofSak10} (see also \cite{HolPer07,HH13}) but tightness (and hence convergence on path space) remains open.
We stress that the actual weak convergence result we will impose on our lattice models (Condition~\ref{cond:finite_int} below) will in fact follow from this weaker convergence of f.d.d.'s and a moment bound on the total mass.

We slightly abuse the above notation and will denote super-Brownian motion under $\N_o$, or the probabilities $\N_o^s$, by $\bs{X}=(X_t)_{t\ge 0}$. 

Our main objective is to give general conditions for the convergence of the rescaled sets of occupied sites.  This convergence follows neither from the notions  of weak convergence  above, nor from the weak convergence of the so-called historical processes (see e.g.~\cite{DP91,H16,CFHP18}).

\subsubsection*{{\em Range}}
The range of ${\bf \mc{T}}$ is $R^{\sss(1)}=\cup_{t\in I}\mT_t$, which by \eqref{finiterange} and \eqref{death} is a finite
subset of $\Z^d$ on $\{S^{\sss(1)}<\infty\}$, and hence under Condition~\ref{cond:surv} will be finite a.s. The range of $\mT^{\sn}$ is $R^{\sn}=R^{\sss(1)}/\sqrt n=\cup_{t\ge 0}\mT^{\sn}_t$. So by the above we see that
\begin{equation}\label{rangefinite}\text{Condition~\ref{cond:surv} implies $R^{\sn}$ is a.s. a finite subset of $\R^d$.}
\end{equation}  
Let $\mc{R}:\mc{D} \ra \text{closed subsets of $\R^d$ }$ be defined by 
\begin{equation}
\nn
\mc{R}(\bs{\nu})=\text{supp}\left(\int_0^\infty \nu_t \,dt\right),
\end{equation}
where $\supp(\mu)$ is the closed support of a measure $\mu$.
Clearly 
$R^{\sn}=\mc{R}(\bs{X}^{\sn})$ for all $n\ge 1$.  The radius mapping $r_0: \mc{K}\rightarrow [0,\infty)$ on the space of compact subsets of $\R^d$ is given by
\begin{equation}
\nn
r_0(K)=\sup\{|x|:x\in K\}.
\end{equation}
Of particular interest is the {\em extrinsic one-arm probability}
\begin{equation}
\nn
\eta_r=\P\big(R^{\sss(1)}\cap \overline{B(o,r)}^c\ne \varnothing\big)=\P(r_0(R^{\sss(1)})>r).
\end{equation}

In the setting of high-dimensional critical percolation, Kozma and Nachmias \cite{KozNac11} have proved that as $r\to\infty$, $r^2\eta_r$ is bounded above and below by positive constants.  It is believed (see e.g. \cite[Open Problem 11.2]{HeyHofBook} and \cite[Conjecture 1.6]{HHP17}) that in fact $r^2\eta_r\ra C>0$ for various critical models (percolation, voter, lattice trees, oriented percolation, and the contact process) all above their respective critical dimensions.  
To understand this $r^{-2}$ behaviour in terms of the above weak convergence results, consider the one-arm probabilities for the limiting super-Brownian motion.

The range of a $(\gamma,\sigma_0^2)$-SBM ${\bs {X}}$ is denoted by 
\begin{equation}
\nn
R=\mc{R}(\bs{X}).
\end{equation}
The a.e.~continuity of $(X_t)_{t\ge0}$ easily shows that
\begin{equation}
\nn
R=\overline {\cup_{t\ge 0}\text{supp}(X_t)}\quad \N_o-\text{a.e.}
\end{equation}
We note that $R$ is a compact subset $\N_o$-a.e.  This is well-known under $\P_{\delta_0}$ (see, e.g. Corollary III.1.4 of \cite{Per02}) and then follows easily under $\N_o$ using \eqref{sbmppp}.  We now state a quantitative version of this from \cite{Isc88}.  For $d\ge 1$, let $v_d:B_d(0,1)\ra \R_+$ be the unique positive radial solution of 
\begin{equation}\label{pde1}\Delta v_d=v_d^2, \qquad \text{ with }\lim_{|x|\uparrow 1}v_d(x)=+\infty.
\end{equation}
(See Theorem 1 of \cite{Isc88} for existence and uniquenss of $v_d$.)\\

\begin{LEM}\label{lem:radrange} For all $d\ge 1$ and $r>0$, $\N_o^{\gamma,\sigma_0^2}(r_0(R)>r)=\frac{v_d(0)\sigma_0^2}{\gamma}r^{-2}$.
\end{LEM}
\begin{proof} Theorem~1 of \cite{Isc88} and a simple scaling argument  
show that
\[\P_{\delta_0}(r_0(R)>r)=1-\exp\Bigl(-\frac{v_d(0)\sigma_0^2}{\gamma}r^{-2}\Bigr).\]
On the other hand, the left-hand side of the above is $1-\exp(-\N_o(r_0(R)>r))$ by \eqref{sbmppp}. 
Combining these equalities completes the proof.
\end{proof}

\section{Statement of Main results}
\label{sec:main}
Our main results depend on certain conditions which we state below.  They will be verified for
the voter model ($d\ge 2$) and sufficiently spread out critical lattice trees ($d>8$), and most, but not all, are known
to hold for critical oriented percolation, and the critical contact process ($d>4$)--both sufficiently spread out. 

Recall our standing assumptions \eqref{cond1} and (AR), 
the function $m$ from \eqref{mdef}, and
 the unconditioned measures ($n\ge 1$)
$\mu_n(\cdot)=m(n)\P(\cdot)$.   Recall also that $(\mc{F}_t)_{t\in I}$ is the filtration introduced prior to (AR) which will
contain the filtration generated by $(\mT_s)_{s\le t}$ or equivalently by $(X^{\sss(1)}_s)_{s \le t}$.
\subsection{Conditions}
\label{sec:cond}
We now introduce additional conditions on $(\bs{\mT},\ara)$. 
Condition~\ref{cond:L1bound} is simple for the voter model, and is one of the outputs of the inductive approach to the lace expansion \cite{HofSla02,HHS08} for other models, while Condition~\ref{cond:self-repel} will usually follow
from Condition~\ref{cond:surv} and a form of the Markov property or Markov inequality.
\begin{COND}\label{cond:L1bound} 
$\quad \sup_{t\in I}\E\big[|\mT_t|\big]=c_{\ref{cond:L1bound}}<\infty$.
\end{COND}

\begin{COND}
\label{cond:self-repel}
There exists $c_{\ref{cond:self-repel}}>0$ such that for all $s\ge 0,t>0$, on the event $\{y \in \mT_s\}$ we have 
$$\P\big(\exists z : (s,y) \ara (s+t,z) \big |\mc{F}_s\big )\le \frac{c_{\ref{cond:self-repel}}}{m(t)}, \qquad \text{ a.s.}$$
\end{COND}
The next condition is the main input for our uniform modulus of continuity for ancestral paths (e.g. Theorem~\ref{thm:mod_con} below).  In general it also is the most difficult to verify as our arguments
below for lattice trees demonstrate.
\begin{COND}
\label{cond:6moment}
There exists a $p>4$ and $c_{\ref{cond:6moment}}=c_{\ref{cond:6moment}}(p)>0$ such that for every $0<s\le t$, 
\begin{equation}
\ARXIV{\label{cond6momentv1}}
\SUBMIT{\nn}
\E\Bigl[\sum_{x\in\mT_t}\sum_{y\in \mT_{t-s}}\1((t-s,y)\ara(t,x))|x-y|^p\Bigr]\le c_{\ref{cond:6moment}} (s\vee 1)^{p/2}.
\end{equation}
\end{COND}

We will need an additional hypothesis to control the ancestral paths just before
the terminal value.

\begin{COND}\label{cond:smallinc} There are $\kappa>4$ and $c_{\ref{cond:smallinc}}>0$ such that for all $s\ge 0$, $y\in\Z^d$, and $N>0$,
\begin{align}\label{smallincbnd}
\P(\exists\ &(t,x)\text{ s.t. } (s,y)\ara(t,x),\ t\in[s,s+2],\ |y-x|\ge N|\mc{F}_s)\\
\nonumber&\le c_{\ref{cond:smallinc}}N^{-\kappa}\text{ on }\{y\in\mT_s\}.
\end{align}
\end{COND}

\begin{REM}\label{smallincdisc}  In discrete time if for some $L>0$,
\begin{equation}\label{jumpbound} 
\forall k\in\Z_+,\ \forall x,y\in\Z^d\ \ [(k,x)\ara(k+1,y)\Rightarrow \Vert x-y\Vert\le L],
\end{equation}
then Condition~\ref{cond:smallinc} holds for any $\kappa>4$.  This is obvious since the conditional probability 
on the left-hand side of \eqref{smallincbnd} is then zero if $N>2\sqrt d L$. 
\end{REM}

If $\bs{\nu}=(\nu_t)_{t\ge0}\in\mc{D}$  and $r>0$, define $\bar{\nu}_{r}\in M_F(\R^d)$ by $\bar{\nu}_r(\cdot)=\int_0^r \nu_t(\cdot)  \dd t$, and $\BB{\nu}{}\in M_F(\R^d)$ by $\BB{\nu}{}(\cdot)=\1(\mc{S}(\nu)<\infty)\int_0^\infty \nu_t(\cdot)  \dd t$.
\begin{COND}
\label{cond:finite_int} There are parameter values $(\gamma,\sigma_0^2)$ for $\N_o$ so that
for every $s>0$ and $0\le t_0<t_1<\infty$, as $n \to \infty$,
\begin{equation}
\P_n^{s}\left(\bar{X}^{\sn}_{t_1}-\bar{X}^{\sn}_{t_0}\in \cdot\right)\cweak \N_o^{s
}\left(\bar{X}_{t_1}-\bar{X}_{t_0}\in \cdot\right)\text{ on }\mc{M}_F(\R^d). 
\nn
\end{equation}
\end{COND}
For critical lattice trees ($d>8$) with $L$ sufficiently large, and voter models ($d\ge 2$) the above is immediate from Proposition~\ref{prop:wkcvgcevoterLT} with the parameter values given there. For potential applications to other models it is worth noting that convergence
of finite dimensional distributions and boundedness of arbitrary moments of the total mass suffice. 
\begin{LEM}\label{lem:int_fdd}
Fix $s>0$ and let $\P_n^s$ and $\N_o^s$ be as in \eqref{Pnsdef} and \eqref{Nosdef}, respectively.  
Assume that for all $p \in \N$:
\begin{align}\label{fddconv}
&\text{for all }0<u_1,\dots,u_p<\infty,  \text{ as }n\to\infty,\\
\nonumber&\P_n^s((X^{\sn}_{u_1},\dots,X^{\sn}_{u_p})\in \cdot)\cweak\N_0^s((X_{u_1},\dots,X_{u_p})\in\cdot)\text{ on }\mc{M}_F(\R^d)^p,
\label{momentbounds}\\
&\text{and for } \bbt>0,\quad \sup_{n\ge 1}\sup_{t\le \bbt}\E^s_n\left[X_t^{\sn}(1)^p\right]:=C_{\ref{lem:int_fdd}}(s,\bbt,p)<\infty.
\end{align}
Then Condition \ref{cond:finite_int} holds.
\end{LEM}

The proof is an easy Fubini argument and is given in Section~\ref{sec:onconds}. 
Thus, to verify Condition~\ref{cond:finite_int} for lattice trees one does not need to invoke the additional hard work that was required in \cite{HHP17} for convergence on path space.  
The final condition is needed to ensure the rescaled ranges of $\bs{\mT}$ converge weakly to the range of super-Brownian motion. Together with uniform control of the ancestral paths, it will
ensure that any occupied regions will be close to regions of positive integrated mass of the limiting
super-Brownian motion. 
\begin{COND}
\label{cond:self-avoiding}
 There exists $c_{\ref{cond:self-avoiding}}>0$ such that for any $t\ge 0$, $M>0$, and $\Delta\ge 4$,
 \begin{align}\label{condsaineq}
\nonumber& \E\Bigl[ \sum_{x\in \mT_t}\1\Bigl(\exists x' \text{ s.t. }(t,x)\ara (t+\Delta,x'), \int_{t+\Delta}^{t+2\Delta}|\{y:(t,x)\ara(s,y)\}|\,ds\le M\Bigr)\Bigr]\\
&\le c_{\ref{cond:self-avoiding}}\P\Bigl(S^{\sss(1)}>\Delta,\ \int_{\Delta+2}^{2\Delta-2}|\mT_s|ds\le M\Bigr).
\end{align}
\end{COND}

Here is a condition which implies the above and is more user-friendly in discrete time; the odd-looking 
slight alteration in the integration limits on the right side ($\Delta$ will be large in applications) in the above handles the round-off errors arising in the elementary proof, which is given in Section~\ref{sec:onconds}.  
\begin{LEM}\label{disccondsa}
Assume the discrete time setting and there is a $c_{\ref{cond:self-avoiding}}$ such that for all $\ell\in\Z_+$, $m\in\N^{\ge 4}$ and $M>0$,
\begin{align}\label{dtsa}
&\E\Big[\sum_{x\in \mT_\ell}\1(\exists x'\text{ s.t. }(\ell,x)\ara(\ell+m,x')\text{ and }\\
\nonumber&\phantom{\E(\sum_{x\in \mT_\ell}1(}|\{(i,y):(\ell,x)\ara(i,y), \ell+m+2\le i\le \ell+2m-1\}|\le M)\Big]\\
\nonumber&\ \ \le c_{\ref{cond:self-avoiding}}\P\Bigl(S^{\sss(1)}>m,\ \sum_{i=m+2}^{2m-1}|\mT_i|\le M\Bigr).
\end{align}
Then Condition~\ref{cond:self-avoiding} holds.
\end{LEM}

\subsection{Main results}
\label{sec:sub_main}
We start with a uniform modulus of continuity for the system of ancestral paths in either discrete or continuous time. As was noted above, this modulus plays an important role in the convergence of the rescaled ranges but is also of independent interest. For critical branching Brownian motion such a modulus was first given in Theorem~4.7 of \cite{DIP89}.  Although we assume Condition~\ref{cond:surv} for convenience, in fact the proof only requires the existence of a non-decreasing function $m$ satisfying \eqref{mdef} and Condition~\ref{cond:self-repel}, as well as Conditions~\ref{cond:L1bound}, \ref{cond:6moment}, and \ref{cond:smallinc} but not the exact asymptotics in \eqref{survival} or \eqref{mregvar}. We will often assume
\begin{equation}\label{alphbet}
\alpha\in(0,1/2), \ \beta\in(0,1]\text{ satisfy }\frac{1-2\alpha}{1+\beta}\ge\frac{4}{p},
\end{equation}
where $p$ is as in Condition~\ref{cond:6moment}.  We will also sometimes assume
\begin{equation}\label{alphbet2}
0<\alpha<\frac{1}{2}-\frac{2}{\kappa},
\end{equation}
where $\kappa$ is as in Condition~\ref{cond:smallinc}.
Note that such $\alpha,\beta$ exist since $\kappa,p>4$.

\begin{customthm}{1}
\label{thm:mod_con}
Assume Conditions \ref{cond:surv} to \ref{cond:smallinc} for $I=\Z_+$ or $\R_+$, and $\alpha,\beta$ satisfy \eqref{alphbet} and \eqref{alphbet2}. Set $q=\frac{\kappa(1/2-\alpha)-2}{2}\wedge 1\in(0,1]$. 
There is a constant $C_{\ref{thm:mod_con}}\ge 1$  and 
for all $n\ge 1$ a random variable $\delta_n\in[0,1]$ such that
\begin{equation}\label{deltabndgen}
\mu_n(\delta_n\le \rho)\le C_{\ref{thm:mod_con}}[\rho^\beta+n^{-q}],\quad  \forall \rho\in[0,1),
\end{equation}
and if 
\begin{equation}\label{modcontcondgen}
 |s_2-s_1|\le \delta_n \text{ and }(s_1,y_1)\aran(s_2,y_2)
\end{equation}
then
\begin{equation}\label{modcontgen}
|y_1-y_2|\le C_{\ref{thm:mod_con}}[|s_2-s_1|^\alpha+n^{-\alpha}].
\end{equation}
\end{customthm}
 In the discrete time setting recall that $\Z_+/n=\{j/n:j\in\Z_+\}$ is the natural time line. In this setting we can get a cleaner statement if we also assume the stronger \eqref{jumpbound} in place of Condition~\ref{cond:smallinc}.
\begin{customthm}{1'}
\label{thm:mod_con_disc}
Assume Conditions \ref{cond:surv} to \ref{cond:6moment} and \eqref{jumpbound}, where $I=\Z_+$.  Assume also that $\alpha,\beta$ are as in \eqref{alphbet}. There is a constant $C_{\ref{thm:mod_con_disc}}$ and for all $n\ge 1$ a random variable $\delta_n\in(0,1]$  such that 
\begin{equation}\label{deltabnd}
\mu_n(\delta_n\le \rho)\le C_{\ref{thm:mod_con_disc}}\rho^{\beta},\quad \forall\rho\in[0,1),
\end{equation}
and if
\begin{equation}\label{modcontcond}
 s_1,s_2\in\Z_+/n, \quad |s_2-s_1|\le \delta_n, \quad \text{ and }\quad (s_1,y_1)\aran(s_2,y_2)
\end{equation}
then
\begin{equation}
\label{modcontdisc}|y_1-y_2|\le C_{\ref{thm:mod_con_disc}}|s_2-s_1|^{\alpha}.
\end{equation}
Moreover if $s_1,s_2, y_1,y_2$ are as in \eqref{modcontcond} but with $s_i\in\R_+$, then
\eqref{modcontgen} holds.
\end{customthm}

Theorems \ref{thm:mod_con} and \ref{thm:mod_con_disc} can be reinterpreted as (uniform) moduli of continuity for all ancestral paths.
\begin{DEF}
\label{def:modcon}
For $\alpha,\beta>0$, the system of ancestral paths $\mc{W}$
for $(\bT,\ara)$ is said to satisfy an $(\alpha,\beta)$-modulus of continuity if there exists a random function $\delta:[1,\infty)\to[0,1]$, a function $\vep:[1,\infty)\to [0,\infty)$ such that $\lim_{n\to\infty}\vep(n)=0$ and a constant $c>0$ such that for any $n\in [1,\infty)$:\\
For  every ancestral path $w\in \mc{W}$, and all $0\le s_1<s_2$,
\begin{align}
\nn
(1)&\ |s_2-s_1|\le \delta_n\Rightarrow |w^{\sn}_{s_2}- w^{\sn}_{s_1}|\le
c(|s_2-s_1|^\alpha + n^{-\alpha}).\\
\nn
(2)&\ m(n)\P(\delta_n\le u)\le c u^{\beta} + \vep(n)\text{ for each }u \in [0,1).
\end{align}
\end{DEF}
\begin{customcorollary}{1}\label{cor:wmod} Assume Conditions \ref{cond:surv} to \ref{cond:smallinc} for $I=\Z_+$ or $\R_+$, then for $\alpha,\beta,q$ as in Theorem~\ref{thm:mod_con}, $\mc{W}$ satisfies an $(\alpha,\beta)$-modulus of continuity with $\vep(n)=C_{\ref{thm:mod_con}}n^{-q}$.
\end{customcorollary}
\begin{proof} If $w$ is an ancestral path to $(nt,\sqrt n x)$ and $0\le s_1<s_2$, then for $s_i\le t$, $(s_1,w^{\sn}_{s_1})\aran(s_2,w^{\sn}_{s_2})$, and (1) and (2) of Definition \ref{def:modcon} follow immediately from Theorem~\ref{thm:mod_con}  with $\delta_n$ as in the theorem and $\vep(n)$ as claimed.
In general the result follows because $w^{\sn}_s=w^{\sn}_{s\wedge t}$. 
\end{proof}
\begin{customcorollary}{1'}\label{cor:dwmod} Assume Conditions \ref{cond:surv} to \ref{cond:6moment} and \eqref{jumpbound}, where $(\mT_t)_{t\in\Z_+}$ is in discrete time, and let $\alpha,\beta$ and $\delta_n$ be as in Theorem~\ref{thm:mod_con_disc}.  Then for any $n\ge 1$ and $w\in\mc{W}$,
\begin{equation}
\nn
s_i\in\Z_+/n, |s_2-s_1|\le \delta_n\quad \text{ implies }\quad |w^{\sn}_{s_2}-w^{\sn}_{s_1}|\le  C_{\ref{thm:mod_con}}[|s_2-s_1|^\alpha].
\end{equation}
If $s_i$ are as above but now in $\R_+$, then 
$|w^{\sn}_{s_2}-w^{\sn}_{s_1}|\le  C_{\ref{thm:mod_con}}[|s_2-s_1|^\alpha+n^{-\alpha}]$.
\end{customcorollary}
\begin{proof} As above but now use Theorem~\ref{thm:mod_con_disc} in place of Theorem~\ref{thm:mod_con}.\end{proof}

Our second main result is that, conditional on longterm survival, the rescaled range converges weakly to the range of (conditioned) SBM.  
\begin{customthm}{2}[Convergence of the range]
\label{thm:range}
Assume Conditions \ref{cond:surv}-\ref{cond:self-avoiding}, and let $\N_o$ be the canonical measure with parameters $(\gamma,\sigma_0^2)$ from Condition~\ref{cond:finite_int}. 
Then for every $s>0$, 
\begin{equation}
\nn
\P(R^{\sn}\in \cdot|S^{\sn}>s) \cweak \N_o(R\in\cdot|S>s)\text{ as }n\to\infty,\ \  n\in[1,\infty),
\end{equation}
as probability measures on $\mc{K}$ equipped with the Hausdorff metric.
\end{customthm}
With Lemma~\ref{lem:int_fdd} in mind it is perhaps a bit surprising that such a result could
be proved without a formal tightness condition.  It is Theorem~\ref{thm:mod_con} that effectively
gives tightness of the approximating ranges.

Recall that  $v_d:B_d(0,1)\ra \R_+$ is  the unique solution of \eqref{pde1}.  The next result uses the previous two results to give exact leading asymptotics for the extrinsic one-arm probability.

\begin{customthm}{3}[One-arm probability]
\label{thm:one-arm}
Assume Conditions \ref{cond:surv}-\ref{cond:self-avoiding}.  Then
 \begin{equation}\label{oa1}
\lim_{r\to\infty} \frac{\P(r_0(R^{\sss(1)})>r)}{\P(S^{\sss(1)}>r^2)}=\frac{\sigma_0^2}{2} v_d(0),
 \end{equation}
 and so (Condition~\ref{cond:surv})
 \begin{equation}\label{oa2}
 \lim_{r\to\infty}m(r^2)\P(r_0(R^{\sss(1)})>r)=\frac{\sigma_0^2}{2} s_Dv_d(0).
 \end{equation}
 \end{customthm}
 
 \begin{REM} The proof in Section~\ref{sec:one-arm} only uses Condition~\ref{cond:surv} and the conclusions of Theorems~\ref{thm:mod_con} and \ref{thm:range}.
 \end{REM}

We finally show that the above conditions are satisfied by the voter model ($d\ge 2$) and lattice trees ($d>8$).
\begin{customthm}{4}[Voter model]
\label{thm:voter}
For the voter model,  Conditions \ref{cond:surv}-\ref{cond:self-avoiding} hold  in dimensions $d\ge 2$ for  any $p>4$ in Condition~\ref{cond:6moment}, any $\kappa>4$ in Condition~\ref{cond:smallinc}, and $(\gamma,\sigma_0^2)=(2\beta_d,\sigma^2)$ in Condition~\ref{cond:finite_int}.   Hence for $d\ge 2$, if $\N_o$ is the canonical measure of SBM with parameters $(2\beta_d,\sigma^2)$, then
\begin{itemize}
\item[(v1)] For any $\alpha\in(0,1/2)$, the system of ancestral paths, $\mc{W}$, satisfies an $(\alpha,1)$-modulus of continuity with  $\vep(n)=C_{\ref{thm:mod_con}}n^{-1}$,
\item[(v2)] $\P_n^{s}(R^{\sss(n)}\in\cdot)\cweak \N_o^{s}(R\in\cdot)$ in $\mc{K}$, as $n\to\infty$, for every $s>0$,
\item[(v3)](i) $\lim_{r\to\infty}r^2\P(r_0(R^{\sss(1)})>r)=\frac{\sigma^2v_d(0)}{2\beta_d}$ if $d>2$,\\
(ii) $\lim_{r\to\infty}\frac{r^2}{\log r}\P(r_0(R^{\sss(1)})>r)=v_2(0)(2\pi)^{-1}$ if $d=2$.
\end{itemize}
\end{customthm}
Part (v1) will give a uniform modulus of continuity for all of the rescaled dual coalescing random walks 
between $1$'s in a voter model conditioned on longterm survival. This result, which may be of independent interest, is stated and proved in Section~\ref{sec:voter} (Corollary~\ref{cor:coaldualmod}). 

\begin{customthm}{5}[Lattice trees]
\label{thm:LT}
For critical lattice trees with $L$ sufficiently large,
Conditions  \ref{cond:surv}-\ref{cond:self-avoiding} hold in dimensions $d>8$, for $p=6$ in Condition~\ref{cond:6moment}, any $\kappa>4$ in Condition~\ref{cond:smallinc} because \eqref{jumpbound} holds, and $(\gamma,\sigma_0^2)=(1,v)$ in Condition~\ref{cond:finite_int}.  Hence for $d>8$, if $\N_o$ is the canonical measure of SBM with parameters $(1,v)$, then
\begin{itemize}
\item[(t1)] For $\alpha\in(0,1/6)$ and $\beta\in(0,1]$ satisfying $1-2\alpha>\frac{2}{3}(1+\beta)$, 
the system of ancestral paths, $\mc{W}$,  satisfies an $(\alpha,\beta)$- modulus of continuity with $\vep(n)\equiv 0$,
 \item[(t2)] $\P_n^{s}(R^{\sn}\in\cdot)\cweak \N_o^{s}(R\in\cdot)$ in $\mc{K}$, as $n\to\infty$, for every $s>0$,
\item[(t3)] $\lim_{r\to\infty}r^2\P(r_0(R^{\sss(1)})>r)=\frac{v v_d(0)}{AV}$.
\end{itemize}
\end{customthm}
Recall that for lattice trees $w_k(m,x)$ is the location of the ancestor of $x\in\mT_m$ in generation $k\le m$ of the tree.  We can also apply Corollary~\ref{cor:dwmod} to obtain a modulus of continuity for the large scale behaviour of $k\mapsto w_k(m,x)$ conditional on longterm survival; see Corollary~\ref{cor:LTwmod} in Section~\ref{sec:LT}.

Theorem \ref{thm:LT} also has important consequences for the behaviour of random walk on lattice trees.  It implies Condition $S$ of \cite{CF18}, which (roughly speaking) states that if we choose an integer $K$ sufficiently large then: (uniformly in $n$) conditional on survival
until time $n$,  any point in $\mT$ will be close to $\cup_{i=1}^K w(U_i)$ (where $(U_i)_{i=1}^K$ are uniformly chosen points in $\vec{\mT}$) with overwhelming probability  (see \cite{CFHP18}).

Note that the lower bounds in (v3) and (t3) above follow easily from the weak convergence results in \cite{BCLG01}  and \cite{HHP17}, respectively, and the lower semi-continuity of the support map on $\mc{M}_F(\R^d)$ (see Lemma~\ref{suppusc} below). So the importance of (v3) and (t3) are the matching upper bounds.   
For the voter model, in an interesting article Merle \cite{Mer08} has studied the probability that a distant site $x$ ever holds the opinion 1 (i.e.~$\P(\xi_t(x)=1 \text{ for some }t\ge 0)$ in the limit as $|x| \ra \infty$). But (v3) seems to be a different question.

\subsection{Discussion on Conditions and Extensions}
\label{sec:discussion}
Note that, although the above list of conditions may appear lengthy, we shall see that for voter models ($d\ge 2$) and critical lattice trees (sufficiently spread out for $d>8$) all but Condition~\ref{cond:6moment} for lattice trees are either already known or are fairly easy to establish from known results.  

Since Condition~\ref{cond:6moment} seems to be the crucial one, we turn to this now.  
For lattice trees our verification of Condition~\ref{cond:6moment} for $p=6$ is based on a suggestion of Remco van der Hofstad, and is presented in Section \ref{sec:tree6}.
The type of argument that we give here should also be applicable in the settings of oriented percolation and the contact process.   Condition \ref{cond:6moment} is verified for the voter model ($d\ge 2$) for any $p>4$ and for lattice trees ($d>8$) for $p=6$ in Sections~\ref{sec:voter} and \ref{sec:LT}, respectively.  For both models the statement can be reduced to a bound on the $p$th moment of the two-point function, i.e., to a bound of the form
\begin{equation}
\sum_x |x|^p \P(x \in \mc{T}_t)\le C(t\vee 1)^{(p/2)}.\label{pboundbasic}
\end{equation}
For the voter model such a reduction is implicit in \eqref{C4vm}, while for lattice trees the reduction 
is given by Lemma \ref{lem:reduction} with $f(x)=|x|^p$.  

It is not hard to verify all the conditions except Condition~\ref{cond:6moment} for critical {\em oriented
percolation} and the critical {\em contact process} (both sufficiently spread out) in more than $4$ spatial dimensions, with model-dependent constants $A,V,v$ playing the same role as, but taking different values to, those for lattice trees.  It is also easy to verify that in these contexts Condition~\ref{cond:6moment} follows from \eqref{pboundbasic} for $p>4$.  

 For oriented percolation these claims are verified  in  \ARXIV{Section \ref{sec:OP},}\SUBMIT{the version of this work on the arXiv \cite[Section 10]{Arxiv_version},}  and yield Theorem \ref{thm:OP} below.  Before stating this, let us give a careful definition of the model, which is a version of that studied in e.g. \cite{HofSla03b}.  The ancestral system is defined in terms of a random graph with vertices in $\Z_+\times \Z^d$ and directed bonds of the form $((n,x),(n+1,y))$.   For simplicity we take $D$ (see Section \ref{sec-setting}) to be uniform on $[-L,L]^d\setminus \{o\}$.  The bond $((n,x),(n+1,y))$ is occupied with probability $pD(y-x)$, where $p\in [0,\|D\|_{\infty}^{-1}]$, independent of all other bonds.  We say that there is an occupied path from $(n,x)$ to $(n',x')$, and write $(n,x) \to (n',x')$, if there is a sequence $x=x_0,x_1,\dots, x_{n'-n}=x'$ in $\Z^d$ such that $((n+i-1,x_{i-1}),(n+i,x_{i}))$ is occupied for each $i=1,\dots, n'-n\ge 0$.  We include the convention that $(n,x) \to (n,x')$ if and only if $x=x'$.  We write $(n,x) \ara (n',x')$ if $(0,o) \to (n,x)$ and $(n,x) \to (n',x')$.  Let $\mT_n=\{x\in \Z^d:(0,o) \to (n,x)\}$.  Let $\mc{F}_n$ denote the $\sigma$-field generated by the bond occupation status for all bonds $((m-1,x),(m,x'))$ for $1\le m\le n$ and $x,x'\in \Z^d$.  
Let $\P_p$ denote the law of the model.  Define $C(m,y)=\{(n,x): (m,y) \to (n,x)\}$ and  $p_c=\sup\{p:\E_p[|C(0,o)|]<\infty\}$, and $\P=\P_{p_c}$.  
\begin{customthm}{6}[Oriented Percolation]
\label{thm:OP}
For critical oriented percolation in spatial dimension $d>4$ with $L$ sufficiently large, suppose that \eqref{pboundbasic} holds \Ed{for all $t\in\N$} with $p=6$.  Then for $d>4$, if $\N_o$ is the canonical measure of SBM with parameters $(\gamma,\sigma_0^2)=(1,v)$ then
\begin{itemize}
\item[(op1)] For $\alpha\in(0,1/6)$ and $\beta\in(0,1]$ satisfying $1-2\alpha>\frac{2}{3}(1+\beta)$, 
the system of ancestral paths, $\mc{W}$,  satisfies an $(\alpha,\beta)$-modulus of continuity with $\varepsilon(n)\equiv 0$,
 \item[(op2)] $\P_n^{s}(R^{\sn}\in\cdot)\cweak \N_o^{s}(R\in\cdot)$ in $\mc{K}$, as $n\to\infty$, for every $s>0$,
\item[(op3)] $\lim_{r\to\infty}r^2\P(r_0(R^{\sss(1)})>r)=\frac{v v_d(0)}{AV}$.
\end{itemize}
\end{customthm}
A similar conditional theorem can be established for the critical spread-out contact process with $d>4$; we leave the details here for the interested reader. We conjecture that \eqref{pboundbasic} holds for oriented percolation (and the contact process) with $p=6$, and in particular, that Theorem \ref{thm:OP} can be made into an unconditional theorem.

Recent work \cite{HHHM17} suggests that these results may also be applicable  in the setting of critical sufficiently spread out percolation above $6$ dimensions.  In this setting, for example, Theorem~\ref{thm:one-arm} would  refine a result of Kozma and Nachmias \cite{KozNac11}, by giving a bona fide limit for the one-arm probability. We quickly point out, however, that the important Condition~\ref{cond:finite_int} has yet to be verified in this setting (see \cite{HS00} for partial results). For critical percolation with $d>6$, in a very interesting paper Tim Hulshof \cite{Hu15} has shown that there is a phase transition in the one-arm exponents which corresponds to the condition $p>4$ in Condition~\ref{cond:6moment}.  More precisely he works with kernels $D$ satisfying $D(z)\sim |z|^{-d-\alpha }$ as $|z|$ becomes large, where $f(z)\sim g(z)$ means 
$cg\le f\le Cg$ for some $0<c\le C<\infty$.  He then shows that if $R^{\sss(1)}$ is the open cluster of the origin, then
\[\P(r_0(R^{\sss(1)})>r)\sim r^{-\min(4,\alpha)/2}\text { as }r\to\infty.\]
For $\alpha\ge 4$ this also gives the one-arm exponent found in \cite{KozNac11} in the finite range setting,
but for $\alpha<4$ the one-arm exponent is no longer $2$. It is easy to check that $\alpha>4$ implies that 
$\sum_z|z|^p D(z)<\infty$ for some $p>4$.  This suggests that the restriction $p>4$ in Condition~\ref{cond:6moment} is sharp. We conjecture that Theorems~\ref{thm:mod_con}, \ref{thm:range} and \ref{thm:one-arm} all fail in general if we allow $p<4$ in this condition. In particular this should be the case for both the voter model and lattice trees with such ``long-range" kernels.  

The remainder of this paper is organised as follows.  In Section \ref{sec:mod_con} we establish the moduli of continuity, i.e., Theorems~\ref{thm:mod_con} and \ref{thm:mod_con_disc}.  In Section \ref{sec:range} we prove our general result on convergence of the rescaled ranges, Theorem \ref{thm:range}.  In Section \ref{sec:one-arm} both of the above ingredients 
are used to  prove the one-arm result, Theorem \ref{thm:one-arm}. Lemmas ~\ref{lem:int_fdd} and \ref{disccondsa} (dealing with checking Conditions~\ref{cond:finite_int} and \ref{cond:self-avoiding}, respectively) and Proposition~\ref{prop:wexist} (existence of ancestral paths) are proved in Section~\ref{sec:onconds}.   In Section ~\ref{sec:voter} we verify our Conditions for the 
voter model and prove Theorem~\ref{thm:voter}.  All of the conditions other than Condition~\ref{cond:6moment} for lattice trees are established in Section~\ref{sec:LT} where Theorem~\ref{thm:LT} is proved modulo checking Condition~\ref{cond:6moment} for lattice trees; this last check
is then done in Section~\ref{sec:tree6}. \SUBMIT{As noted above, in a longer version of this paper on the arxiv \cite{Arxiv_version} we also verify all the conditions except Condition~\ref{cond:6moment} for critical spread-out oriented percolation with $d>4$ and so establish Theorem~\ref{thm:OP}.}
\ARXIV{We verify all the conditions except Condition~\ref{cond:6moment} for critical spread-out oriented percolation with $d>4$ in Section \ref{sec:OP}, and hence establish Theorem~\ref{thm:OP}.}  
\section{Modulus of continuity}
\label{sec:mod_con}
In this section we prove Theorems~\ref{thm:mod_con} and \ref{thm:mod_con_disc}.
\medskip

\begin{PRP}\label{prop:prelmod} Assume Conditions \ref{cond:surv}, \ref{cond:self-repel} and \ref{cond:6moment}, and let $\alpha,\beta$ satisfy \eqref{alphbet}.  There is a constant $C_{\ref{prop:prelmod}}$, and for all $n\ge 1$ a random variable $\delta_n\in(0,1]$ such that 
\begin{equation}\label{deltabndp}
\mu_n(\delta_n\le \rho)\le C_{\ref{prop:prelmod}}\rho^\beta, \quad \text{ for every }\rho\in[0,1),
\end{equation}
and if 
\begin{equation}\label{modcontcondp}
t\in\Z_+/n,\ x\in\mT_t^{\sn},\ 0\le s_1<s_2\le t-n^{-1}, \ s_i\in\Z_+/n,\text{ and }|s_2-s_1|\le \delta_n,
\end{equation}
then
\begin{equation}\label{modhyp}
(s_1,y_1)\aran(s_2,y_2)\aran(t,x)
\end{equation}
implies
\begin{equation}\label{modcontp}
|y_2-y_1|\le C_{\ref{prop:prelmod}}|s_2-s_1|^\alpha.
\end{equation}
\end{PRP}
\begin{proof}
We first note that it suffices to consider $n=2^{n_0}$ for some $n_0\in\N$. Assuming the result for this case, for $n\ge 1$,  choose $n_0\in\N$ so that $2^{n_0-1}\le n<2^{n_0}$ and set $\delta_n(\omega)=\delta_{2^{n_0}}(\omega)\in(0,1]$. The monotonicity of $m(n)$ shows that 
\[\mu_n(\delta_n\le \rho)\le \mu_{2^{n_0}}(\delta_{2^{n_0}}\le \rho)\le C_{\ref{prop:prelmod}}\rho^{\beta}, \quad \text{ for every }\rho\in[0,1).\]
Assume now that the conditions in \eqref{modcontcondp} are satisfied by $t=k/n$, $x=z/\sqrt{n}$ and $s_i=k_i/n$ where $k,k_i\in\Z_+$ and $z\in\Z^d$, and that \eqref{modhyp} is satisfied by $y_i=\frac{x_i}{\sqrt n}$, where $x_i\in\Z^d$ for $i=1,2$.  Then $k_2\le k-1$, which implies $k_2 2^{-n_0}\le k 2^{-n_0}-2^{-n_0}$, and $0<(k_2-k_1)2^{-n_0}<(k_2-k_1)n^{-1}\le \delta_{2^{n_0}}$. By scaling this implies $(\frac{k_1}{2^{n_0}},\frac{x_1}{2^{n_0/2}})\arano(\frac{k_2}{2^{n_0}},\frac{x_2}{2^{n_0/2}})\arano (\frac{k}{2^{n_0}},\frac{z}{2^{n_0/2}})$. So the result for $2^{n_0}$ implies that
\begin{align*}
|y_2-y_1|
&=\frac{2^{n_0/2}}{\sqrt n}\Bigl|\frac{x_2}{2^{n_0/2}}-\frac{x_1}{2^{n_0/2}}\Bigr|\le \sqrt 2C_{\ref{prop:prelmod}}\Bigl(\frac{k_2-k_1}{2^{n_0}}\Bigr)^{\alpha}\le \sqrt 2C_{\ref{prop:prelmod}}|s_2-s_1|^{\alpha}.
\end{align*}
So the result follows for general $n\ge 1$  by increasing $ C_{\ref{prop:prelmod}}$ to $\sqrt 2  C_{\ref{prop:prelmod}}$.

We set $n=2^{n_0}$ for $n_0\in\N$.  For a natural number $m$ define
\[I(n_0,m)=\{k\in\Z_+:k\le n_0, 2^{n_0-k+1}\le m\},\]
and for $k\in I(n_0,m)$ define
\begin{align*}A_k(n_0,m)=\Bigl\{\omega:\max\{|y_2-y_1|:&(m-2^{n_0-k+1},y_1)\ara(m-2^{n_0-k},y_2)\\
&\qquad\ara(m,x)\ \exists \, x\in \mT_m\}\ge 2^{n_0/2}2^{-k\alpha}\Bigr\}.
\end{align*}
Note that the $\max$ exists by \eqref{cond1}.
For $\ell\in\Z_+$,  let
\[B_\ell(n_0,m)=\begin{cases}\displaystyle\Bigg[\bigcup_{\substack{k> \ell:\\k\in I(n_0,m)}}A_k(n_0,m)\Bigg]\cup A_{n_0}(n_0,m+1),&\quad\text{if }\ell\le n_0,\\
\varnothing,&\quad\text{if }\ell>n_0.
\end{cases}
\]
\begin{LEM}\label{dmoduli1}
Assume \eqref{alphbet}.  There is a constant $c_{\ref{dmoduli1}}$ so that for all $m\in\N$,\\
(a) for all $k\in I(n_0,m)$, $\mu_{2^{n_0}}(A_k(n_0,m))\le c_{\ref{dmoduli1}}2^{-k(p/2-p\alpha-1)}$,\\
(b) \begin{equation*}
\mu_{2^{n_0}}(B_\ell(n_0,m))\le\begin{cases} c_{\ref{dmoduli1}}2^{-\ell(p/2-p\alpha-1)},&\qquad\text{if }
0\le \ell\le n_0,\\
0, & \qquad \text{if }\ell> n_0.\end{cases}
\end{equation*}
\end{LEM}
\begin{proof}(a) 
Clearly we have
\begin{align}\label{muAbound}
\mu&_{2^{n_0}}(A_k(n_0,m))\\
\nonumber&\le m(2^{n_0})\E\Bigg[\sum_{y_2\in \mT_{m-2^{n_0-k}}}\sum_{y_1\in \mT_{m-2^{n_0-k+1}}}\! \! \! \!  \1\big((m-2^{n_0-k+1},y_1)\ara (m-2^{n_0-k},y_2),\\\
&\phantom{\le m(2^{n_0})\E\Bigl(\sum_{y_2\in \mT_{m-2^{n_0-k}}}\sum_{y_1\in \mT_{m-2^{n_0-k+1}}}\1(} |y_2-y_1|\ge 2^{n_0/2}2^{-k\alpha}\big)\nn\\
\nonumber&\phantom{\le m_{2^{n_0}}\E\Bigl(\sum_{y\in T_{m-2^{n_0-k}}}     }\times \1\big(\exists x\in \mT_m\text{ s.t. }(m-2^{n_0-k},y_2)\ara(m,x)\big)\Bigg].
\end{align}
By Condition~\ref{cond:self-repel},
\[\P\big(\exists x\in \mT_m\text{ s.t. }(m-2^{n_0-k},y_2)\ara(m,x)\,\big |\,\mc{F}_{m-2^{n_0-k}}\big)\le \frac{c_{\ref{cond:self-repel}}}{m(2^{n_0-k})}.\]
So, recalling that 
\[\1\big((m-2^{n_0-k+1},y_1)\ara(m-2^{n_0-k},y_2)\big)=e_{m-2^{n_0-k+1},m-2^{n_0-k}}(y_1,y_2),\]
and  $\mT_{m-2^{n_0-k}}$ are both $\mc{F}_{m-2^{n_0-k}}$-measurable, we can use the above in \eqref{muAbound} to conclude
\begin{align*} \mu&_{2^{n_0}}(A_k(n_0,m))\\
&\le \frac{c_{\ref{cond:self-repel}} m(2^{n_0})}{m(2^{n_0-k})}\E\Bigg[\sum_{y_2\in \mT_{m-2^{n_0-k}}}\sum_{y_1\in \mT_{m-2^{n_0-k+1}}}\!\!\!\!\1\big((m-2^{n_0-k+1},y_1)\ara (m-2^{n_0-k},y_2)\big)\\
&\phantom{\le c_{\ref{cond:self-repel}}\frac{m(2^{n_0})}{m(2^{n_0-k})}\E\Bigl(\sum_{y_2\in \mT_{m-2^{n_0-k}}}\sum_{y_1\in \mT_{m-2^{n_0-k+1}}}\1(}\times\frac{|y_2-y_1|^p}{2^{pn_0/2}2^{-pk\alpha}}\Bigg]\\
&\le  c_{\ref{cond:self-repel}}c_{\ref{mdef}}2^kc_{\ref{cond:6moment}}\frac{2^{(n_0-k)p/2}}{2^{pn_0/2}2^{-p\alpha k}}=:c2^{-k(p/2-p\alpha-1)},
\end{align*}
where Condition~\ref{cond:6moment} and \eqref{mdef} are used in the second inequality. 

\noindent(b) Note first that \eqref{alphbet} implies $p/2-p\alpha>2$.   Sum the bound in (a) over $k>\ell$, note that $n_0\in I(n_0,m+1)$ (to apply (a) to $A_{n_0}(n_0,m+1)$), and use $B_\ell(n_0,m)=\varnothing$ if $\ell> n_0$ to derive (b) (where we can adjust the constants after the fact). 
\end{proof}

\begin{LEM}\label{dmoduli2}
Assume \eqref{alphbet}.  There is a $c_{\ref{dmoduli2}}=c_{\ref{dmoduli2}}(\alpha,L)$ such that if $m\in\N$ and $\ell\in\{0,\dots,n_0\}$ satisfies $2^{n_0-\ell}\le m$, then 
\begin{align*}&B_\ell(n_0,m)^c\subset\Bigl\{\omega:\max\{|x'-y|: (m-2^{n_0-\ell},y)\ara(m,x')\ara(m+1,x)\\
&\phantom{B_\ell(n_0,m)^c\subset\Bigl\{\omega:\sup\{|x'-y|:}\text{ for some }x\in\mT_{m+1}\}\le c_{\ref{dmoduli2}}2^{(n_0/2)-\ell \alpha}\Bigr\}.
\end{align*}
\end{LEM}
\begin{proof} The conditions on $\ell$ and $m$ show that $\{\ell+1,\dots,n_0\}\subset I(n_0,m)$.
This implies that 
\[B_\ell(n_0,m)^c\subset \cap_{k={\ell+1}}^{n_0}A_k(n_0,m)^c\cap A_{n_0}(n_0,m+1)^c.\]
Choose $\omega\in B_\ell(n_0,m)^c$ and assume $(m-2^{n_0-\ell},y)\ara(m,x')\ara(m+1,x)$ for some $x\in\mT_{m+1}$. By \eqref{convtr} we may choose $x''$ so that $(m-2^{n_0-\ell},y)\ara(m-1,x'')\ara(m,x')$ (set $x''=y$ if $\ell=n_0$).  Then
by $\omega\not\in A_{n_0}(n_0,m+1)$, we have
\begin{equation}\label{lastinc}
|x'-x''|<2^{n_0/2-n_0\alpha}.
\end{equation}
Let $y_\ell=y$. By \eqref{convtr} for $k=\ell+1,\dots,n_0$ we may choose $y_k\in\mT_{m-2^{n_0-k}}$ such that $y_{n_0}=x''$, and 
\[(m-2^{n_0-k+1},y_{k-1})\ara(m-2^{n_0-k},y_k)\text{ for }k=\ell+1,\dots,n_0.\]
 By $\omega\notin A_k(n_0,m)$ and the triangle inequality, we have
 \begin{align*}
 |x''-y|=|y_{n_0}-y_\ell|&\le \sum_{k=\ell+1}^{n_0}|y_k-y_{k-1}|\le 2^{n_0/2}\sum_{k=\ell+1}^{n_0}2^{-k\alpha}\le C2^{n_0/2}2^{-\ell\alpha}.
 \end{align*}
This and \eqref{lastinc} imply $|x'-y|\le C2^{(n_0/2)-\ell \alpha}$ and so completes the proof.
\end{proof}
Returning now to the proof of Proposition~\ref{prop:prelmod}, we define
\begin{align}
\nonumber&C_r(n_0)=\cup_{\ell=r}^{n_0}\cup_{i=1}^{\lceil2^{\ell(1+\beta)}\rceil}B_\ell(n_0,i2^{n_0-\ell})\text{ for }r\in\Z_+\ \text{ ($C_r(n_0)=\varnothing$ if }r>n_0),\\
\nonumber&K^1_{n_0}=\min\{r\in\{0,\dots,n_0+1\}:\omega\notin C_r(n_0)\}\le n_0+1,\\
\nonumber&K^2_{n_0}=\min\{r\in\Z_+:2^{r\beta+n_0}\ge S^{\sss(1)}\},\\
&\nn
K_{n_0}=(K^1_{n_0}\vee K^2_{n_0})\wedge (n_0+1)\in\Z_+.
\end{align}
Note that $C_r(n_0)$ is non-increasing in $r$.  
Set $\delta_{2^{n_0}}(\omega)=2^{-K_{n_0}(\omega)}\in(0,1]$. Then for $r\in\{0,\dots,n_0\}$,
\begin{align*}
\mu_{2^{n_0}}(K_{n_0}>r)&\le \mu_{2^{n_0}}(C_r(n_0))
+\mu_{2^{n_0}}(S^{\sss(1)}>2^{r\beta+n_0})\\
&\le \Bigl[\sum_{\ell=r}^{n_0}\sum_{i=1}^{\lceil2^{\ell(1+\beta)}\rceil} c_{\ref{dmoduli1}}2^{-\ell(p/2-p\alpha-1)}\Bigr]+\frac{m(2^{n_0})c_{\ref{cond:self-repel}}}{m(2^{r\beta}2^{n_0})},
\end{align*}
where we have used Lemma~\ref{dmoduli1}(b) and Condition~\ref{cond:self-repel} 
with $(s,y))=(0,o)$ and $t=2^{r\beta}2^{n_0}$.  Since $2^{r\beta}\le 2^{n_0}$ (recall $\beta\le 1$) we may use \eqref{mdef} to see that
\[\frac{m(2^{n_0})}{m(2^{r\beta}2^{n_0})}\le  c_{\ref{mdef}}2^{-r\beta},\]
and so from the above,
\begin{equation}\label{Kbound}
\mu_{2^{n_0}}(K_{n_0}>r)\le c_{\ref{dmoduli1}}\sum_{\ell=r}^{n_0}2^{1-\ell(p/2-2-p\alpha-\beta)}+c_{\ref{cond:self-repel}}c_{\ref{mdef}}2^{-r\beta}\le C2^{-r\beta},
\end{equation}
where \eqref{alphbet} is used in the last inequality.  If $r\in \{n_0+1,n_0+2,\dots\}$ then \eqref{Kbound} is trivial since the left-hand side is zero. If $\rho\in(0,1]$ choose $r\in\Z_+$ so that $2^{-r-1}<\rho\le 2^{-r}$.  Then by \eqref{Kbound}, 
\[\mu_{2^{n_0}}(\delta_{2^{n_0}}<\rho)\le\mu_{2^{n_0}}(K_{n_0}>r)\le C2^{-r\beta}\le C2^\beta\rho^\beta.\]
Take limits from the right in $\rho<1$ in the above to derive \eqref{deltabndp}.

Turning now to \eqref{modcontp}, we can rescale and restate our objective as (note that $t=(m+1)/n$ for some $m\in\N$ or the conclusion is vacuously true):\\ If
\begin{align} 
\nn
&k_1,k_2\in \Z_+, m\in \N, x\in \mT_{m+1}, 0\le k_1<k_2\le m\text{ satisfies }\\
\nn
&|k_2-k_1|\le 2^{n_0-K_{n_0}},\text{ and }(k_1,y_1)\ara(k_2,y_2)\ara(m+1,x),
\end{align}
then
\begin{equation}
\nn
|y_2-y_1|\le C_{\ref{prop:prelmod}}|k_2-k_1|^{\alpha}2^{(n_0/2)-(n_0\alpha)}.
\end{equation}
It suffices to consider $k_2=m$ in the above because if we choose (by \eqref{convtr}) $x'$ s.t. 
$(k_2,y_2)\ara(k_2+1,x')\ara(m+1,x)$, then
we can work with $(k_2+1,x')$ in place of $(m+1,x)$. 
So let us assume
\begin{align}\label{contcond2}
&k\in\Z_+, \,m\in\N, \,x\in \mT_{m+1}, \,0\le k< m, \, |m-k|\le 2^{n_0-K_{n_0}},\\
&\qquad\text{and }(k,y)\ara(m,x')\ara(m+1,x).
\nn
\end{align}
We must show that 
\begin{equation}\label{modcont2}
|x'-y|\le C_{\ref{prop:prelmod}}|m-k|^{\alpha}2^{(n_0/2)-(n_0\alpha)}.
\end{equation}
If $K_{n_0}=n_0+1$, then \eqref{contcond2} leads to a contradiction and so we have 
\begin{equation}\label{Knice}
K_{n_0}\le n_0,\text{ and so }K_{n_0}=K_{n_0}^1\vee K_{n_0}^2\le n_0.
\end{equation}
By \eqref{contcond2}, $2^{-n_0}\le (m-k)2^{-n_0}\le 2^{-K_{n_0}}$, and so we may choose $r\in\Z_+$ so that 
\begin{equation}\label{rdefn}
2^{-r-1}<\frac{m-k}{2^{n_0}}\le 2^{-r},\quad K_{n_0}\le r\le n_0.
\end{equation}
Using the fact that $0<(m-k)2^{-n_0}\le 2^{-r}$, we can choose $i_r,j_r\in\Z_+$ with $j_r-i_r\in \{0,1\}$ and  $i_\ell,j_\ell\in\{0,1\}$ for $\ell \in (r,n_0]\cap \N$ so that 
\begin{equation}
\nn
m2^{-n_0}=j_r2^{-r}+\sum_{\ell=r+1}^{n_0}j_\ell 2^{-\ell},\quad k2^{-n_0}=i_r 2^{-r}+\sum_{\ell=r+1}^{n_0}i_\ell 2^{-\ell}.
\end{equation}
For $q\in (r,n_0]\cap \N$ define
\begin{align*}
&m_r=j_r2^{n_0-r},\quad m_q=m_r+\sum_{\ell=r+1}^qj_\ell2^{n_0-\ell},\\
&k_r=i_r2^{n_0-r}, \quad k_q=k_r+\sum_{\ell=r+1}^qi_\ell 2^{n_0-\ell},
\end{align*}
so that $k_{n_0}=k$ and $m_{n_0}=m$.  
By Lemma~\ref{anseq} 
 there are $(y_q)_{q=r}^{n_0}$ and $(z_q)_{q=r}^{n_0}$ s.t. $y_{n_0}=y$, $z_{n_0}=x'$ and
\begin{equation}\label{yzanc}
(k_{q-1},y_{q-1})\ara(k_q,y_q),\quad (m_{q-1},z_{q-1})\ara(m_q,z_q)\quad \text{for }q \in (r,n_0]\cap \N
\end{equation}
(if $k_{q-1}=k_q$ set $y_q=y_{q-1}$ and similarly if $m_{q-1}=m_q$). 
In addition if $i_r=j_r$ so that $k_r=m_r$ we may assume
\begin{equation}\label{yzic}
y_r=z_r\quad \text{ if }\quad i_r=j_r.
\end{equation}
On the other hand if $j_r=i_r+1$, then 
\[k_{n_0}\le k_r+\sum_{\ell=r+1}^{n_0}2^{n_0-\ell}<k_r+2^{n_0-r}=(i_r+1)2^{n_0-r}=m_r,\]
and so (by Lemma~\ref{anseq}) we may also choose $z_r$ in the above s.t. $(k_{n_0},y_{n_0})\ara (m_r,z_r)$. Hence we have
\begin{align}\label{yzic2}
\nn\text{if }j_r=i_r+1,\text{ then }(k_r,y_r)\ara\dots&\ara(k_{n_0},y_{n_0})\ara(m_r,z_r)\\
&\ara(m,x')\ara(m+1,x).
\end{align}
Recalling from \eqref{Knice}, and \eqref{rdefn} that $r\ge K_{n_0}\ge K^1_{n_0}$, we have from the definition of $K^1_{n_0}$ that 
$\omega\notin C_r(n_0)$ 
and so for all $\ell \in [r,n_0]\cap \Z_+$, 
\[\omega\in\cap_{i=1}^{\lceil2^{\ell(1+\beta)}\rceil}B_\ell(n_0,i2^{n_0-\ell})^c,\]
which by Lemma~\ref{dmoduli2} (and $2^{n_0-\ell}\le i2^{n_0-\ell}$) implies
\begin{align}\label{incboundI}
\nn\forall&\, 1\le i\le \ceil{2^{\ell(1+\beta)}}\text{ if }((i-1)2^{n_0-\ell},y)\ara(i2^{n_0-\ell},y')\ara(i2^{n_0-\ell}+1,z)\\
&\text{for some }z\in\mT_{i2^{n_0-\ell}+1},\text{ then }|y-y'|\le c_{\ref{dmoduli2}}2^{n_0/2}2^{-\ell \alpha}.
\end{align}
As $r\ge K_{n_0}\ge K^2_{n_0}$, we also have $S^{\sss(1)}\le2^{r\beta+n_0}\le 2^{\ell\beta+n_0}$ for all $\ell\ge r$ and so if $i>\lceil 2^{\ell(1+\beta)}\rceil$, then
\[i2^{n_0-\ell}>2^{\ell(1+\beta)}2^{n_0-\ell}=2^{\ell\beta+n_0}\ge S^{\sss(1)},\]
and so $\mT_{i2^{n_0-\ell}+1}=\varnothing$. This implies \eqref{incboundI} holds vacuously and we may conclude
\begin{align}\label{incboundII}\nn\forall\, i&\in\N\text{ if }((i-1)2^{n_0-\ell},y)\ara(i2^{n_0-\ell},y')\ara(i2^{n_0-\ell}+1,z)\\
&\text{for some }z\in\mT_{i2^{n_0-\ell}+1},\text{ then }|y-y'|\le c_{\ref{dmoduli2}}2^{n_0/2}2^{-\ell \alpha}.
\end{align}

By $y_{n_0}=y$, $z_{n_0}=x'$ and the triangle inequality, we have
\begin{align}\label{binaryexpansionI}|x'-y|=|z_{n_0}-y_{n_0}|\le
 \sum_{q=r+1}^{n_0}\Bigl[|z_q-z_{q-1}|+|y_{q}-y_{q-1}|\Bigr]+|z_r-y_r|.
\end{align}
Note that for $q\in (r,n_0]\cap \N$, $m_q=\bar j_q2^{n_0-q}$ for some $\bar j_q\in\Z_+$ and $m_{q-1}=(\bar j_q-j_q)2^{n_0-q}$, where $j_q=0$ or $1$.  Therefore \eqref{yzanc} implies
\[((\bar j_q-j_q)2^{n_0-q},z_{q-1})\ara(\bar j_q 2^{n_0-q},z_q)\ara(m+1,x),\]
and so by \eqref{convtr} there is a $z\in\mT_{\bar j_q 2^{n_0-q}+1}$ s.t. 
\[((\bar j_q-j_q)2^{n_0-q},z_{q-1})\ara(\bar j_q 2^{n_0-q},z_q)\ara(\bar{j}_q 2^{n_0-q}+1,z).\]
Therefore \eqref{incboundII} implies 
\begin{equation}
\label{incboundIII}
|z_q-z_{q-1}|\le c_{\ref{dmoduli2}}2^{n_0/2}2^{-q \alpha}\quad\text{for }q\in (r,n_0]\cap \N
\end{equation}
(note here that if $\bar j_q=0$, then $j_q=0$ and so $z_q=z_{q-1}$ and the above inequality is trivial).
Similar reasoning shows
\begin{equation}
\nn
|y_q-y_{q-1}|\le c_{\ref{dmoduli2}}2^{n_0/2}2^{-q\alpha}\quad\text{for }(r,n_0]\cap \N.
\end{equation}

To handle the last term in \eqref{binaryexpansionI}, first observe that if $j_r-i_r=1$, then $k_r=(j_r-1)2^{n_0-r}$ and so (by \eqref{yzic2}) $(k_r,y_r)\ara(m_r,z_r)\ara(m+1,x)$ implies
\[(( j_r-1)2^{n_0-r},y_r)\ara(j_r 2^{n_0-r},z_r)\ara((j_r 2^{n_0-r}+1,z)\]
for some $z\in\mT_{j_r2^{n_0-r}+1}$ (by \eqref{convtr}).  It follows from \eqref{incboundII} with $\ell=r$ that
\begin{equation}\label{incboundV}
\text{if }j_r-i_r=1,\text{ then }|z_r-y_r|\le c_{\ref{dmoduli2}}2^{n_0/2}2^{-r \alpha}.
\end{equation}
Use \eqref{incboundIII}-\eqref{incboundV} and \eqref{yzic} in \eqref{binaryexpansionI} and so conclude that
\begin{align*}
|x'-y|&\le \Bigl[2c_{\ref{dmoduli2}}2^{n_0/2}\sum_{q=r+1}^{n_0}2^{-q\alpha}\Bigr]+c_{\ref{dmoduli2}}2^{n_0/2}2^{-r\alpha}\\
&\le C2^{n_0/2}2^{-r\alpha}\\
&\le C2^\alpha\Bigl(\frac{m-k}{2^{n_0}}\Bigr)^\alpha2^{n_0/2}\quad\text{(by \eqref{rdefn})},
\end{align*}
which gives \eqref{modcont2}, as required.
\end{proof}

\noindent{\it Proof of Theorem~\ref{thm:mod_con_disc}.}  Let $n\in [1,\infty)$ and take $\delta_n$ as in Proposition~\ref{prop:prelmod}, so that $\eqref{deltabnd}$ and $\delta_n\in(0,1]$ are immediate from that proposition. Assume $s_1,s_2,y_1,y_2$ are as in \eqref{modcontcond}.   If $s_1=s_2$ then $y_1=y_2$ by (AR)(i), and the result is trivial.  Otherwise $s_1<s_2$, and by \eqref{convtr} we may choose $x'$ s.t. 
 \[(s_1,y_1)\aran(s_2-n^{-1},x')\aran(s_2,y_2).\]
 We also have $|s_2-n^{-1}-s_1|\le |s_2-s_1|\le \delta_n$. Proposition~\ref{prop:prelmod} and \eqref{jumpbound} imply
 \begin{align*}
|y_2-y_1|&\le |y_2-x'|+|x'-y_1|\\
&\le\frac{\sqrt d L}{\sqrt n}+C_{\ref{prop:prelmod}}|s_2-n^{-1}-s_1|^\alpha\\
&\le(\sqrt d L+C_{\ref{prop:prelmod}})|s_2-s_1|^\alpha,
\end{align*}
the last since $|s_2-s_1|\ge 1/n$ and $\alpha<1/2$.  This proves \eqref{modcontdisc} with $C_{\ref{thm:mod_con_disc}}=\sqrt d L+C_{\ref{prop:prelmod}}$.   

Now suppose instead that $s_1,s_2\in \R_+$ (otherwise as above).  
\SUBMIT{The validity of \eqref{modcontgen} 
now follows by an elementary argument using \eqref{jumpbound} and the fact that for $0\le s_1<s_2$, $(s_1,y_1)\aran(s_2,y_2)$ iff $(\floor{s_1n}/n,y_1)\aran(\floor{s_2n}/n,y_2)$. \qed} 
\ARXIV{Set $[s]_n=\floor{sn}/n$, and note that for $0\le s_1<s_2$, $(s_1,y_1)\aran(s_2,y_2)$ iff $([s_1]_n,y_1)\aran([s_2]_n,y_2)$.  We may assume that $[s_1]_n<[s_2]_n$, otherwise $y_2=y_1$ and the claim is trivial.  Therefore $0\le [s_2]_n-n^{-1}-[s_1]_n\le |s_2-s_1|\le\delta_n$ and there exists $z\in \mT^{\sn}_{[s_1]_n+n^{-1}}$ such that 
\[([s_1]_n,y_1)\aran ([s_1]_n+n^{-1},z)     \aran ([s_2]_n,y_2).\]
Using \eqref{jumpbound} and the above result for times in $\Z_+/n$ we have that 
\begin{align*}
|y_2-y_1|&\le |z-y_1|+|y_2-z|\\
&\le \frac{\sqrt d L}{\sqrt n}+C_{\ref{thm:mod_con_disc}}|[s_2]_n-n^{-1}-[s_1]_n|^\alpha\le \frac{\sqrt d L}{n^\alpha}+C_{\ref{thm:mod_con_disc}} |s_2-s_1|^\alpha.\qed
\end{align*}
}

\medskip

Consider next the proof of Theorem~\ref{thm:mod_con} and assume the hypotheses of that theorem. We first use Condition~\ref{cond:smallinc} to handle the small increments of $w^{\sn}(t,x)$ near $t$.

\begin{LEM}\label{endpointinc} There is  a $C_{\ref{endpointinc}}$ such that for any $\alpha\in(0,1/2)$, $t^*\ge 1$ and $n\ge 1$,
\begin{align}
\nn
\mu_n\Bigl(&\max\{|y-x|:(s,y)\aran (t,x), t\in[s,s+(2/n)], s\in\Z_+/n,s\le t^*\} >n^{-\alpha}\Bigr)\\
\nonumber&\le C_{\ref{endpointinc}}t^*n^{2-\kappa((1/2)-\alpha)}.
\end{align}
\end{LEM}
\begin{proof} If $s=j/n\in\Z_+/n$ is fixed, then by scaling,
\begin{align*}
\mu_n&\Bigl(\max\{|y-x|:(s,y)\aran (t,x), t\in[s,s+(2/n)]\}>n^{-\alpha}\Bigr)\\
&=\mu_n\Bigl(\max\{|y-x|:(j,y)\ara(t,x), t\in[j,j+2]\}>n^{(1/2)-\alpha}\Bigr)\\
&\le m(n)\sum_{y\in \Z^d}\E\Big[\1(y\in\mT_j)\P(\exists\ t\in[j,j+2]\text{ and }x\in\mT_t\text{ s.t. }\\
&\phantom{\le m_n\sum_{y\in \Z^d}\E[\1(y\in\mT_j)\P(}|x-y|>n^{(1/2)-\alpha}\text{ and }(j,y)\ara(t,x)|\mc{F}_j)\Big].
\end{align*}
Now use Conditions~\ref{cond:smallinc} and \ref{cond:L1bound}, and \eqref{mboundi} to see the above is at most
\[c_{\ref{mboundi}}nc_{\ref{cond:L1bound}}c_{\ref{cond:smallinc}}n^{-\kappa((1/2)-\alpha)}.\]
Finally sum the above over $s\in(\Z_+/n)\cap[0,t^*]$ to obtain the desired upper bound.
\end{proof}

\noindent{\it Proof of Theorem~\ref{thm:mod_con}}. For $n\ge 1$ and $q\in(0,1]$ as in the Theorem statement and define
\begin{align*}
\Omega_n=&\{S^{\sn}\le n^q\}\\
&\cap\Bigl\{\max\{|y-x|:(s,y)\aran(t,x), t\in[s,s+2],s\in\Z_+/n,s\le n^q\}\le n^{-\alpha}\Bigr\}.
\end{align*}
By Lemma~\ref{endpointinc} and \eqref{sprobbnds}  and the fact that $n^q\le n$ 
\begin{align}\label{omeganbnd}
\mu_n(\Omega_n^c)\le \frac{C}{n^q}+C_{\ref{endpointinc}}n^{q+2 -\kappa((1/2)-\alpha)}
&\le\frac{C}{n^q},
\end{align}
where the definition of $q$ is used in the last line.  To avoid confusion we denote the $\delta_n$ arising in
Proposition~\ref{prop:prelmod} by $\delta_n^{\ref{prop:prelmod}}$ and then define
\[\delta_n(\omega)=\begin{cases} \delta_n^{\ref{prop:prelmod}}(\omega),&\text{ if }\omega\in\Omega_n,\\
0,&\text{ if }\omega\in\Omega_n^c.
\end{cases}\]
Then for $\rho\in[0,1)$,
\[\mu_n(\delta_n\le \rho)\le\mu_n(\Omega_n^c)+\mu_n(\delta_n^{\ref{prop:prelmod}}\le\rho)\le\frac{C}{n^q}+C_{\ref{prop:prelmod}}\rho^\beta,\]
the last by \eqref{omeganbnd} and Proposition~\ref{prop:prelmod}. This proves \eqref{deltabndgen}.

In proving \eqref{modcontcondgen} implies \eqref{modcontgen} to reduce subscripts we assume
\begin{equation}
\nn
0\le s<t, \quad |t-s|\le\delta_n, \quad \text{ and }(s,y)\aran(t,x).
\end{equation}
This implies $\delta_n>0$ and so $\omega\in\Omega_n$,
 which in turn implies $S^{\sn}\le n^q$ and so (as $\mT_t^{\sn}$ is non-empty),
\begin{equation}\label{tlen}
t\le S^{\sn}\le n^q.
\end{equation}
We must show that, for an appropriate $C_{\ref{thm:mod_con}}$,
\begin{equation}\label{genmod}
|x-y|\le C_{\ref{thm:mod_con}}[|t-s|^\alpha+n^{-\alpha}].
\end{equation}
If $t\le 2/n$ this follows easily from $\omega\in\Omega_n$ and the triangle inequality, so assume without
loss of generality that $t>2/n$.  Write $[u]_n=\floor{nu}/n$, and set $t_n=[t-n^{-1}]_n, t'_n=t_n+n^{-1}, s_n=[s]_n\wedge t_n\le t'_n$, all in $\in\Z_+/n$.   Therefore 
\begin{equation}\label{tnbounds}0\le t_n-s_n\le t-n^{-1}-(s-n^{-1})=t-s\le\delta_n
\end{equation}
(we can assume $s_n=[s]_n$ in the above derivation, since otherwise $s_n=t_n$ and the desired upper bound in \eqref{tnbounds} is trivial), 
and 
\begin{equation}\label{tsn}
0\le t-t_n\le2/n,\ 0\le s-s_n\le 2/n.
\end{equation}

Since $s_n\le s$ and $y\in \mT^{\sn}_s$ by Lemma~\ref{anseq} there exists $y_n$ s.t. 
\[(s_n,y_n)\aran(s,y).\]
We also have
$s_n\le t_n<t_n+n^{-1}\le t\le n^q$ (by \eqref{tlen}) and $x\in\mT^{\sn}_t$, so by Lemma~\ref{anseq} there are $x_n,x_n'$ s.t.
\begin{equation}\label{chain}(s_n,y_n)\aran(t_n,x_n)\aran(t_n+n^{-1},x_n')\aran(t,x).\end{equation}
 Therefore by  $\omega\in\Omega_n$ and \eqref{tsn} we may conclude,
\begin{align*}
|x-y|&\le |x-x_n|+|x_n-y_n|+|y_n-y|\\
&\le n^{-\alpha}+|x_n-y_n|+n^{-\alpha}\\
&\le 2n^{-\alpha}+C_{\ref{prop:prelmod}}|t_n-s_n|^\alpha\quad\text{(by \eqref{tnbounds}, \eqref{chain} and Proposition~\ref{prop:prelmod})}\\
&\le 2n^{-\alpha}+C_{\ref{prop:prelmod}}|t-s|^\alpha,
\end{align*}
the last by \eqref{tnbounds}.
This proves \eqref{genmod} and the proof is complete.\qed

\section{Convergence of the range}
\label{sec:range}
In this section we prove Theorem \ref{thm:range}.  Recall that $\N_0^s(\cdot)=\N_o(\cdot\,|S>s)$.

\begin{LEM}\label{N0intdistn}
\begin{align}
\nn
\N_o^1\left(\int_1^2X_s(1)ds\le a \right)&\sim \frac{4}{ \sqrt{2\pi\gamma}} \sqrt{a}, \qquad \text{ as }a \downarrow 0\\
\label{N0intbnd}\N^1_o\left(\int_1^2X_s(1)ds\le a\right)&\le 2  \sqrt{\frac{2}{\gamma}}\ee\sqrt{a}, \qquad \text{ for every }a>0.
\end{align}
\end{LEM}
\begin{proof}

First recall from \eqref{sbmext} that $\N_o(S>1)=2/\gamma$, and by Theorem~II.7.2(iii) of \cite{Per02}
\begin{equation}
\N_o^1\big(X_1(1) \in \dd x \big)=(2/\gamma)\ee ^{-(2/\gamma)x}\dd x.\label{density1}
\end{equation}
Let $v^{(\lambda)}_t=\sqrt{\frac{2\lambda}{\gamma}}\frac{[\ee ^{t \sqrt{2\gamma\lambda}}-1]}{[\ee^{t \sqrt{2\gamma\lambda}}+1]}$
be the unique solution of 
\begin{equation}
\nn
\frac{\dd v}{\dd t}=\frac{-\gamma v_t^2}{2}+\lambda, \qquad v(0)=0.
\end{equation}
Then the Markov property under $\N_o$ and exponential duality (see Theorem~II.5.11(c) of \cite{Per02}), together with \eqref{density1} gives
\begin{align*}
\N_o^1&\left[\exp\Big\{-\lambda \int_1^2X_s(1)ds\Big\}\right]\\
&=\N_o\left[ \E_{X_1}\Big[\exp\Big\{-\lambda \int_0^1X_s(1)ds\Big\}\indic{S>1}\Big]\right]/\N_o(S>1)\\
&=\int_0^\infty \exp\{-x v_1^{(\lambda)}\}\frac{2}{\gamma} \ee^{-2x/\gamma} \dd x,\qquad  \\
&=\frac{2}{2+\gamma v_1^{(\lambda)}}.
\end{align*}
Since $v_1^{(\lambda)}\sqrt{\gamma/(2 \lambda)} \ra 1$ as $\lambda \ra \infty$, 
\begin{equation*}
\N_o^1\left[\exp\Big\{-\lambda \int_1^2X_s(1)ds\Big\} \right]\sim \sqrt{\frac{2}{\gamma \lambda}}, \qquad \text{ as } \lambda \ra \infty.
\end{equation*}
A Tauberian theorem (e.g.~Theorems 2,3 in Section XIII.5 of \cite{Feller2}) now gives
\begin{equation*}
\N_o^1\left(\int_1^2X_s(1)ds\le a\right)\sim \frac{4}{ \sqrt{2\pi\gamma}} \sqrt{a}, \qquad \text{ as }a \downarrow 0.
\end{equation*}

Now 
\begin{align*}
v_1^{(\lambda)}&=\sqrt{\frac{2\lambda}{\gamma}}\frac{[\ee ^{ \sqrt{2\gamma\lambda}}-1]}{[\ee^{ \sqrt{2\gamma\lambda}}+1]}=\sqrt{\frac{2\lambda}{\gamma}}\left[1-\frac{2}{\ee^{ \sqrt{2\gamma\lambda}}+1}\right]\ge  \frac{1}{2}\sqrt{\frac{2\lambda}{\gamma}}, \quad \text{ if }\lambda \ge \gamma^{-1}.
\end{align*}
If $a\le \gamma$ and so $\lambda:=1/a\ge\gamma^{-1}$, then using $\1_{Y\le a}\le \ee^{\lambda(a-Y)}$ we have
\begin{align*}
\N_o^1\left(\int_1^2X_s(1)ds\le a \right)&\le \ee^{\lambda a} \N^1_o\left[\exp\Big\{-\lambda \int_1^2X_s(1) \dd s\Big\} \right]\\
&\le \ee^{\lambda a} \frac{2}{2+\gamma v_1^{(\lambda)}}\le \dfrac{ 2\ee}{2+\sqrt{\frac{\lambda\gamma}{2}}},\qquad \\
&\le \frac{2\sqrt{2} \ee}{\sqrt\gamma} \sqrt{a}.
\end{align*}
For $a>\gamma$ the bound is trivial.
\end{proof}

\begin{LEM}\label{zerorange}
$0\in R\ \  \N_o-a.e.$
\end{LEM}
\begin{proof} This is immediate from the description of the integral of SBM in terms of  the 
Brownian snake, and the continuity of the snake under $\N_o$. See Proposition~5 and Section 5
of Chapter IV of \cite{Legall}. 
\end{proof}
\bigskip

Recall that $\mc{K}$ is the Polish space of compact subsets of $\R^d$, equipped with the Hausdorff metric $d_0$ and $\Delta_1$ is as in \eqref{Delta'def}.  \SUBMIT{The elementary proof of the following lower semicontinuity result may be found in \cite{Arxiv_version}.}
\begin{LEM}\label{suppusc}
If $\nu_n \rightarrow \nu$ in $M_F(\R^d)$, and $\supp(\nu)\in \mc{K}$, then 
$$\Delta_1(\supp(\nu),\supp(\nu_n)) \ra 0.$$  
\end{LEM}
\ARXIV{
\begin{proof}
Fix $\vep>0$.  We must show that $\supp(\nu)\subset \supp(\nu_n)^\vep$ for all $n$ sufficiently large. 
Let $x\in \supp(\nu)$.  Then $\liminf_{n \ra \infty}\nu_n(B(x,\vep/2))\ge \nu(B(x,\vep/2))>0$.  Therefore there exists $n(x)$ such that for every $n\ge n(x)$, 
\begin{align}
\nu_n(B(x,\vep/2))\ge \frac{1}{2} \nu(B(x,\vep/2))>0,\nn
\end{align}
and therefore $x \in \supp(\nu_n)^{\vep/2}$.
As $\supp(\nu)$ is compact there exist $x_1,\dots, x_k \in \supp(\nu)$ such that $\supp(\nu)\subset \cup_{i=1}^k B(x_i,\vep/2)$.  Thus if $n\ge n_0=\max_{i\le k}n(x_i)$ then
\[\supp(\nu)\subset \cup_{i=1}^kB(x_i,\vep/2)\subset \supp(\nu_n)^\vep,\]
as required.
\end{proof}
}
In the rest of the Section we will assume Conditions~\ref{cond:surv}-\ref{cond:self-avoiding}, let $\alpha,\beta$ satisfy \eqref{alphbet} and \eqref{alphbet2}, and assume $(\gamma,\sigma_0^2)$ are as in Condition~\ref{cond:finite_int}. The parameter $q$ is as in Theorem~\ref{thm:mod_con}. We start with some elementary consequences of Theorem~\ref{thm:mod_con}. Recall from \eqref{rangefinite} that $R^{\sn}$ is a.s. finite.  For $t\ge 0$, define
\[R^{\sn}_t=\cup_{s\le t}\mc{T}^{\sn}_s\subset R^{\sn}.\]

\begin{LEM}\label{mcc} (a) There is a $c_{\ref{mcc}}>0$ such that on $\{\delta_n\ge n^{-1}\}$ we have
\begin{equation}\label{r0bnd}
r_0(R^{\sn})\le c_{\ref{mcc}}(S^{\sn}\delta_n^{\alpha-1}+1).
\end{equation}
(b) $\eta_{\ref{mcc}}(u):=\sup_{n\ge 1}\mu_n(r_0(R^{\sn})\ge u)\to 0\text{ as }u\to\infty$.\\
(c) For any $\tau_0,\vep,s>0$ there is a $\tau=\tau(\tau_0,\vep,s)>0$ and $n_0(\tau_0,\vep,s)\ge 1$ so that 
\[\P(R_\tau^{\sn}\not\subset B(0,\tau_0)|S^{\sn}>s)\le \vep\quad \forall n\ge n_0.\]
\end{LEM}
\begin{proof} (a) Assume $\delta_n(\omega)\ge 1/n$. Assume also $x\in\mT_s^{\sn}$ for some $s> 0$.  Choose $M\in\N$ so that $(M-1)\delta_n<s\le M\delta_n$, and set $s_i=i\delta_n$ for $0\le i<M$ and $s_M=s$.  Clearly $s<S^{\sn}$ (since $\mc{T}^{\sn}_s$ is non-empty) and so
\begin{equation}\label{Mbound1}
M\le s\delta_n^{-1}+1\le S^{\sn}\delta_n^{-1}+1.
\end{equation}
By Lemma~\ref{anseq} there are $y_i\in\mT_{s_i}$ for $0\le i\le M$ s.t. $y_0=o$, $y_M=x$ and $(s_{i-1},y_{i-1})\aran(s_i,y_i)$ for $1\le i\le M$.  
Theorem~\ref{thm:mod_con} implies that for all $1\le i\le M$, $|y_i-y_{i-1}|\le 2C_{\ref{thm:mod_con}}\delta_n^{\alpha}$, and so by the triangle inequality, \eqref{Mbound1}, and $\delta_n\le 1$,
\[|x|=|y_M|\le M 2C_{\ref{thm:mod_con}}\delta_n^{\alpha}\le 2C_{\ref{thm:mod_con}}[S^{\sn}\delta_n^{\alpha-1}+1].\]
This gives (a) with $c_{\ref{mcc}}=2C_{\ref{thm:mod_con}}$.\\
(b) Use (a) to see that for $u,n\ge 1$ and $u> 2c_{\ref{mcc}}$,
\begin{align}\label{Rbnd1}\mu_n&(r_0(R^{\sn})\ge u)\\
\nonumber&\le \mu_n(\delta_n\le u^{-1}\vee n^{-1})+\mu_n(S^{\sn}\delta_n^{\alpha-1}\ge (u/c_{\ref{mcc}})-1,\delta_n\ge 1/u)\\
\nonumber&\le C_{\ref{thm:mod_con}}(u^{-\beta}+n^{-\beta}+n^{-q})+\mu_n(S^{\sn}\ge u^{\alpha-1}[(u/c_{\ref{mcc}})-1]> u^\alpha/(2c_{\ref{mcc}})),
\end{align}
where in the last line we have used Theorem~\ref{thm:mod_con} and the lower bound on $u$.  Now \eqref{sprobbnds}  implies that
\begin{align*}\mu_n(S^{\sn}>u^\alpha/(2c_{\ref{mcc}}))&\le \overline s_Dc_{\ref{mdef}}/((u^\alpha/2c_{\ref{mcc}})\wedge n).
\end{align*}
Using this in the bound \eqref{Rbnd1} we see that for any $\vep>0$ there is an $n_0\ge 1$ such that 
\begin{equation}\label{largenbnd}
\mu_n(r_0(R^{\sn})\ge u)<\vep\qquad\text{for all } n,u\ge n_0.
\end{equation}
But for $n \in [1,n_0)\cap \N$ we have 
\[\mu_n(r_0(R^{\sn})\ge u)\le m(n_0)\P(r_0(R^{\sss(1)})\ge u)<\vep,\]
for $u>u_0(\vep)$. The result follows from this last inequality and \eqref{largenbnd}.\\
(c) Fix $\tau_0,\vep,s>0$ and then choose $n_1(\tau_0)\ge 1$ and $\tau_1(\tau_0)>0$ so that 
\begin{equation}
\nn
n>n_1(\tau_0)\text{ and }\tau\in [0,\tau_1(\tau_0))\text{ imply } C_{\ref{thm:mod_con}}(\tau^\alpha+n^{-\alpha})<\tau_0.
\end{equation}
So for $n>n_1$ and $ \tau\in (0,\tau_1)$, by Theorem~\ref{thm:mod_con} (and the fact that $y\in\mT^{\sn}_s$ for some $s\le \tau$ implies $(0,o)\aran(s,y)$) we have
\[\tau\in (0,\delta_n]\text{ implies }R_\tau^{\sn}\subset B(o,C_{\ref{thm:mod_con}}(\tau^\alpha+n^{-\alpha}))\subset B(o,\tau_0).\]
Therefore using \eqref{sprobbnds} and \eqref{deltabndgen} we see that for $n>n_1$ and $\tau\in (0,\tau_1)$,
\begin{align*}
\P(R^{\sn}_\tau\not\subset B(o,\tau_0)|S^{\sn}>s)&\le\P(\delta_n<\tau,S^{\sn}>s)/\P(S^{\sn}>s)\\
&\le \mu_n(\delta_n<\tau)/\mu_n(S^{\sn}>s)\\
&\le C_{\ref{thm:mod_con}}(\tau^\beta+n^{-q})c_{\ref{mdef}}(s\vee 1)(\underline s_D)^{-1}\\
&<\vep,
\end{align*}
where the last inequality holds for $\tau$ sufficiently small and $n$ sufficiently large, depending only 
on $\vep$ and $s$.  The result follows.
\end{proof}
\begin{LEM}\label{intcvgce}
For every $s>0$, 
\[\P_n^{s}\left(\xnint\in \cdot\right) \cweak \N_o^s\left(\xint \in \cdot\right) \text{ on }\mc{M}_F(\R^d).\]
\end{LEM}
\begin{proof} 
\SUBMIT{If $t_\vep=s\vee \vep^{-1}$, then on $\{S^{\sn}<t_\vep\}$ we have $\xnint=\bar X^{\sn}_{t_\vep}$ which, by Condition~\ref{cond:finite_int}, converges weakly under $\P^s_n$ to $\N_o^s(\bar X_{t_\vep}\in\cdot)$. Condition~\ref{cond:surv} shows that the contribution from $\{S^{\sn}>t_\vep\}$ is small uniformly in $n$ as $\vep$ becomes small.  The details are standard and appear in \cite{Arxiv_version}.}
\ARXIV{
Fix $s>0$ and $\vep>0$.  Choose $t_\vep>\max\{1/\vep,s\} $ and let $\varphi\in C_b(M_F(\R^d))$.  Then $\big|\P_n^{s}[\varphi(\xnint)]-\N_0^{s}[\varphi(\bar X_\infty)]\big|$ is at most 
\begin{align}
&\big|\P_n^{s}[\varphi(\xnint)\indic{S^{\sn}\le t_\vep}]-\N_0^{s}[\varphi(\xint)\indic{S\le t_\vep}]\big|\nn\\
&\quad+\big|\P_n^{s}[\varphi(\xnint)\indic{S^{\sn}>t_\vep}]-\N_0^{s}[\varphi(\xint)\indic{S>t_\vep}]\big|\nn\\
\le &  \big|\P_n^{s}[\varphi(\bar{X}_{t_\vep}^{\sn})]-\N_0^{s}[\varphi(\bar{X}_{t_\vep})]\big|\label{l1}\\
&\quad +\big|\P_n^{s}[\varphi(\bar{X}_{t_\vep}^{\sn})\indic{S^{\sn}> t_\vep}]-\N_0^{s}[\varphi(\bar{X}_{t_\vep})\indic{S>t_\vep}]\big|\label{l2}\\
&\quad + \|\varphi\|_\infty (\P_n^{s}(S^{\sn}>t_\vep)+\N_0^{s}(S>t_\vep)).\label{l3}
\end{align}
For $n $ sufficiently large (depending on $s,t_\vep,\vep$) the term \eqref{l1} is less than $\vep$ by Condition \ref{cond:finite_int}.  The quantity \eqref{l2} is equal to
\begin{align}
\big|\P_n^{t_\vep}[\varphi(\bar{X}_{t_\vep}^{\sn})]\frac{\P(S^{\sn}>t_\vep)}{\P(S^{\sn}>s)}-\N_0^{t_\vep}[\varphi(\bar{X}_{t_\vep})]\frac{\N_0(S>t_\vep)}{\N_0(S>s)}\big|.\nn
\end{align}
This is less than $\vep$ for $n$ sufficiently large depending on $s,t_\vep,\vep$ by Condition \ref{cond:finite_int} and Condition \ref{cond:surv}.

Similarly, \eqref{l3} is equal to 
\[ \|\varphi\|_\infty  \Big(\frac{\P(S^{\sn}>t_\vep)}{\P(S^{\sn}>s)}+\frac{\N_0(S>t_\vep)}{\N_0(S>s)}\Big),\]
which is less than or equal to $c \vep $ for $n$ sufficiently large (where $c$ depends on $s$ and $\|\varphi\|_\infty $) due to Condition \ref{cond:surv}. 
}
\end{proof}

We set $K=1024$ and for $i\in\Z_+$, $\ell\in\N$ and $n\in[1,\infty)$ define
\begin{align*}
\Omega^n_{i,\ell}=\{\omega:&\exists x\in\mT^{\sn}_{i2^{-\ell}}, \exists x'\text{ s.t. }(i2^{-\ell},x)\aran((i+1)2^{-\ell},x')\text{ and }\\
& \int_{(i+1)2^{-\ell}}^{(i+2)2^{-\ell}}|\{y:(i2^{-\ell},x)\aran (s,y)\}|ds/m(n)\le K^{-\ell},
\end{align*}
and
\[\hat\Omega_\ell^n=\cup_{i=0}^{2^{2\ell}-1}\Omega^n_{i,\ell}.\]
The following result together with the modulus of continuity will ensure that, with high probability, points 
in the discrete range are near areas of significant integrated mass.\\

\begin{LEM}\label{notinymass} There is a $c_{\ref{notinymass}}>0$ and for any $n\ge 1$ there is an $M_n\in\N$ so that $M_n\uparrow \infty$ as $n\to\infty$, and if $\Omega_m^n=\cup_{\ell=m}^{M_n}\hat\Omega_\ell^n$ for $m
\in\N$, then 
\begin{equation}
\nn
\mu_n(\Omega_m^n)\le c_{\ref{notinymass}}2^{-m}\quad\text{for all }n\in[1,\infty),\ m\in\N.
\end{equation}
\end{LEM}
\begin{proof} Let $\ell\in\N$, $i\in\{0,1,\dots,2^{2\ell}-1\}$ and assume
\begin{equation}\label{nlowerbound}
n\ge K^{3\ell/2}.
\end{equation}
A simple change of variables ($u=ns$) in the time integral in the definition of $\Omega^n_{i,\ell}$ shows
that 
\begin{align*}
\mu_n(\Omega^n_{i,\ell})\le m(n)\E&\Bigg[\sum_{x\in\mT_{ni2^{-\ell}}}\1(\exists\,x'\text{ s.t. }(ni2^{-\ell},x)\ara(n(i+1)2^{-\ell},x'))\\
&\times \1\Bigl(\int_{ni2^{-\ell}+n2^{-\ell}}^{ni2^{-\ell}+2n2^{-\ell}}|\{y:(ni2^{-\ell},x)\ara(u,y)\}|du\le nm(n)K^{-\ell}\Bigr)\Bigg].
\end{align*}
So using Condition~\ref{cond:self-avoiding} with $\Delta=n 2^{-\ell}\ge 4$ (by \eqref{nlowerbound}) and $t=ni2^{-\ell}$ we see that
\begin{align}\label{munbound1}
&\mu_n(\Omega^n_{i,\ell})\\
\nonumber&\ \le c_{\ref{cond:self-avoiding}}m(n)\P\Bigl(S^{\sss(1)}>n2^{-\ell},\, \int_{n2^{-\ell}+2}^{2n2^{-\ell}-2}\frac{|\mT_s|}{m(n2^{-\ell})n2^{-\ell}}ds\le \frac{m(n) K^{-\ell}}{m(n 2^{-\ell})2^{-\ell}}\Bigr)\\
\nonumber&\ \le c_{\ref{cond:self-avoiding}}m(n)\P\Bigl( S^{\sss(n2^{-\ell})}>1,\, \int_{1+(2/(n2^{-\ell}))}^{2-(2/(n2^{-\ell}))}X_u^{(n2^{-\ell})}(1)du\le (K/2)^{-\ell}c_{\ref{mdef}}2^{\ell}\Bigr),
\end{align}
where in the last inequality we used  \eqref{mdef}, which applies because $n2^{-\ell}\ge 1$ by \eqref{nlowerbound}. 

If $I_{n,\ell}=[1,1+\frac{2^{\ell+1}}{n}]\cup[2-\frac{2^{\ell+1}}{n},2]$, then
\begin{align}
\nonumber\E\Bigg[&\int\1(u\in I_{n,\ell})X_u^{(n2^{-\ell})}(1)du\,\Big|\,S^{\sss(n2^{-\ell})}>1\Bigg]\\
\nonumber&=\int_{I_{n,\ell}}\E\big[X_u^{(n2^{-\ell})}(1)\big]du/\P(S^{\sss(1)}>n2^{-\ell})\\
\label{Inlbound}&=\int_{I_{n,\ell}}\frac{\E[|\mT_{un2^{-\ell}}|]}{m(n2^{-\ell})\P(S^{\sss(1)}>n2^{-\ell})}du\le C_{\ref{Inlbound}}\frac{2^\ell}{n},
\end{align}
where Condition~\ref{cond:L1bound} and \eqref{survivbnds} are used in the last inequality.

Use the upper bound from \eqref{munbound1}, \eqref{mdef}, and then \eqref{survivbnds} to conclude that
\begin{align}
\nonumber \mu_n&(\Omega^n_{i,\ell})\\
\nonumber&\le c_{\ref{cond:self-avoiding}}\frac{m(n)}{m(n2^{-\ell})}m(n2^{-\ell})\P(S^{\sss(1)}>n2^{-\ell})\\
&\qquad \times \P\Big(\int_{1+(2/(n2^{-\ell}))}^{2-(2/(n2^{-\ell}))}X_u^{(n2^{-\ell})}(1)du\le c_{\ref{mdef}}(K/4)^{-\ell}\Bigr|S^{\sss(n2^{-\ell})}>1\Big)\nn\\
\nonumber&\le c_{\ref{cond:self-avoiding}}c_{\ref{mdef}}2^\ell \bar s_D\P\Bigl(\int_{1+(2^{\ell+1}/n)}^{2-(2^{\ell+1}/n)} X_u^{(n2^{-\ell})}(1)du\le c_{\ref{mdef}}(K/4)^{-\ell}\Bigr|S^{\sss(n2^{-\ell})}>1\Bigr)\\
\nonumber&\le C\Bigl[2^{\ell}\P\Bigl(\int_1^2 X_u^{(n2^{-\ell})}(1)\,du\le 2c_{\ref{mdef}}(K/4)^{-\ell}\Bigr|S^{\sss(n2^{-\ell})}>1\Bigr)\\
\nonumber&\phantom{\le C2^{\ell}\Bigl[}+2^{\ell}\P\Bigl(\int_{I_{n,\ell}}X_u^{(n2^{-\ell})}(1)du>c_{\ref{mdef}}(K/4)^{-\ell}\Bigr|S^{\sss(n2^{-\ell})}>1\Bigr)\Bigr]\\
\label{Tsdefn}&:=C[T_1(\ell,n)+T_2(\ell,n)].
\end{align}

Markov's inequality and \eqref{Inlbound} imply that
\begin{equation}\label{T2bound}
T_2(\ell,n)\le \frac{C_{\ref{Inlbound}}2^{2\ell}}{n}c_{\ref{mdef}}^{-1}(K/4)^{\ell}:=C\frac{K^\ell}{n}\le C\,K^{-\ell/2},\ \forall n\ge K^{3\ell/2}.
\end{equation}
By Condition~\ref{cond:finite_int} and the upper bound in \eqref{N0intbnd}, 
\begin{align*}
\limsup_{n\to\infty}&\ 2^\ell\P\Bigl(\int_1^2 X_u^{(n2^{-\ell})}(1)du\le 2c_{\ref{mdef}}(K/4)^{-\ell}\Bigr|S^{(n2^{-\ell})}>1\Bigr)\\
&\le 2^\ell\N_o\Bigl(\int_1^2X_u(1)du\le 2 c_{\ref{mdef}}(K/4)^{-\ell}\Bigr|S>1\Bigr)\\
&\le (C/2)(\sqrt K/4)^{-\ell}.
\end{align*}
Therefore there is an increasing sequence $\{n(\ell):\ell\in\N\}$ such that $n(\ell)\ge K^{3\ell/2}$ and 
\begin{equation}\label{T1bound}
T_1(\ell,n)\le C(\sqrt K/4)^{-\ell}\quad\text{ for all }n\ge n(\ell).
\end{equation}
Use \eqref{T2bound} and \eqref{T1bound} in \eqref{Tsdefn} and then sum over $i=0,\dots, 2^{2\ell}-1$ and deduce that 
\begin{equation}\label{omegahatbound1}
\mu_n(\hat\Omega_\ell^n)\le C(\sqrt K/16)^{-\ell}=C2^{-\ell}\quad\text{for all }n\ge n(\ell).
\end{equation}
For $n\ge 1$, define $M_n=\max\{\ell:n(\ell)\le n\}\uparrow\infty$ as $n\to\infty$ ($\max\varnothing=0$).  Then by \eqref{omegahatbound1} for any $m\in\N$,
\[\mu_n(\Omega_m^n)\le \sum_{\ell=m}^{M_n}\mu_n(\hat\Omega_\ell^n)\le \sum_{\ell=m}^{M_n}C2^{-\ell}\le 2C2^{-m},\]
because $\ell\le M_n$ implies that $n(\ell)\le n$.
\end{proof}

\noindent{\it Proof of Theorem~\ref{thm:range}.} Recall \eqref{edefn}.  Let 
$$\tilde e^{\sn}=\big(({\hat e}^{\sn}_t(y/\sqrt n,x/\sqrt n))_{t \ge 0}:y,x\in\Z^d)\in\big(\mc{D}([0,\infty),\mc{D}_{\R})\big)^{\Z^d\times\Z^d}:=\tilde{\mc{D}},$$  
where \eqref{cadlage} is used to show that, using the usual countable product metric,
\begin{align}
\nn
\tilde e^{\sn}&\text{ belongs to the complete separable metric space $\tilde{\mc{D}}$.}
\end{align}
For $s>0$, the joint probability law of $(\bs{X}^{\sn},\tilde e^{\sn})$ (on $\mc{D}\times\tilde{\mc{D}}$) conditional on $S^{\sn}>s$ is written as
\begin{align*}
&\mu^s_n((\bs{X}^{\sn},\tilde e^{\sn})\in\cdot)\\
&=\mu_n((\bs{X}^{\sn},\tilde e^{\sn})\in\cdot|S^{\sn}>s).
\end{align*}
Although $\mu^s_n=\P^s_n$, we will soon be trading in our familiar $\P$ to apply Skorokhod's Theorem and so to avoid confusion it will be useful to work with $\mu^s_n$ on our original probability space.

Fix $s>0$. It suffices to show weak convergence along any sequence $n_k\to\infty$ and to ease the notation
we will simply assume $n\in\N$ (the proof being the same in the general case).  By Lemma~\ref{intcvgce} and the Skorokhod Representation Theorem we may work on some $(\Omega,\mc{F},\PS)$ on which there are random measures $(\xnint)_{n\in\N}$ and $\xint$ such that $\xnint$ has law $\mu_n^s(\xnint\in\cdot)$ for each $n$, $\xint$ has law $\N_o^s(\xint\in\cdot)$, and 
\begin{equation}\label{asconv}
\xnint \cas\xint\quad\text{ as }n\to\infty.
\end{equation}
For now we may assume $\mc{F}=\sigma((\xnint)_{n\in\N}, \xint)$.  We claim that we may assume that for all $n\in\N$ there are processes $(\bs{X}^{\sn},\tilde e^{\sn})\in\mc{D}\times\tilde{\mc{D}}$ defined on $(\Omega,\mc{F},\PS)$ such that
\begin{align}
\nn
\forall n\in\N\ \ &\text{ the law of }(\bs{X}^{\sn},\tilde e^{\sn})\text{ on }\mc{D}\times\tilde{\mc{D}}\text{ is }\mu_n^s((\bs{X}^{\sn},\tilde e^{\sn})\in\cdot)\\
\nn&\text{ and }\xnint=\int_0^\infty X_u^{\sn}\,du\ \text{ a.s.}
\end{align}
To see this, work on $\bar\Omega=\Omega\times\mc{D}\times\tilde{\mc{D}}$ with the product $\sigma$-field $\bar{\mc{F}}$, and define $\bar\P^{\{s\}}$ on $\bar{\mc{F}}$ by
\begin{align*}
\bar\P^{\{s\}}&\Bigl((\xint,(\xnint)_{n\in\N})\in A, (\bs{X}^{\sn},\tilde e^{\sn})_{n\le N}\in\prod_{n=1}^N B_n\Bigr)\\
&=\int \1_A(\xint,(\xnint)_{n\in\N})\prod_{n=1}^N\mu_n^s((\bs{X}^{\sn},\tilde e^{\sn})\in B_n|\xnint)d\PS,
\end{align*}
where the above conditional probabilities are taken to be a regular conditional probabilities. 
In short, given $(\xint,(\xnint)_{n\in\N})$, we take $(\bs{X}^{\sn},\tilde e^{\sn})_{n\in\N}$ to be conditionally independent and with laws $\mu_n^s((\bs{X}^{\sn},\tilde e^{\sn})\in\cdot|\xnint)$.  Clearly the resulting enlarged space
satisfies the above claim. We now relabel the enlarged space as $(\Omega,\mc{F},\PS)$ to ease the notation.

Note that for each fixed $n$, $\bs{\mT}^{\sn}=(\mT^{\sn}_t)_{t\ge 0}:=(\text{supp}(X^{\sn}_t))_{t\ge0}$ is a copy of our rescaled set-valued process so we can define $S^{\sn}, R^{\sn}$ and $\aran$ just as before using $\tilde e^{\sn}$ and $\bs{\mT}^{\sn}$. For the latter, set $e^{\sn}_{s,t}(y,x)=\tilde e^{\sn}_t(y,x)(s)$ and write $(s,u)\aran(t,x)$ iff $s\le t$ and $e^{\sn}_{s,t}(x,y)=1$ (recall \eqref{edefn} and \eqref{endef}).  
In particular $\Omega^n_{i,\ell}$ and $\Omega^n_m$, as in Lemma~\ref{notinymass} can be defined as subsets of $\Omega$, and that result and \eqref{sprobbnds} imply
\begin{align}\label{omegamnbound}
\nonumber\PS(\Omega^n_m)=\mu_n^s(\Omega_m^n) \le \frac{\mu_n(\Omega_m^n)}{\mu_n(S^{\sn}>s)}&\le c_{\ref{notinymass}}2^{-m}/\mu_n(S^{\sn}>s)\\
&\le \frac{c_{\ref{notinymass}}c_{\ref{mdef}}}{\underline s_D}(s\vee 1)2^{-m}:=c_{\ref{omegamnbound}}(s)2^{-m}.
\end{align}
To reinterpret Theorem~\ref{thm:mod_con} we introduce
\begin{align}\label{Deltandefn}\Delta^{\sn}(\rho)=\sup\{|y_2-y_1|:\  |s_2-s_1|\le\rho,\ (s_1,y_1)\aran(s_2,y_2)\},
\end{align}
and note that the law of $\Delta^{\sn}(\rho)$ is the same under $\PS$ and $\mu_n^s$.  Therefore
\begin{equation}
\PS(\Delta^{\sn}(\rho)>C_{\ref{thm:mod_con}}(\rho^\alpha+n^{-\alpha}))\le \frac{\mu_n^s(\Delta^{\sn}(\rho)>C_{\ref{thm:mod_con}}(\rho^\alpha+n^{-\alpha}))}{\mu_n(S^{\sn}>s)}.\nn
\end{equation}
On the original space $\Delta^{\sn}(\rho)>C_{\ref{thm:mod_con}}(\rho^\alpha+n^{-\alpha})$ implies $\delta_n<\rho$ and so by Theorem~\ref{thm:mod_con} and \eqref{sprobbnds} (as used in \eqref{omegamnbound})
\begin{align}\label{Deltanbnd}
\nonumber\forall\rho\in[0,1],\ \ \PS(\Delta^{\sn}(\rho)>C_{\ref{thm:mod_con}}(\rho^\alpha+n^{-\alpha}))&\le \mu_n(\delta_n<\rho)/\mu_n(S^{\sn}>s)\\
&\le c_{\ref{Deltanbnd}}(s) [\rho^\beta+n^{-q}].
\end{align}
 Next note that \eqref{sprobbnds} implies for $m\in\N$,
\begin{equation}\label{SunderP}
\PS(S^{\sn}>2^{m-1})\le \frac{\mu_n(S^{\sn}>2^{m-1})}{\mu_n(S^{\sn}>s)}\le \frac{c_{\ref{SunderP}}(s)}{2^{m-1}\wedge n}.
\end{equation}
Finally use  \eqref{sprobbnds} and recall Lemma~\ref{mcc}(b) to see that for $m\in\N$,
\begin{equation}\label{r0underP}
\PS(r_0(R^{\sn})\ge 2^{m-1})\le\frac{\mu_n(r_0(R^{\sn})\ge 2^{m-1})}{\mu_n(S^{\sn}>s)}\le c_{\ref{r0underP}}(s)\eta_{\ref{mcc}}(2^{m-1}).
\end{equation}
We have $R^{\sn}=\mc{R}(\bs{X}^{\sn})=\supp(\xnint)$, and $R=\supp(\xint)$ defines the range of a SBM. It suffices to show
\[d_0(R^{\sn},R)\overset{\PS}{\longrightarrow} 0 \quad\text{ as }n\to\infty.\]
Lemma~\ref{suppusc} and \eqref{asconv} show that $\Delta_1(R,R^{\sn})\cas 0$ and so with $R^\delta:=\{x:d(x,R)\le \delta\}$  it suffices to fix $\tau_0>0$ and prove
\begin{equation}\label{rgoal}
\lim_{n\to\infty}\PS(R^{\sn}\not\subset R^{\tau_0})=0.
\end{equation}
Fix $\vep>0$ and let $\tau(\tau_0,\vep,s)$, $n_0(\tau_0,\vep,s)$ be as in Lemma~\ref{mcc}(c), and choose $0<\tau<\tau(\tau_0,\vep,s)$.  Recalling $R^{\sn}_\tau=\cup_{s\le \tau}\mT_s^{\sn}$, we see from Lemmas~\ref{zerorange} and \ref{mcc}(c) that
\begin{align}\label{Rdelta}
&\PS(R_\tau^{\sn}\not\subset R^{\tau_0})\le\PS(R_\tau^{\sn}\not\subset B(0,\tau_0))=\mu^s_n(R_\tau^{\sn}\not\subset B(0,\tau_0))\le\vep\\
\nn&\phantom{\P(R_\delta^{\sn}\not\subset R^{\delta_0})\le\P(R_\delta^{\sn}\not\subset B(0,\delta_0))\le \mu^s_n(R}\text{ for }n\ge n_0(\tau_0,\vep,s).
\end{align}
For $m\in\N$, define the finite grid of points $$G_m= \Bigl\{\frac{\vec{i}2^{-m}}{K_0}:\vec{i}\in\Z^d\Bigr\}\cap[-2^m,2^m]^d,$$ 
where $K_0\in\N$ is chosen so that 
\begin{equation}\label{Gapprox}
\forall x\in B(o,2^{-m})\ \exists q\in G_m\text{ so that }|q-x|\le 2^{-m}.
\end{equation}
Define the finite collection 
\[\mc{B}_m=\{B(x,\tau_0/2):\,x\in G_m\}.\]
Fix $m\in\N$ sufficiently large so that
\begin{align}\label{mcond1}
&(i)\ 2^{-m}+2C_{\ref{thm:mod_con}}2^{(1-m)\alpha}<\tau_0/10, \\
\label{mcond2}&(ii)\ 2^{1-m}<\tau, \\
\label{mcond3}&(iii)\ c_{\ref{omegamnbound}}(s)2^{-m}+c_{\ref{Deltanbnd}}(s)2^{(1-m)\beta}
+\frac{c_{\ref{SunderP}}(s)}{2^{m-1}}\\
\nn&\phantom{(iii)\ c_{\ref{omegamnbound}}}\qquad\qquad\qquad+c_{\ref{r0underP}}(s)\eta_{\ref{r0bnd}}(2^{m-1})<\vep.
\end{align}

If $\bs{W}$ denotes a $d$-dimensional Brownian motion with variance parameter $\sigma_0^2$ starting at $o$ under a probability measure $\P_o$, then for any open ball $B$, 
\begin{align*}
\E^{\sss\{s\}}[\xint(\partial B)]=\N_o^s[\xint(\partial B)]&\le\int_0^\infty\N_o[X_u(\partial B)]\,du/\N_o(S>s)\\
&=\int_0^\infty \P_o(W_u\in\partial B)\,du/\N_o(S>s)\\
&=0,
\end{align*}
where the third equality is standard (e.g. Theorem II.7.2(iii) of \cite{Per02}). 
Therefore by \eqref{asconv}, the above equality, and standard properties of the weak topology we have  (recall $K=1024$)
\begin{equation*}
\lim_{n\to\infty}\PS\Bigl(\sup_{B\in\mc{B}_m}|\xnint(B)-\xint(B)|\ge \frac{1}{2}K^{-m}\Bigr)=0.
\end{equation*}
So there is an $n_1=n_1(m,\vep,\tau_0,s)$ so that for $n\ge n_1$ we have
\begin{equation}\label{convballs}
\PS\Big(\sup_{B\in\mc{B}_m}|\xnint(B)-\xint(B)|\ge \frac{1}{2}K^{-m}\Big)<\vep,
\end{equation}
and, if $M_n$ is as in Lemma~\ref{notinymass}, then for $n\ge n_1$
\begin{align}\label{ncond}
&(i)\ m\le M_n,\\
\label{ncond2}&(ii)\ 2C_{\ref{thm:mod_con}}n^{-\alpha}<\tau_0/10 ,\\
\label{ncond3}&(iii)\ n>2^{m-1}\text{ and } [c_{\ref{Deltanbnd}}(s)+c_{\ref{SunderP}}(s)]n^{-q}<\vep.
\end{align}
Assume $n\ge n_1$, and then choose $\omega$ so that ($\Delta^{\sss(n)}$ is as in \eqref{Deltandefn})
\begin{align}\label{omegacond}
&(i)\ \omega\in(\Omega^n_m)^c,\quad (ii)\ \Delta^{\sn}(2^{1-m})\le C_{\ref{thm:mod_con}}(2^{(1-m)\alpha}+n^{-\alpha}),\\
\label{omegacond3}&(iii)\ S^{\sn}\le 2^{m-1}, \quad (iv)\ r_0(R^{\sn})<2^{m-1},\\
\label{omegacond5}&(v)\  \sup_{B\in\mc{B}_m}|\xnint(B)-\xint(B)|< \frac{1}{2}K^{-m}.
\end{align}
Let $x\in R^{\sn}\setminus R_\tau^{\sn}$, and then choose $s>\tau$ so that $x\in\mT_s^{\sn}$. By \eqref{mcond2}, (iii) in \eqref{omegacond3}, and $s<S^{\sn}$ (since $\mT_s^{\sn}$ is non-empty), we have 
\[2^{1-m}<\tau<s<S^{\sn}\le 2^{m-1},\]
and so we can choose $i\in\{0,\dots,2^{2m-1}\}$ so that $s\in[(i+1)2^{-m},(i+2)2^{-m})$.
Since $|x|\le r_0(R^{\sn})<2^{m-1}$ (by (iv) of \eqref{omegacond3}), \eqref{Gapprox} shows there is a $q\in G_m$ so that
\begin{equation}\label{xq}
|x-q|\le2^{-m}<\tau_0/10,
\end{equation}
the last by \eqref{mcond1}. By \eqref{convtr} $\exists\, x_0\in\mT^{\sn}_{i2^{-m}}$ such that $(i2^{-m},x_0)\aran(s,x)$.  Assume $u\in[(i+1)2^{-m},(i+2)2^{-m}]$ and $(i2^{-m},x_0)\aran (u,y)$ for some $y\in\mT^{\sn}_u$.  Then use \eqref{xq}, \eqref{Deltandefn}, and (ii) in \eqref{omegacond} to see that 
\begin{align}
\nonumber |y-q|&\le |y-x_0|+|x_0-x|+|x-q|\\
\nonumber&\le 2\Delta^{\sn}(2^{1-m})+\tau_0/10\\
\nonumber&\le 2C_{\ref{thm:mod_con}}(2^{(1-m)\alpha}+n^{-\alpha})+\tau_0/10\\
&\le 3\tau_0/10<\tau_0/2,\nn
\end{align}
where we have used \eqref{mcond1} and \eqref{ncond2} in the last line. This implies that
\begin{align}
\nonumber \xnint(B(q,\tau_0/2))&\ge\int_{(i+1)2^{-m}}^{(i+2)2^{-m}} X_u^{\sn}(B(q,\tau_0/2))\,du\\
\label{intlb}&\ge \int_{(i+1)2^{-m}}^{(i+2)2^{-m}} X_u^{\sn}(\{y\in \mT^{\sn}_u:\,(i 2^{-m},x_0)\aran (u,y)\})\,du.
\end{align}
By \eqref{convtr} and $(i2^{-m},x_0)\aran (u,y)$, there is an $x'\in\mT_{(i+1)2^{-m}}$ s.t. 
\begin{equation}
\nn
(i2^{-m},x_0)\aran((i+1)2^{-m},x')\aran (u,y).
\end{equation}
The fact that $\omega\notin \Omega_m^n$ (by (i) of \eqref{omegacond}), $m\le M_n$ (by \eqref{ncond}), $i\in\{0\dots,2^{2m-1}\}$, $x_0\in\mT_{i2^{-m}}^{\sn}$, and $(i2^{-m},x_0)\aran((i+1)2^{-m},x')$ shows that the right-hand side of \eqref{intlb} is at least $K^{-m}$.
  This inequality and 
\eqref{omegacond5} imply
\[\xint(B(q,\tau_0/2))>K^{-m}/2.\]
Use this and \eqref{xq} to conclude $\xint(B(x,\tau_0))>K^{-m}/2$ and hence that $x\in R^{\tau_0}$. We have shown that
\begin{equation}\label{omegaconc}
\text{For $n\ge n_1$, conditions }(\ref{omegacond})-(\ref{omegacond5})\text{ imply that }(R^{\sn}\setminus R^{\sn}_\tau)\subset R^{\tau_0}.
\end{equation}
Now use \eqref{omegamnbound}, \eqref{Deltanbnd}, \eqref{SunderP}, \eqref{r0underP}, and \eqref{convballs} to see that the 
probability, $P(m,n)$ that one of the $5$ conditions listed in \eqref{omegacond}-\eqref{omegacond5} fails
is at most
\begin{align*}
\PS&(\Omega_m^n)+\PS(\Delta^{\sn}(2^{1-m})>C_{\ref{thm:mod_con}}(2^{(1-m)\alpha}+n^{-\alpha}))\\
&\phantom{(\Omega_m^n)}+\PS(S^{\sn}>2^{m-1})+\PS(r_0(R^{\sn})\ge 2^{m-1})+\vep\\
&\le c_{\ref{omegamnbound}}(s)2^{-m}+c_{\ref{Deltanbnd}}(s)[2^{(1-m)\beta}+n^{-q}]+\frac{c_{\ref{SunderP}}(s)}{2^{m-1}\wedge n} \\
&\phantom{\le c_{\ref{omegamnbound}}(s)2^{-m}}+c_{\ref{r0underP}}(s)\eta_{\ref{mcc}}(2^{m-1})+\vep.
\end{align*}
Our lower bounds on $m$ and $n$ (in particular use \eqref{mcond3} and \eqref{ncond3}) shows
that the above is at most $3\vep$, and so we have shown
\[P(m,n)<3\vep\text{ for }n>n_1(m,\vep,\tau_0) \text{ and where $m$ is chosen as above}.\]
Recalling \eqref{omegaconc} we conclude that 
\[\PS((R^{\sn}\setminus R_\tau^{\sn})\not\subset  R^{\tau_0})<3\vep\text{ for }n>n_1.\]
This, together with \eqref{Rdelta}, proves \eqref{rgoal} and so completes the proof.\qed

\section{The extrinsic one-arm probability}
\label{sec:one-arm}
In this section we prove Theorem \ref{thm:one-arm}.  Recall that $r_0(G)=\sup\{|x|: x\in G\}$.  
\begin{LEM}\label{r0cont}
The map $r_0: \mc{K} \ra [0,\infty)$ is continuous.
\end{LEM}
\begin{proof}
Let $K_n \ra K$, and let $\vep>0$.  Choose $n_0$ sufficiently large so that $K_n \subset K^\vep$ and $K\subset K_n^\vep$ for all $n\ge n_0$.
Then for such $n$, $r_0(K_n)\le r_0(K)+\vep$ and $r_0(K)\le r_0(K_n)+\vep$.
\end{proof}

\noindent{\it Proof of Theorem \ref{thm:one-arm}.} Let $\vep\in(0,1)$, $\alpha,\beta$ be as in \eqref{alphbet} and \eqref{alphbet2}, and let $\delta_n$, $C_{\ref{thm:mod_con}}$ be as in Theorem~\ref{thm:mod_con}. $\N_o$ is the canonical measure for the $(\gamma,\sigma_0^2)$-SBM arising in Theorem~\ref{thm:range}.  By Lemma~\ref{lem:radrange} we have $\N_o(r_0(R)>1)<\infty$, and as $S>0$ $\N_0$-a.e., we may choose $s\in(0,1)$ small enough so that 
\begin{equation}\label{r0tiny}
\N_o(r_0(R)>1,S\le s)<\vep,
\end{equation}
and
\begin{equation}\label{ssmall}
C_{\ref{thm:mod_con}}(s^\alpha+s^\beta)<\vep/2.
\end{equation}
Assume $n_1\ge 1$ is large enough so that
\begin{equation}\label{n1a}
C_{\ref{thm:mod_con}}(n_1^{-\alpha}+n_1^{-q})<\vep/2,
\end{equation}
and
\begin{equation}\label{Sncvgt}
|\mu_n(S^{\sn}>s)-(s_D/s)|<\vep\quad\forall n\ge n_1,
\end{equation}
where \eqref{cond1'} is used for the last.

Assume that  $n\ge n_1$ and $S^{\sn}\le s\le\delta_n$.  Then for any $x\in\mT^{\sn}_t$ we must have $t\le S^{\sn}\le s\le \delta_n$, and so by Theorem~\ref{thm:mod_con} and $(0,o)\aran(t,x)$,
\[|x|=|x-o|\le C_{\ref{thm:mod_con}}(t^\alpha+n^{-\alpha})\le C_{\ref{thm:mod_con}}(s^\alpha+n^{-\alpha})<\vep,\]
the last inequality by \eqref{ssmall} and \eqref{n1a}. This proves that for all $n\ge n_1$,
\begin{equation}
\nn
S^{\sn}\le s\le \delta_n\Rightarrow r_0(R^{\sn})\le\vep<1.
\end{equation}
Therefore for $n\ge n_1$, we have by Theorem~\ref{thm:mod_con},
\begin{equation}\label{Sr0b}
\mu_n(r_0(R^{\sn})>1,S^{\sn}\le s)\le\mu_n(\delta_n<s)\le C_{\ref{thm:mod_con}}(s^\beta+n^{-q})<\vep,
\end{equation}
the last by \eqref{ssmall} and \eqref{n1a} again.  

If $D_c(h)$ denotes the set of discontinuity points of $h:\mc{K}\to\R$, then by Lemma~\ref{r0cont},
\[\N_o(R\in D_c(\1_{\{r_0>1\}}))\le \N_o(r_0(R)=1)=0,\]
where we have used Lemma~\ref{lem:radrange} in the last equality. So Theorem~\ref{thm:range} shows
that we may also assume that $n_1$ is large enough so that
\begin{equation}\label{n1b}
|\P_n^s(r_0(R^{\sn})>1)-\N_o^s(r_0(R)>1)|<\vep\quad\text{ for }n\ge n_1.
\end{equation}

Write $O(\vep)$ for any quantity bounded in absolute value by $C\vep$ for some constant $C$ independent of $n$.  Then for $n\ge n_1$ we have
\begin{align}\label{TPa}
\nonumber m(n)\P(r_0(R^{\sss(1)})>\sqrt n)&=\mu_n(r_0(R^{\sn})>1)\\
\nonumber&=\mu_n(r_0(R^{\sn})>1,S^{\sn}>s)+O(\vep)\quad\text{(by (\ref{Sr0b}))}\\
\nonumber&=\P_n^s(r_0(R^{\sn})>1)\mu_n(S^{\sn}>s)+O(\vep)\\
\nonumber&=\N_o^s(r_0(R)>1)\mu_n(S^{\sn}>s)+O(\vep)\ \ \text{(by (\ref{n1b}), (\ref{sprobbnds}))}\\
&=\N_o^s(r_0(R)>1)\frac{s_D}{s}+O(\vep)\quad\text{(by (\ref{Sncvgt}))}.
\end{align}
Next use \eqref{r0tiny} and $\N_o(S>s)=\frac{2}{\gamma s}$ to see that
\begin{align}\label{TPb}
\nonumber\N_o^s(r_0(R)>1)\frac{s_D}{s}&=\frac{\N_o(r_0(R)>1,S>s)}{2/(\gamma s)}\frac{s_D}{s}\\
\nonumber&=\N_o(r_0(R)>1)\frac{\gamma s_D}{2}+O(\vep)\\ 
&=\frac{\sigma_0^2}{2}s_Dv_d(0)+O(\vep),
\end{align}
where Lemma~\ref{lem:radrange} is used in the final equality. Combining \eqref{TPa} with \eqref{TPb}, we see that 
\[\lim_{n\to\infty}m(n)\P(r_0(R^{\sss(1)})>\sqrt n)=\frac{\sigma_0^2}{2}s_Dv_d(0),\]
which gives \eqref{oa2} in Theorem~\ref{thm:one-arm}.  \eqref{oa1} then follows from this and \eqref{cond1'}.\qed
\ARXIV{\begin{REM}\label{Sles}
Theorem~\ref{thm:range} gives weak convergence of $R^{\sn}$ under \break
$\P_n^s$ for all $s>0$ and so the above
proof uses Theorem~\ref{thm:mod_con} to control 
$$\mu_n(r_0(R^{\sn})>1,S^{\sn}\le s)$$ in \eqref{Sr0b}.
As this kind of bound nicely complements the information in Theorem~\ref{thm:range}, we note that the same reasoning shows, more generally, that 
\begin{equation}
\nn
\mu_n(r_0(R^{\sn})>C_{\ref{thm:mod_con}}(s^\alpha+n^{-\alpha}),\ S^{\sn}\le s)\le C_{\ref{thm:mod_con}}(s^\beta+n^{-q})\ \forall\, 0<s<1, \ n\ge 1,
\end{equation}
where $C_{\ref{thm:mod_con}}$, $\alpha$,  $\beta$, and $q$ are as in Theorem~\ref{thm:mod_con}. A corresponding bound for the limiting super-Brownian motion which quantifies \eqref{r0tiny} is
\begin{align}\label{smallSsbm}
\nn\N_0(r_0(R)>(C_{\ref{thm:mod_con}}\vee 2)s^\alpha,\ \mc{S}\le s)\le C_{\ref{smallSsbm}}&s^{-d((1/2)-\alpha)-1} \exp(-2s^{2\alpha-1}),\\
&\forall \, 0<s<1,\ 0<\alpha<1/2.
\end{align}
This is an easy consequence of Theorem~3.3(b) of \cite{DIP89}, its proof and \eqref{sbmppp} (see Theorem~2.3(i) of \cite{HP18} for the $d=1$ case). 
\end{REM}}

\section{On checking conditions \ref{cond:finite_int}-\ref{cond:self-avoiding} and the existence of ancestral paths}\label{sec:onconds}
Here we prove Lemmas~\ref{lem:int_fdd} and \ref{disccondsa} as well as Proposition~\ref{prop:wexist}.

\medskip

\noindent{\it Proof of Lemma~\ref{lem:int_fdd}.} To prove Condition~\ref{cond:finite_int}, it suffices to prove convergence along any sequence.   To simplify notation we 
assume $n\in\N$ and that the branching and diffusion parameters of the limiting super-Brownian motion are both one.  Fix $0\le t_0<t_1$.  Let $p\in\N$ and $u_1,\dots,u_p\in[t_0,t_1]$.  
Let $\phi\ge 0$ be in $C_b(\R^d)$. Then by \eqref{momentbounds},
\begin{equation}\label{L2pbnd}
\P^s_n\Big[\Bigl(\prod_{i=1}^pX^{\sn}_{u_i}(\phi)\Bigr)^2\Big]\le \Vert\phi\Vert_\infty^{2p}C(s,t_1,p)\quad\forall n\in\N.
\end{equation}
It follows easily from the finite-dimensional convergence \eqref{fddconv} and the above $L^2$ bound (e.g. use Skorokhod's representation to get a.s. convergence in \eqref{fddconv})  that 
\begin{equation}\label{momentcvgce}
\lim_{n\to\infty}\P_n^s\Big[\prod_{i=1}^pX^{\sn}_{u_i}(\phi)\Big]=\N_0^s\Big[\prod_{i=1}^p X_{u_i}(\phi)\Big].
\end{equation}
Using Fubini's theorem we have
\begin{align}\label{intmomentcvgce}
\nonumber\lim_{n\to\infty}&\P_n^s\Big[\Bigl(\int_{t_0}^{t_1}X^{\sn}_u(\phi)du\Bigr)^p\Big]\\
\nonumber&=\lim_{n\to\infty}\int_{t_0}^{t_1}\dots\int_{t_0}^{t_1}\P_n^s\Big[\prod_{i=1}^pX^{\sn}_{u_i}(\phi)\Big]du_1\dots du_p\\
\nonumber&=\int_{t_0}^{t_1}\dots\int_{t_0}^{t_1}\N_o^s\Big[\prod_{i=1}^p X_{u_i}(\phi)\Big]du_1\dots du_p\\
&=\N_o^s\Big[\Bigl(\int_{t_0}^{t_1} X_u(\phi)du\Bigr)^p\Big]<\infty,
\end{align}
where  \eqref{momentcvgce}, \eqref{L2pbnd} and dominated convergence are used in the second equality.
If $W$ denotes a standard Brownian motion starting at $o$ under $P_o$, then take $p=1$ in the above to see that
\begin{align}
\nn
\nonumber\lim_{n\to\infty}\P_n^s\big[(\bar X_{t_1}^{\sn}-\bar X_{t_0}^{\sn})(\phi)\big]&=\N_o^s\Big[\int_{t_0}^{t_1}X_u(\phi)du\Big]\\
&\le E_o\Big[\int_{t_0}^{t_1}\phi(W_u)du\Big]/\N_o(\mc{S}>s),\nn%
\end{align}
where the last is because the mean measure of $X_t$ under $\N_o$ is $P_o(B_t\in\cdot)$ (e.g. Theorem~II.7.2(iii) in \cite{Per02}). It now follows easily from the above that sequence of laws  $\{\P^s_n(\bar X^{\sn}_{t_1}-\bar X^{\sn}_{t_0}\in\cdot):n\in\N\}$ on $\mc{M}_F(\R^d)$ are tight.  

To show the limit points are unique, assume 
\begin{equation}\label{wklimpt}
\bar X^{\sss(n_k)}_{t_1}-\bar X^{\sss(n_k)}_{t_0}\cweak\tilde X_{t_0,t_1}\text{ in }\mc{M}_F(\R^d).
\end{equation}
It remains to show that
\begin{equation}\label{lpuniq}
\P(\tilde X_{t_0,t_1}\in\cdot)=\N_o^s\Bigl(\int_{t_0}^{t_1}X_udu\in\cdot\Bigr).
\end{equation}
\eqref{wklimpt} and the convergence, hence boundedness, in \eqref{intmomentcvgce} with $2p$ in place of $p$,
imply that (again one can use Skorokhod's representation theorem)
\begin{equation}
\nn
\E[\tilde X_{t_0,t_1}(\phi)^p]=\lim_{k\to\infty} \E^s_{n_k}\Big[\Big[\int_{t_0}^{t_1}X_u^{(n_k)}(\phi)du\Big]^p\Big].
\end{equation}
This together with \eqref{intmomentcvgce}, implies that 
\begin{equation}\label{momentid}
\E[\tilde X_{t_0,t_1}(\phi)^p]=\N_o^s\Big[\Big(\int_{t_0}^{t_1}X_u(\phi)du\Big)^p\Big]\text{ for all }\phi\in C_b(\R^d)\text{ and }p\in\N.
\end{equation}
So to conclude \eqref{lpuniq} we must show the moment problem is well-posed.
Assume $\Vert\phi\Vert_\infty\le 1$.  Recall that $\P_{\delta_o}$  is the probability law of a SBM started from a unit mass at the origin (with $(\gamma,\sigma_0^2)=(1,1)$).  Then for $0\le \theta<2/t_1^2$,
\begin{align}
\nn
\nonumber\E_{\delta_0}&\Big[\exp\Bigl(\theta\int_{t_0}^{t_1}X_u(\phi)du\Bigr)\Big]\\
&\le \nonumber\E_{\delta_0}\Big[\exp\Bigl(\theta\int_{0}^{t_1}X_u(1)du\Bigr)\Big]\\
\nonumber&\le\E_{\delta_0}\Big[\int_0^{t_1}e^{\theta t_1X_u(1)}du/t_1\Big]\text{ (Jensen applied to the integrand)}\\
&\le\exp(t_1\theta[1-(t_1^2\theta/2)]^{-1})<\infty,\nn
\end{align}
where the last line uses the exponential bound in Lemma~III.3.6 of \cite{Per02}.   
By \eqref{sbmppp}, the left-hand side of the above equals 
\begin{align}
\nn&\exp\Big[\int \Bigl(\exp\Bigl(\theta\int_{t_0}^{t_1}\nu_u(\phi)du\Bigr)-1\Bigr)\,d\N_o(\nu)\Big]\\
&\ \ge\exp\Bigl(\int\Bigl( \exp\Bigl(\theta\int_{t_0}^{t_1}\nu_u(\phi)du\Bigr)-1\Bigl)\1(S>s)\,d\N_o(\nu)\Bigr).\nn
\end{align}
Noting that $\N_o(S>s)<\infty$, the above implies that 
\[\int\exp\Bigl(\theta\int_{t_0}^{t_1}X_u(\phi)du\Bigr)\,d\N_o^s<\infty\text{ for }0\le\theta<2t_1^{-2},\]
which in turn implies that the moment problem for the random variable $\int_{t_0}^{t_1}X_u(\phi)du$ under $\N_o^s$ is well-posed (see, e.g. Theorem 3.3.11 in \cite{Du10}).
Therefore \eqref{momentid} implies that for  non-negative $\phi\in\C_b(\R^d)$ satisfying $\Vert\phi\Vert_\infty\le 1$,
\[\P(\tilde X_{t_0,t_1}(\phi)\in\cdot)=\N_o^s\Bigl(\int_{t_0}^{t_1} X_u(\phi)du\in\cdot\Bigr).\]
This clearly then follows for all non-negative $\phi\in C_b(\R^d)$ by linearity.  This shows
the Laplace functionals of the above two measures are identical and so \eqref{lpuniq} holds
(e.g. by Lemma~II.5.9 of \cite{Per02}) and the proof is complete.
\qed
\medskip

\noindent{\it Proof of Lemma~\ref{disccondsa}.} Let $t,M,\Delta$ be as in Condition~\ref{cond:self-avoiding} and set $\ell=\lfloor t\rfloor\in\Z_+$, $m=\floor{\Delta}\in\N^{\ge 4}$. Assume $x\in \mT_t$, $\exists x'$ s.t. $(t,x)\ara (t+\Delta,x')$ and 
\begin{equation}\label{contcond}
\int_{t+\Delta}^{t+2\Delta}|\{y:(t,x)\ara(s,y)\}|\,ds\le M.
\end{equation}
Clearly $x\in\mT_\ell$, and $(\ell,x)\ara (t,x)\ara(t+\Delta,x')$.  So by \eqref{transitive} and \eqref{convtr} $\exists x''\in\mT_{\ell+m}$ s.t.
\begin{equation}\label{disccond1}
(\ell,x)\ara(\ell+m,x'')\ara(t+\Delta,x').
\end{equation}
Next by \eqref{contcond} and \eqref{winterp},
\begin{align}
\label{disccond2}\nonumber M&\ge\int_{t+\Delta}^{t+2\Delta}|\{y:(t,x)\ara(s,y)\}|\,ds\\
\nonumber&\ge \int_{\ell+m+2}^{\ell+2m}|\{y:(t,x)\ara(s,y)\}|\,ds\\
&=|\{(i,y):(\ell,x)\ara(i,y),\ \ell+m+2\le i\le \ell+2m-1, i\in\N\}|.
\end{align}
From \eqref{disccond1} and \eqref{disccond2} we see that the left-hand side of \eqref{condsaineq} 
(in Condition~\ref{cond:self-avoiding}) is at most
\begin{align*}
\E&\Big[\sum_{x\in\mT_\ell}\1\Big(\exists x''\text{ s.t. }(\ell,x)\ara(\ell+m,x''),\\
&\phantom{\Bigl(\sum_{x\in\mT_\ell}\1\Big(\exists x''}\,|\{(i,y):(\ell,x)\ara(i,y), \ell+m+2\le i\le \ell+2m-1\}|\le M\Big)\Big]\\
&\le c_{\ref{cond:self-avoiding}}\P\Bigl(S^{\sss(1)}>m,\, \sum_{i=m+2}^{2m-1}|\mT_i|\le M\Bigr)\quad\text{(by \eqref{dtsa}})\\
&=c_{\ref{cond:self-avoiding}}\P\Bigl(S^{\sss(1)}>\Delta,\,\int_{m+2}^{2m}|\mT_s|ds\le M\Bigr)\\
&\le c_{\ref{cond:self-avoiding}}\P\Bigl(S^{\sss(1)}>\Delta,\int_{\Delta+2}^{2\Delta-2}|\mT_s|ds\le M\Bigr).
\end{align*}
This proves \eqref{condsaineq}, as required.
\qed

\begin{LEM}\label{lem:stepfn} W.p.$1$ there is a random variable $M\in\N$ and $(\mc{F}_t)$-stopping times $0=\tau_0<\tau_1<\dots<\tau_M=S^{\sss(1)}$ such that 
\begin{align*}\mT_t=\begin{cases} \mT_{\tau_{i-1}}&\text{on }[\tau_{i-1},\tau_i)\text{ for }1\le i\le M\\
\varnothing&\text{on }[\tau_M,\infty)=[S^{\sss(1)},\infty).\end{cases}
\end{align*}
\end{LEM}
\begin{proof} Choose $\omega$ so that $S^{\sss(1)}<\infty$ and (AR)(i)-(iii) hold.  In the discrete case the 
result is clear. Just set $M=S^{\sss(1)}$ and $\tau_i=i$, and recall \eqref{death}.

Consider next $I=[0,\infty)$.  Recall (\eqref{cond1}, \eqref{finiterange}) that $(\mT_t,t\ge0)$ is a cadlag $\mc{K}$-valued process taking values in the finite subsets of $\Z^d$, $P_F$. It follows that $d_0(\mT_{t-},\mT_t)\ge 1$ whenever $\mT_{t-}\neq\mT_t$ and $\mT$ is constant between jumps (distinct points in $P_F$ are distance $1$ apart). Therefore the jump times cannot accumulate and so can be listed as an increasing sequence of stopping times $0<\tau_1<\tau_2<\dots$, where $\tau_m$ is the $m$th jump time and $\tau_m=\infty$ if there are fewer than $m$ jumps. By \eqref{death} $\mT_t=\varnothing$ for all $t\ge S^{\sss(1)}<\infty$ and as $S^{\sss(1)}$ is a jump time, clearly $S^{\sss(1)}$ is the last jump time.  
Moreover the above implies that the number of jumps $M$ is an $\N$-valued random variable and
$\tau_M=S^{\sss(1)}$.  The proof is complete.
\end{proof}

\medskip

\noindent  {\it {Proof of Proposition~\ref{prop:wexist}.}}  Consider first $I=\Z_+$.
 Choose $\omega$ s.t.  (AR)(i)-(iii) hold. Let $(t,x)\in\vec{\mT}$. If $t<1$, then $x=o$, and $w\equiv o$ is the required ancestral path, so assume $j=\floor{t}\in\N$.  We have $(0,o)\ara(j,x)$ and so applying \eqref{convtr} $j-1$ times we can find $y_0,y_1,\dots,y_j$ such that $(i-1,y_{i-1})\ara(i,y_i)$ for all $1\le i\le j$, $y_0=o$ and $y_j=x$.  Now define
 \begin{align*}
 w_s=\begin{cases}y_{i-1}&\text{ if }s\in[i-1,i)\text{ for }1\le i\le j\\
 x&\text{ if }s\ge j.
 \end{cases}\end{align*}
Then clearly $w$ is an ancestral path to $(t,x)$.  
 
 If $I=[0,\infty)$ one proceeds as above, choosing $\omega$ so that the conclusion of the previous lemma also holds, and now working with $\{\tau_0,\tau_1,\dots,\tau_M\}\cap[0,t]$ in place of $\{0,1,\dots,j\}$.\qed

\section{Verifying the conditions for the voter model}
\label{sec:voter}
Here we verify that \eqref{cond1}, (AR) and Conditions \ref{cond:surv}-\ref{cond:self-avoiding} hold for the voter model in dimensions $d\ge 2$ (and hence prove Theorem \ref{thm:voter}). 
We first briefly describe the graphical construction of the voter model.  This is a standard construction
 so we refer the reader to Section~2 of \cite{BCLG01} for most justifications and further details (or alternatively Example 3.2 of \cite{Du95}). Let $\{\Lambda(x,y):x,y\in\Z^d\}$ be a collection of independent Poisson point processes (ppp's) on $\R_+$ where $\Lambda(x,y)$ has intensity $D(x-y)ds$. The points in $\Lambda(x,y)([t_1,t_2])$ are the times in $[t_1,t_2]$ when a voter at $y$ imposes its opinion at site $x$.  At such times an arrow
is drawn from $y$ to $x$.  Let $\mc{F}^0_t=\sigma(\{\Lambda(x,y)([0,s]):s\le t, x,y\in\Z^d\})$ and $\mc{F}_t=\mc{F}^0_{t+}$.  We assume below that the points in these point processes are all mutually disjoint and strictly positive, thus omitting a set of measure $0$.  

Recall that the voter model $(\xi_t)_{t\ge 0}$ is a $\{0,1\}^{\Z^d}$-valued Feller process. For each $t\ge 0$ and $x\in\Z^d$ we use the above ppp's to trace the opinion $\xi_t(x)$ back to its source at time $0$ by defining a ``dual" random walk $(W^{t,x}_s,0\le s\le t)$.  For this, set $s_0=0$, $W^{t,x}_0=x=y_0$ and from here we assume $t>0$.  
Let $t-s_1$ be the largest time in $(0,t]$ when there is an arrow from some $y_1$ to $y_0$ if such a time exists (so $s_1=0$ is possible if there is an arrow at $t$).  If no such time exists, set $s_1=t$ and $n=0$.  In general assume we are given
$0<t-s_k<t-s_{k-1}<\dots<t-s_1\le t$ and points $y_0,\dots, y_k$ so that there is an arrow from $
y_i$ to $y_{i-1}$ at time $t-s_i$ for $i=1,\dots,k$. Let $t-s_{k+1}\in(0,t-s_k)$ be the largest time
in this interval when there is an arrow from some $y_{k+1}$ to $y_k$. If no such time exists set $s_{k+1}=t$ and $n=k$. It is easy to see this process stops after a finite number of steps for all $t>0$, $x\in\Z^d$ a.s. (for fixed $(t,x)$ it is clear as the arrows are arising with rate $1$, and if it is finite
for all rational $t>0$ and  $x\in\Z^d$, it will be finite for all $(t,x)$ because for some rational $q>t$ there
will be no arrows into $x$ in $(t,q]$). Note that $n\in\Z_+$ gives the number of steps in the walk and $s_{n+1}=t$.  Define $W^{t,x}_s$ for $0<s\le t$ by
\begin{equation}
\nn
W^{t,x}_s=y_k\text{ if }s_k<s\le s_{k+1}\text{ for }0\le k\le n.
\end{equation}
Then
\begin{align}\label{Corw}
&\{W^{t,x}_s:s\le t\}_{x\in\Z^d}\text{ is a system of left-continuous, right-limited,} \\
\nonumber&\qquad\text{rate one coalescing random walks with step distribution $D$.}
\end{align}
The above definition easily implies
\begin{equation}\label{WFt}
W^{x,t} \text{ is $\mc{F}_t-$measurable,}
\end{equation}
\begin{align}\label{Windce}
&\forall\ 0\le s\le t\ \ W^{t,x}_{t-s}\text{ is }\sigma\Bigl(\Lambda(x',y')([s,t']):t'\ge s,\ x',y'\in\Z^d\Bigr)-\text{measurable}\\
\nonumber&\qquad\text{and, in particular, is independent of }\mc{F}_s,
\end{align}
\begin{equation}
\nn
W^{t',x'}_{t'-s}=W^{t,x}_{t-s}\text{ for some }s\in[0,t\wedge t')\Rightarrow W^{t',x'}_{t'-u}=W^{t,x}_{t-u}\ \ \forall u\in[s,t\wedge t'],
\end{equation}
and
\begin{equation}\label{coalW2}
W^{t,x}_{t-u}=W_{s-u}^{s,W_{t-s}^{t,x}}\quad\text{for }0\le u\le s\le t,\ x\in\Z^d.
\end{equation}

If $\xi_0\in\{0,1\}^{\Z^d}$, then 
\begin{equation}\label{voterdef}
\xi_t(x)=\xi_0(W^{t,x}_t)\quad\text{for }x\in\Z^d,t\ge 0
\end{equation}
defines an $\mc{F}_t$-adapted voter model starting at $\xi_0$ with cadlag paths in $\{0,1\}^{\Z^d}$ and
law $P_{\xi_0}$ on $\mc{D}([0,\infty),\{0,1\}^{\Z^d})$.  
Right-continuity follows from the fact that we include arrows at $t$ in our definition of $W^{t,x}$ so
for some $\delta>0$, there are no arrows to $x$ in $(t,t+\delta]$ and so $W^{t+\delta,x}_{t+\delta}=W^{t,x}_t$.  Note that \eqref{coalW2} with $u=0$ and \eqref{voterdef} imply
\begin{equation}\label{genvoterdef}
\xi_t(x)=\xi_s(W^{t,x}_{t-s})\quad\forall\,0\le s\le t,\ x\in\Z^d.
\end{equation}

If $\xi_0=\1_{\{o\}}$ we write $\xi^o_t(x)=\1(W^{t,x}_t=o)$ and define
\begin{equation}\label{Tdef}
\mT_t=\{x\in\Z^d:\xi^o_t(x)=1\}=\{x:W^{t,x}_t=o\}.
\end{equation}

\begin{LEM}\label{markovmart}
(a) If $\xi_0\in\{0,1\}^{\Z^d}$ and $A$ is a Borel subset of $\{0,1\}^{\Z^d}$, then
\begin{equation}\label{VMP}
\P(\xi_t\in A|\mc{F}_s)=P_{\xi_s}(\xi_{t-s}\in A) \text{ a.s.  for all }0\le s\le t.
\end{equation}

\noindent(b) $t\to |\mc{T}_t|$ is a cadlag $(\mc{F}_t)$-martingale s.t. $S^{\sss(1)}<\infty$ a.s. and in particular
$\sup_{t\ge 0}|\mc{T}_t|<\infty$ a.s.

\noindent(c) $(\mc{T}_t)_{t\ge 0}$ satisfies \eqref{cond1}.
\end{LEM}
\begin{proof}
(a) Use \eqref{genvoterdef} to see that the left-hand side of \eqref{VMP} is
\[\P(\xi_s(W^{t,\cdot}_{t-s})\in A|\mc{F}_s)=P_{\xi_s}(\xi_{t-s}\in A)\ \text{ a.s.}.\]
In the above equality we used the fact that \[\Lambda^s(x,y)([0,u])=\Lambda(x,y)([s,s+u])\] defines
a collection of ppp's equal in law to $\{\Lambda(x,y):x,y\in\Z^d\}$ and independent of $\mc{F}_s$, which
implies that $\{W^{t,x}_{t-s}:x\in\Z^d\}$ are equal in law to $\{W^{t-s,x}_{t-s}:x\in\Z^d\}$ and are independent of $\mc{F}_s$. We also used the fact that $\xi_s$ is $\mc{F}_s$-measurable.

\noindent(b) See Proposition~V.4.1 of \cite{Li85} and its proof for this, except for the martingale
property with respect to the larger filtration $(\mc{F}_t)$. This, however, then follows immediately from (a) and Proposition V.4.1(a) of \cite{Li85}.

\noindent(c) The fact that $\xi_t$ is cadlag in $\{0,1\}^{\Z^d}$ and $|\mT_t|<\infty$ for all $t\ge 0$ a.s. (by (b)) 
shows $t\to \mT_t$ is cadlag in $\mc{K}$.  This establishes \eqref{cond1}.
\end{proof}

Define $(s,y)\ara(t,x)$ iff $s\le t$, $x\in\mT_t$ and $y=W^{t,x}_{t-s}$.  It follows that
\begin{equation}\label{votere}
e_{s,t}(y,x)=\1(W^{t,x}_{(t-s)^+}=y,x\in\mT_t)\text{ for all }s,t\ge 0,x,y\in\Z^d,
\end{equation}
where $(t-s)^+$ is the positive part of $t-s$.
\begin{LEM}
\label{lem:vmara}
$\ara$ defines an ancestral relation for the voter model.
\end{LEM}
\begin{proof} Starting with AR(i), note that \eqref{eqara} is immediate. Assume $(s,y)\ara(t,x)$. By definition $x\in\mT_t$ and $s\le t$.  \eqref{genvoterdef} and \eqref{Tdef} imply that $\xi_s(W^{t,x}_{t-s})=\xi_t(x)=1$ and so $y=W^{t,x}_{t-s}\in\mT_s$, proving \eqref{arbasica}. \eqref{arbasic0} follows from \eqref{Tdef}.

Turning to (ii), \eqref{transitive} is a consequence of \eqref{coalW2}. Assume now that $(s_1,y_1)\ara(s_3,y_3)$.  Then $y_1=W^{s_3,y_3}_{s_3-s_1}$ and if $y_2=W^{s_3,y_3}_{s_3-s_2}$, then $(s_2,y_2)\ara(s_3,y_3)$ by definition. By \eqref{coalW2} we have $W^{s_2,y_2}_{s_2-s_1}=W^{s_3,y_3}_{s_3-s_1}=y_1$ and so $(s_1,y_1)\ara(s_2,y_2)$. This gives (ii).

We will use Remark~\ref{APremark}(2) to verify (iii). The fact that $s\to e_s(t,x)$ is cadlag on $[0,\infty)$ is immediate from \eqref{votere} and the fact that $s\to W_s^{t,x}\in\Z^d$ is left-continuous with right limits on $[0,t]$ (by \eqref{Corw}).  There is a $\delta>0$ such that there is no arrow towards $x$ in $(t,t+\delta]$.
Let $r\in(t,t+\delta]$. Then by definition
\begin{equation}\label{noearljsa}
W^{r,x}_s=x\quad\forall s\in[0,r-t],
\end{equation}
which by \eqref{genvoterdef} implies
\[\xi_r(x)=\xi_t(W^{r,x}_{r-t})=\xi_t(x),\]
and therefore 
\begin{equation}\label{xTt}
x\in \mT_t\ \text{ iff }\ x\in\mT_r.
\end{equation}
Next, use \eqref{coalW2} with $(u,s,t)$ replaced by $(s,t,r)$ to see that,
\begin{equation}\label{Wequal}W_{r-s}^{r,x}=W_{t-s}^{t,W^{r,x}_{r-t}}=W_{t-s}^{t,x}\ \text{ for all $s\le t$},
\end{equation}
where \eqref{noearljsa} is used in the last equality. We also have 
\begin{equation}\label{Wequalb}
W^{r,x}_{(r-s)^+}=x=W^{t,x}_{(t-s)^+}\quad\text{for all }s\ge t,
\end{equation}
where we use \eqref{noearljsa} for the first equality when $r\ge s\ge t$. Now use \eqref{xTt}, \eqref{Wequal} and \eqref{Wequalb} in \eqref{votere} to conclude that $\hat e_t(y,x)=\hat e_r(y,x)$ for all $r\in[t,t+\delta]$. 
This proves the first condition in Remark~\ref{APremark}(2).  For the second condition, \eqref{estepl}, if $x\in\mT_{t-}$ choose $\delta>0$ such that there are no arrows to $x$ in $[t-\delta,t)$ and proceed in a similar manner. This completes the proof of AR(iii).

If $s< t$, \eqref{votere} shows that $e_{s,t}(y,x)$ is $\mc{F}_t$-measurable by \eqref{WFt} and the $\mc{F}_t$-adaptedness of ${\bT}$. This gives (AR)(iv) and the proof is complete.
\end{proof}

 For $t\ge 0$ and $x\in\Z^d$ we define our candidate for an  ancestral path to $(t,x)\in\vec{\mT}$ by \begin{align}\label{voterwdef} w_s(t,x)= W^{t,x}_{(t-s)^+}.
\end{align}

\begin{LEM}
\label{lem:vw_tree} For any $(t,x)\in\vec{\mT}$, 
$w(t,x)$ is the unique ancestral path to $(t,x)$, and therefore $\mc{W}=\{w(t,x):(t,x)\in\vec{\mT}\}$.\end{LEM}
\begin{proof} 
Assume that $(t,x)\in\vec{\bf \mT}$.  Then $s\mapsto w_s(t,x)$ is cadlag by definition and the fact that $W_s^{t,x}$ is left-continuous with right limits in $s$.
 \eqref{genvoterdef} implies that if $0\le s\le t$, then $1=\xi_t(x)=\xi_s(w_s(t,x))$ and so 
\begin{equation}\label{inT}w_s(t,x)\in\mT_s\text{ for all }s\le t. 
\end{equation}
Let $0\le s\le s'\le t$. Then, using \eqref{inT}, we see that $(s,W^{t,x}_{t-s})\ara(s',W^{t,x}_{t-s'})$ iff $W_{t-s}^{t,x}=W_{s'-s}^{s', W^{t,x}_{t-s'}}$, which holds by \eqref{coalW2}.
As $w_s(t,x)=W^{t,x}_0=x$ for all $s\ge t$, we see that $w(t,x)$ is an ancestral path to $(t,x)$.  

Turning to uniqueness, let $\tilde w_s(t,x)$ be any ancestral path to $(t,x)$. Then $(0,o)\ara(s,\tilde w_s(t,x)$ implies $\tilde w_s(t,x)=W^{t,x}_{t-s}$, and so $\tilde w(t,x)=w(t,x)$ is unique.  The last assertion is then immediate.
\end{proof}
Before proving Theorem~\ref{thm:voter} we note that the above definition of $w(t,x)$ and part (v1) of the Theorem give
a uniform modulus of continuity for the rescaled dual coalescing random walks connecting one-valued sites in the voter model conditioned on longterm survival.
\begin{COR}\label{cor:coaldualmod}
Assume $\{W^{t,x}_s:0\le s\le t, x\in\Z^d\}$ ($d\ge 2$) is the coalescing dual of a voter model $(\mT_t)_{t\ge 0}$ starting with a single one at the origin, with bounded range kernel $D$ and survival time $S^{\sss(1)}$.  Let $\alpha\in(0,1/2)$. There is a constant $C_{\ref{cor:coaldualmod}}$ and for all $n\ge 1$ a random variable $\delta_n\in[0,1]$ so that 
\begin{equation}\label{deltabndv}
\P(\delta_n\le \rho|S^{\sss(1)}>nt^*)\le C_{\ref{cor:coaldualmod}}(t^*\vee 1)[\rho+n^{-1}],\ \forall \rho\in[0,1),\ t^*>0,
\end{equation}
and if $(t,x)\in\vec{\mT}$, $0\le s_1<s_2\le t$, and $|s_2-s_1|\le n\delta_n$,
\begin{equation}\label{modconvo}
|W^{t, x}_{s_1}-W^{t, x}_{s_2}|\le C_{\ref{cor:coaldualmod}}n^{(1/2)-\alpha}[|s_2-s_1|^\alpha+1].
\end{equation}
\end{COR}
\begin{proof} Let $\delta_n$ be as in Theorem~\ref{thm:voter}(v1)  (see Definition \ref{def:modcon}). Then for $n\ge 1$, $t^*>0$, and $\rho\in[0,1)$, that Theorem gives
\begin{align*} \P(\delta_n\le \rho|S^{\sss(1)}>nt^*)&\le m(n) \P(\delta_n\le \rho)/(m(n)\P(S^{\sss(1)}>nt^*))\\
&\le C_{\ref{thm:mod_con}}[\rho+n^{-1}]\underline s_D^{-1}c_{\ref{mdef}}(t^*\vee 1),
\end{align*}
where \eqref{sprobbnds} is used in the last line.  This gives \eqref{deltabndv}, and \eqref{modconvo} is then immediate from Corollary~\ref{cor:wmod}, \eqref{voterwdef} and Lemma~\ref{lem:vw_tree}.
\end{proof}

\noindent{\it Proof of Theorem~\ref{thm:voter}.} Parts (v1), (v2) and (v3) will follow from Theorems~\ref{thm:mod_con}, \ref{thm:range} and \ref{thm:one-arm}, respectively, once we verify Conditions~\ref{cond:surv}-\ref{cond:self-avoiding} for the parameter values given in Theorem~\ref{thm:voter}.  Here we need to recall that $s_D=\beta_d^{-1}$ for the voter model (Proposition~\ref{cond1check}),  and carry out a bit of arithmetic (especially for (v3)).  We have already noted that Conditions~\ref{cond:surv} and \ref{cond:finite_int} follow from Propositions~\ref{cond1check}(b) and \ref{prop:wkcvgcevoterLT}, respectively.  Condition~\ref{cond:L1bound} follows immediately from the martingale problem of $|\mT_t|$ (Lemma~\ref{markovmart}(b)). So it remains to check Conditions~\ref{cond:self-repel}, \ref{cond:6moment}, \ref{cond:smallinc} and \ref{cond:self-avoiding} for the voter model.

\medskip
\noindent {\bf Condition~\ref{cond:self-repel}.} On $\{y\in\mT_s\}$ we have 
\begin{align*}
\P(\exists z\text{ s.t. }(s,y)\ara(s+t,z)|\mc{F}_s)&=\P(\exists z\text{ s.t. }W^{s+t,z}_t=y|\mc{F}_s)\\
&=\P(\exists z\text{ s.t. }W_t^{s+t,z}=y)\quad(\text{by }\eqref{Windce})\\
&=\P(\exists z\text{ s.t. }W_s^{s,z}=0),
\end{align*}
where in the last line we used translation invariance in both space and time of the system of 
Poisson point processes $\{\Lambda(x,y)\}$. More specifically we use the fact that
$\{\Lambda(x'-y,y'-y)([t,t+u]):x',y'\in\Z^d,u\ge0\}$ has the same law as $\{\Lambda(x',y')([0,u]):x',y'\in\Z^d,u\ge0\}$. Recalling \eqref{Tdef} we see that the right-hand side of the above equals
\[\P(\mT_s\text{ is non-empty})\le\frac{\overline s_D}{m(s)},\]
by \eqref{survivbnds} (which applies because Condition~\ref{cond:surv} holds). 

\medskip
\noindent {\bf Condition~\ref{cond:self-avoiding}.} On $\{x\in\mT_t\}$ we can argue as above to see that
\begin{align*}
\P\Bigl(&\exists x'\text{ s.t. }(t,x)\ara(t+\Delta,x'),\ \int_{t+\Delta}^{t+2\Delta}|\{y:(t,x)\ara(s,y)\}|ds\le M\Bigr|\mc{F}_t\Bigr)\\
&=\P\Bigl(\exists x'\text{ s.t. }W_\Delta^{t+\Delta,x'}=x,\ \int_{t+\Delta}^{t+2\Delta}|\{y:W^{s,y}_{s-t}=x\}|ds\le M\Bigr|\mc{F}_t\Bigr)\\
&=\P\Bigl(\exists x'\text{ s.t. }W_\Delta^{t+\Delta,x'}=x,\ \int_{t+\Delta}^{t+2\Delta}|\{y:W^{s,y}_{s-t}=x\}|ds\le M\Bigr)\ \text{(by \eqref{Windce})}\\
&=\P\Bigl(\exists x'\text{ s.t. }W_\Delta^{t+\Delta,x'}=x,\ \int_{\Delta}^{2\Delta}|\{y:W^{s'+t,y}_{s'}=x\}|ds'\le M\Bigr)\ (s'=s-t)\\
&=\P\Bigl(\exists x'\text{ s.t. }W_\Delta^{\Delta,x'}=o,\ \int_\Delta^{2\Delta}|\{y:W^{s,y}_{s}=o\}|ds\le M\Bigr),
\end{align*}
where in the last line we use the the translation invariance in space
and time of the system of ppp's as above.  Now use \eqref{Tdef} to see that the above equals
\[\P\Big(S^{\sss(1)}>\Delta,\ \int_\Delta^{2\Delta}|\mT_s|\,ds\le M\Big).\]
Using this equality, the left-hand side of \eqref{condsaineq} (in Condition~\ref{cond:self-avoiding}) is equal to
\begin{align*}
\E\Big[&\sum_{x\in\mT_t}\P\Bigl(S^{\sss(1)}>\Delta,\ \int_\Delta^{2\Delta}|\mT_s|\,ds\le M\Bigr)\Big]\\
&=\E[|\mT_t|]\P\Bigl(S^{\sss(1)}>\Delta,\ \int_\Delta^{2\Delta}|\mT_s|\,ds\le M\Bigr).
\end{align*}
Recall that $\E[|\mT_t|]=1$ by Lemma~\ref{markovmart}(b), and so if $\Delta\ge 4$ the above is trivially bounded above by the right-hand side of \eqref{condsaineq} with $c_{\ref{cond:self-avoiding}}=1$, and
so Condition~\ref{cond:self-avoiding} is established.

\medskip
\noindent{\bf Condition~\ref{cond:6moment}.} Recall that for $0\le s\le t$, $(t-s,y)\ara(t,x)$ iff $x\in \mT_t$ and $y=W^{t,x}_{s}$.  Therefore by translation
invariance of $W^{t,x}_s$, we have for any $p>4$,
\begin{align}
\nn\E\Big[&\sum_{x\in\mT_t}\sum_{y\in\mT_{t-s}}\1((t-s,y)\ara(t,x))|x-y|^p\Big]\\
\nn&=\E\Big[\sum_{x\in\mT_t}\sum_{y\in\Z^d}\1(W_s^{t,x}=y)|x-W^{t,x}_s|^p\Big]\\
\nn&=\E\Big[\sum_{x\in\Z^d}\1(W^{t,x}_t=o)|x-W^{t,x}_s|^p\Big]\ \text{\quad(by \eqref{Tdef})}\\
\nn&=\E\Big[\sum_{x\in\Z^d}\1(W^{t,o}_t=-x)|W^{t,o}_s|^p\Big]\\
\label{C4vm}&=\E[|W^{t,o}_s|^p].
\end{align}
Recall (see \eqref{Corw}) $s\mapsto W^{t,o}_s$ is a rate one continuous time rw with step distribution $D$ and 
so has steps bounded in Euclidean norm by $L$.  If $S_n=\sum_{i=1}^n Z_i$ denotes the corresponding discrete time
rw and $N_s$ is an independent rate one Poisson process, then we can use Burkholder's predictable
square function inequality (Theorem 21.1 in \cite{B}) to see that
\begin{align*}
\E[|W^{t,o}_s|^p]&=\E[|S_{N_s}|^p]\\
&\le c_p\E[N_s^{p/2}+\max_{i\le N_s}|Z_i|^p]\\
&\le c_p\E\Big[N_s^{p/2}+N_s\Bigr[\sum_x |x|^pD(x)\Bigr]\Big]\\
&\le c(p,L)(s\vee 1)^{p/2},
\end{align*}
and we arrive at Condition~\ref{cond:6moment} for any $p>4$.

\medskip
\noindent{\bf Condition~\ref{cond:smallinc}.}
To verify Condition~\ref{cond:smallinc} for all $\kappa>4$ we will dominate the range of the voter model by a pure birth process. The following result is standard (e.g. see Theorem~11 in Sec. 6.11 of \cite{GS} and use a stopping time argument to add values of $s=e^\lambda>1$ to those considered there).

\begin{LEM}\label{pbirth}
Let $M_t$ denote a rate one pure birth process with $M_0=1$.  Then for all $t,\lambda\ge 0$
satisfying $\lambda<-\ln(1-e^{-t})$,
\begin{equation}
\nn
\E[e^{\lambda M_t}]=(1+e^{t-\lambda}-e^t)^{-1},
\end{equation}
and so if $\lambda_{\ref{pbirth}}=-\ln(1-e^{-2})/2$, there is a $C_{\ref{pbirth}}$ such that
\begin{equation}\label{bptp}
 \P(M_2\ge N)\le C_{\ref{pbirth}} e^{-\lambda_{\ref{pbirth}}N}\quad \forall N>0.
\end{equation}
\end{LEM}

\noindent To verify Condition~\ref{cond:smallinc} we couple the voter model $\xi_t(x)=\1(x\in\mT_t)$ with a rate one branching random walk
$Z_t(x)\in\Z_+,\ t\ge 0, x\in\Z^d$, so that $Z_0(x)=\xi_0(x)=\1(x=o)$ and $\xi_t(x)\le Z_t(x)$ for all $t, x$. This is standard so we only sketch the construction. 

We extend the system of Poisson point processes used to construct the dual coalescing rw's $\{W^{t,x}_s\}$ by
considering an i.i.d.~system of such processes $\Lambda_i,i\ge 1$, where $\Lambda_1=\Lambda$.
Then every time $\Lambda_i(x,y)$ jumps at time $t$, and $Z_t(y)\ge i$, particle $i$ at $y$ will
produce an offspring at $x$.  In this way one can easily check that $Z_t$ is a rate one branching random walk with offspring law $D$. Moreover since $M_t=\sum_xZ_t(x)$ is a rate one pure birth process, and so is finite for all times, we can
order the jumps of $Z$ and $\xi$ as $0<T_1<T_2<\dots$ (recall that the range of the voter model is finite a.s.). It is then easy to induct on $n$ to check that $\xi_{T_n}\le Z_{T_n}$ (coordinatewise). (Here one really only needs check times at which a new one appears in $\xi_{T_n}$ at location $x$.) Since $Z_t$ is monotone increasing, this implies that 
\begin{equation}\label{pbpbnd}|R^{\sss(1)}_t|\le M_t\quad\text{for all }t\ge 0,
\end{equation}
where we recall that $R^{\sss(1)}_t$ is the range of the voter model up to time $t$.

Recall that $w_s(t,x)=W^{t,x}_{(t-s)^+}$ is the unique ancestral path to $(t,x)\in \vec{\mT}$.  The independence in \eqref{Windce} and translation invariance of the system of Poisson point processes in 
the graphical construction of $\xi$, imply that for $s\ge 0, y\in\Z^d$ fixed, and on $\{y\in\mT_s\}$,
\begin{align*}
\P&(\exists\ (t,x)\text{ s.t. } (s,y)\ara(t,x),\ t\in[s,s+2],\ |y-x|\ge N|\mc{F}_s)\\
&=\P(\exists\ (t,x)\text{ s.t. } (0,o)\ara(t,x),\ t\in[0,2],\ |x|\ge N)\\
&\le \P(|\{w_s(t,x):s\in [0,t]\}|\ge N/(\sqrt d L)\text{ for some }t\in[0,2]\text{ and }x\in\mT_t)\\
&\le\P(|R^{\sss(1)}_2|\ge N/(\sqrt d L)).
\end{align*}
The first inequality holds since $s\to w_s(t,x)=W^{t,x}_{t-s}$ is a step function from $o$ to $x$ taking steps
of (Euclidean) length at most $\sqrt d L$, and the second holds since for $x\in\mT_t$, the range of $w_\cdot(t,x)$ is in $R^{\sss(1)}_t\subset R^{\sss(1)}_2$ for $t\le 2$.

Now use \eqref{bptp} and \eqref{pbpbnd} to see the above upper bound is at most
\[\P(M_2\ge N/(\sqrt d L))\le C_{\ref{pbirth}}\exp{\Bigl(-\frac{\lambda_{\ref{pbirth}}}{\sqrt d L}N}\Bigr).\]
This implies Condition~\ref{cond:smallinc} for each $\kappa>4$, and so completes the proof of Theorem~\ref{thm:voter}.\qed

\begin{REM} A very similar argument would confirm Condition~\ref{cond:smallinc} for the critical bounded range contact process $\zeta_t(x)$.
Again a standard argument will couple $\zeta$ with a dominating constant rate branching random walk--one 
ignores deaths and allows multiple occupancies. The rest of the reasoning will be the same
once the infection relation $(s,x)\ara(t,y)$ is defined for $t>s$. We leave the details for the interested reader.
\end{REM}

\section{Verifying the conditions for lattice trees}
\label{sec:LT}

Recall that we defined $\ara$ by 
\[(k,y)\ara(m,x) \quad \iff \quad x\in \mT_m, \, 0\le k\le m, \text{ and }w_k(m,x)=y,\]
where $w(m,x)=(w_k(m,x))_{k\le m}$ is the unique path in the tree $\bT$ from $o$ to $x$.  Recall also that we had verified (AR) except for (AR)(iv) after Definition~\ref{def:modcon}.
  Given $T\in\T_L(o)=:\T$, we let $T_{\le n}$ denote the subtree consisting of vertices in $\cup_{m\le n}T_m$ and all the bonds in $E(T)$ between these vertices.  Clearly $T_{\le n}$ is connected because for any $n'\le n$ and $x\in T_{n'}$, $(w_m(n',x))_{m\le n'}$ is a path in $T_{\le n}$ from $o$ to $x$. It follows that $T_{\le n}$ is a tree and clearly the set $\T_{\le n}=\{T_{\le n}:T\in\T\}\subset\T$ is a  finite set of trees.  
It also follows that for any $x\in\Z^d$ 
\begin{equation}
 \1(x\in T_n)\1(w(n,x)\in A)\text{ is a function of }T_{\le n} \text{ for any} A\subset \T_{\le n}.\label{wFn}
\end{equation}
Choosing a random tree $\bT$ according to $\P$, we see that $\mT_{\le n}$ is a random tree.  We define
\begin{equation}\label{LTFn}
\mc{F}_n=\{\{\mc{T}_{\le n}\in A\}: A\text{ is a subset of }\T_{\le n}\},\ n\in\Z_+,
\end{equation}
that is, $\mc{F}_n$ is just the $\sigma$-field generated by $\mT_{\le n}$. Since $\mT_{\le n}$ is a function of $\mT_{\le n+1}$, $(\mc{F}_n)_{n\in\Z_+}$ is a filtration and clearly 
\begin{equation}
\nn
\mT_n=V(\mT_{\le n})\setminus V(\mT_{\le n-1})\text{ is }\mc{F}_n\text{-measurable}.
\end{equation}

We can now verify (AR)(iv).  Let $m,n\in \Z_+$ and $x,y\in \Z^d$.  If $m< n$ then $e_{m,n}(y,x)=\1((m,y)\ara(n,x))=\1(x \in \mc{T}_n,w_m(n,x)=y)$, which is $\mc{F}_n$-measurable by \eqref{wFn}.  If $m\ge n$ then $e_{m,n}(y,x)=\1(x=y \in \mT_n)$, which is also $\mc{F}_n$-measurable by\eqref{LTFn}.  This verifies (AR)(iv) as required.

For $x \in \mT_n$ define the extended path $w'(n,x)$ by 
\[w'_m(n,x)=w_m(n,x)\1(m<n)+x\1(m\ge  n).\]  
It is then immediate from the definition of $w(n,x)$ that  $w_{m_1}(n,x)\in \mT_{m_1}$ for $m_1\le n$ and that $w_{m_0}(m_1,w_{m_1}(n,x))=w_{m_0}(n,x)$ for $m_0\le m_1\le n$.  Thus $w'(n,x)$ is an ancestral path to $(n,x)\in \vec{\mc{T}}$. Moreover it is easy to see $w'(n,x)$ is the only ancestral path to $(n,x)$ and hence 
\begin{equation}\label{LTaps}
\mc{W}:=\{w'(n,x):(n,x)\in \vec{\mc{T}}\}
\end{equation}
 is the system of ancestral paths for $(\bT,\ara)$.
Before verifying Conditions \ref{cond:surv}-\ref{cond:self-avoiding} with parameters as in Theorem~\ref{thm:LT}, and hence verifying the conclusion of Corollary~\ref{cor:dwmod}, we can use the above characterization of $\mc{W}$ in \eqref{LTaps} to give an explicit interpretation of this corollary. It is a large scale modulus of continuity for $w_k(m,x), k\le m$ conditional on longterm survival of the tree.
\begin{COR}\label{cor:LTwmod} Let $\bs{\mT}$ be the critical lattice tree with $d>8$, $L$ sufficiently large so that the hypotheses of Theorem~\ref{thm:LT} hold, and survival time $S^{\sss(1)}$.  Assume $\alpha\in(0,1/6)$, $\beta\in(0,1]$ satisfy $1-2\alpha>\frac{2}{3}(1+\beta)$. Then there is a constant $C_{\ref{cor:LTwmod}}$, and for any $n\ge 1$ a random variable $\delta_n\in(0,1]$ so that 
\begin{equation}\label{deltabndlt}
\P(\delta_n\le \rho|S^{\sss(1)}>nt^*)\le C_{\ref{cor:LTwmod}}(t^*\vee 1)\rho^\beta\ \ \forall\rho\in[0,1),\ t^*>0,
\end{equation}
and if $(m,x)\in\vec{\mT}$, $k_1,k_2\in \Z_+$, $k_i\le m$, and $|k_2-k_1|\le n\delta_n$, then 
\begin{equation*}
|w_{k_2}(m,x)-w_{k_1}(m,x)|\le C_{\ref{cor:LTwmod}}|k_2-k_1|^\alpha n^{(1/2)-\alpha}.
\end{equation*}
\end{COR}
\begin{proof} This follows immediately from Corollary~\ref{cor:dwmod}, \eqref{LTaps} and a short calculation to derive \eqref{deltabndlt}.  The latter is similar to the derivation of \eqref{deltabndv} in the proof of Corollary~\ref{cor:coaldualmod}.
\end{proof}

In the remainder of this section we verify Conditions \ref{cond:surv}-\ref{cond:self-avoiding} for critical sufficiently spread-out lattice trees in dimensions $d>8$.   Condition \ref{cond:6moment} is verified subject to a bound on the 6th moment of the two-point function:
\begin{LEM}
\label{lem:LT6}
For $d>8$ and $L$ sufficiently large there exists $C>0$ such that for all $n\in \Z_+$,
\begin{align}
\nn
\sum_x |x|^6 \P(x \in \mc{T}_n)\le CL^6 n^3.
\end{align}
\end{LEM}
Section \ref{sec:tree6} is devoted to the proof of Lemma \ref{lem:LT6}.

\medskip

\noindent{\bf Condition \ref{cond:surv}:} This is immediate from \cite[Theorem 1.4]{HH13} with $m(t)=A^2V(t \vee 1)$ and $s_D=2A$.\qed

\medskip

\noindent{\bf Condition \ref{cond:L1bound}:}  This is immediate from \cite[Theorem 1.12]{H08} with $k=0$.\qed

\medskip

\noindent{\bf Condition \ref{cond:smallinc}:} For lattice trees \eqref{jumpbound} holds and hence so does Condition~\ref{cond:smallinc} for any $\kappa>4$ (Remark \ref{smallincdisc}).\qed

\medskip

\noindent{\bf Condition \ref{cond:finite_int}:}  This is immediate with $(\gamma,\sigma_0^2)=(1,v)$ by Proposition \ref{prop:wkcvgcevoterLT}. \qed \\
It is worth noting however that we can also invoke Lemma~\ref{lem:int_fdd} by checking its simpler hypotheses.  The first hypothesis of Lemma~\ref{lem:int_fdd} was verified for lattice trees with $d>8$ in \cite[Theorem 1.5]{HH13} as a consequence of the survival asymptotics proved therein and \cite[Theorem 1.15]{H08} and \cite[Proposition 2.4]{HolPer07}.  The second hypothesis of Lemma~\ref{lem:int_fdd} is easily obtained from the identity 
\[\E^s_n[X_t^{\sn}(1)^p]=\frac{C_s n}{n^p}\sum_{x_1,\dots ,x_p}\sum_{\substack{T\in \T_L(o):\\  x_1,\dots, x_p \in T_{\floor{nt}}}}W(T).\]
This identity gives rise to the bound 
\begin{align}
\nn
\sup_{t\le t^*}\E^s_n[X_t^{\sn}(1)^p]\le \sup_{t\le t^*}\frac{C_{s,p}n}{n^p} \floor{nt}^{p-1}\le C_{s,p}{ t^*}^p,
\end{align}
where the factor $\floor{nt}^{p-1}$ comes from the possible temporal locations of $p-1$ branch points in the minimal subtree connecting $o$ to the points $x_1,\dots, x_p\in T_{\floor{nt}}$ (see e.g.~\cite[(4.4)-(4.5)]{H08}).

\medskip

In preparation for proving the remaining conditions, we introduce a bit of notation:

\noindent For any tree $T\in\T_L$ and any $x\in T$, let $R_x(T)$ denote the lattice tree consisting of $x$ and the descendants of $x$ in $T$, together with the edges in $T$ connecting them.  So in particular if $x\in T_n$, then
\[V(R_x(T))=\{y\in \Z^d: \exists m\ge n \text{ s.t. }x=w_n(m,y)\}.\]

Let $T_{\ngtr x}=(T\setminus R_x(T))\cup \{x\}$ denote the tree consisting of all vertices in $T$ that are not descendants of $x$. It is connected, and hence a tree, since for any such vertex $y$, the path from $o$ to $y$ cannot contain
any descendants of $x$ or else $y$ would also be a descendant. For any $B\subset\T_L$, let $B_x$ denote the event $B$ shifted by $x$ (i.e.~for $T\ni o$, $T \in B \iff T+x \in B_x$, where $+$ is addition in $\Z^d$). 

\medskip

\noindent {\bf Notation.} We will write $\vec{\omega}_n=(\omega_0,\omega_1,\dots,\omega_n)$ to denote an $n$-step random walk path, that is, a sequence of points $\omega_i\in\Z^d$ so that 
$\Vert \omega_i-\omega_{i-1}\Vert_\infty\le L$ for all $1\le i\le n$.  We write $\vec{\omega}_n:y\to z$ if, in addition, $\omega_0=y$ and $\omega_n=z$, in which case $\vec{\omega}_n$ is a random walk path from $y$ to $z$. If $\vec{\omega}$ is an $n$-step random walk path we write $\vec{\omega}\in\T_L(\omega_0)$ iff $\vec{\omega}$ is also self-avoiding (i.e., $\omega_0,\dots,\omega_n$ are distinct),
where it is understood that the edge set is precisely the set of $n$ edges $\{\{\omega_{i-1},\omega_i\}\}_{i=1}^n$.

If $(T_i)_{i \in I}$ are lattice trees, we define the union of these trees as the lattice subgraph with vertex set equal to the union of the vertex sets of the $T_i$, and edge set equal to the union of the edge sets of the $T_i$.

If $\vec{\omega}_n$ is an $n$-step random walk path and $\vec{R}_n=(R_0,\dots,R_n)$ where $R_i\in \T_L(\omega_i)$ for each $i=0,\dots, n$ then we write $\vec{R}_{_n}\ni \vec{\omega}_n$.

\begin{REM}\label{remark:ribs} Here we describe a bijection between $T \in \T_L(o)$ such that $x\in T_n$ and collections $(\vec{\omega}_n,\vec{R}_n)$, where $\vec{R}_{_n}\ni \vec{\omega}_n$, $\omega_0=o$ and $\omega_n=x$,
  and the $(R_i)_{i=0}^n$ are mutually avoiding (i.e.~vertex disjoint, which implies that  $\vec{\omega}_n\in \T_L(o)$).

Firstly note that any lattice tree  $T\in \T_L(o)$ such that $x\in T_n$ has a unique $n$-step random walk path $\vec{\omega}_n:=w(n,x)\in \T_L(o)$ of vertices and edges in $T$ from $o$ to $x$.  Define $R_i$ to be the connected component of $\omega_i$ in the tree after removing the edges of $\vec{\omega}_n$ (but not the vertices) from $T$.  Then trivially each $R_i\in \T_L(\omega_i)$, and the $(R_i)_{i=0}^n$ are mutually avoiding (i.e.~vertex disjoint).  Moreover $T$ is the union of the trees $\vec{\omega}_n$ and $(R_i)_{i=0}^n$. 

On the other hand, given an $n$-step random walk path $\vec{\omega}_n\in \T_L(o)$ from $o$ to $x$, and $\vec{R}_n \ni \vec{\omega}_n$, the (edge and vertex) union of these trees is a tree if (and only if) the $(R_i)_{i=0}^n$ are mutually avoiding.
\end{REM}

It is immediate from Remark \ref{remark:ribs} (and the product form of $W(T)$)   
that the two-point function, $\P(x \in \mc{T}_n)$, can be written as 
\begin{align}
\P(x \in \mc{T}_n)&:=\rho^{-1}\sum_{T \in \T_L(o)}W(T) \indic{x \in T_n}\nn\\
&\, =\rho^{-1}\sum_{ \vec{\omega}_n: o \to x}W(\vec{\omega}_n)\sum_{\vec{R}_{_n}\ni \vec{\omega}_n}\left(\prod_{i=0}^nW(R_i)\right)\indic{R_0,\dots, R_n \text{ avoid each other}}.\label{2point0}
\end{align} 
We henceforth write $W(\vec{R}_n):=\prod_{i=0}^nW(R_i)$ when $\vec{R}_n=(R_0,\dots, R_n)$.

Obviously we obtain an upper bound for \eqref{2point0} by replacing the indicator therein with a less restrictive one.  This observation and generalisations of it will play a crucial role in our verification of the conditions for lattice trees.

Conditions \ref{cond:self-repel} and \ref{cond:self-avoiding} will be simple consequences of the following Lemma.
\begin{LEM}
\label{lem:basictreebounds}
For all $A,B\subset\T_L$, and every $n\in \N$,
\begin{equation}
\P(x \in \mc{T}_n,\mc{T}_{\ngtr x}\in A, R_x(\mc{T})\in B_x)\le \rho \P(x \in \mc{T}_n,\mc{T}_{\ngtr x}\in A)\P(\mc{T}\in B).\label{boo}
\end{equation}
\end{LEM}
\begin{proof} 
Using Remark~\ref{remark:ribs} we see that
the left hand side of \eqref{boo} is equal to
\begin{align}
\nn&\frac{1}{\rho}\sum_{T \in\T_L}W(T) \indic{x\in T_n}\indic{T_{\ngtr x}\in A}\indic{R_x(T)\in B_x}\\
\nn&=\frac{1}{\rho}\sum_{\vec{\omega}_n:o \ra x}W(\vec{\omega}_n)\sum_{\vec{R}_{_n}\ni \vec{\omega}_n}W(\vec{R}_n)\indic{R_0,\dots, R_n \text{ avoid each other}}\\
\label{ribs}&\phantom{=\frac{1}{\rho}\sum_{\vec{\omega}_n:o \ra x}W(\vec{\omega}_n)\sum_{\vec{R}_{_n}\ni \vec{\omega}_n}W(\vec{R}_n)}\times\indic{\vec{R}_{n-1}\cup \vec{\omega}_{n}\in A}\indic{R_n\in B_x},
\end{align}
where $\vec{R}_{n}\cup \vec{\omega}_{n}=:T'$ is a lattice tree (containing $o$, and $x$ at generation $n$) due to the indicator of avoidance, and $\vec{R}_{n-1}\cup \vec{\omega}_{n}=(T'\setminus R_n)\cup \{x\}$ is a tree as well.

Now $R_n$ is a tree containing $x=\omega_n$, so by weakening the avoidance constraint this is at most
\begin{align}
\nn&\frac{1}{\rho}\sum_{\vec{\omega}_n:o \ra x}W(\vec{\omega}_n)\sum_{\vec{R}_{n}\ni \vec{\omega}_n}W(\vec{R}_n)\indic{R_0,\dots, R_{n-1} \text{ avoid each other and }x}\\
\nn&\phantom{\frac{1}{\rho}\sum_{\vec{\omega}_n:o \ra x}W(\vec{\omega}_n)\sum_{\vec{R}_{n}\ni \vec{\omega}_n}W(\vec{R}_n)}\times\indic{\vec{R}_{n-1}\cup \vec{\omega}_{n}\in A}\indic{R_n\in B_x}\\
\nn&=\frac{1}{\rho}\sum_{\vec{\omega}_n:o \ra x}W(\vec{\omega}_n)\sum_{\vec{R}_{n-1}\ni \vec{\omega}_{n-1}}W(\vec{R}_{n-1})\indic{R_0,\dots, R_{n-1} \text{ avoid each other and }x}\\
&\phantom{=\frac{1}{\rho}\sum_{\vec{\omega}_n:o \ra x}W(\vec{\omega}_n)\sum_{\vec{R}_{n-1}\ni \vec{\omega}_{n-1}}\left(\prod_{i=0}^{n-1}W(R_i)\right)}\times\indic{\vec{R}_{n-1}\cup \vec{\omega}_{n}\in A}\label{platypus0}\\
&\qquad \times \sum_{R_n \in\T_L(x)}W(R_n)\indic{R_n\in B_x},\label{platypus1}
\end{align}
where we have used the fact that $\indic{\vec{R}_{n-1}\cup \vec{\omega}_{n}\in A}$ does not depend on $R_n\setminus \{x\}$.  Now note that \eqref{platypus1} is equal to 
\[\rho \frac{1}{\rho}\sum_{R \in\T_L(x)}W(R)\indic{R \in B_x}=\rho \P(\mc{T}+x\in B_x)=\rho\P(\mc{T}\in B).\]
Next note that the weight of a lattice tree consisting of a single vertex $\{x\}$ is 1, so \eqref{platypus0} is at most
\begin{align}
\frac{1}{\rho}\sum_{\vec{\omega}_n:o \ra x}W(\vec{\omega}_n)\sum_{\vec{R}_{n}\ni \vec{\omega}_{n}}W(\vec{R}_n)\indic{R_0,\dots, R_{n} \text{ avoid each other}}\indic{\vec{R}_{n-1}\cup \vec{\omega}_{n}\in A},\label{platypus2}
\end{align}
since \eqref{platypus2} contains the case where $R_n=\{x\}$.  But \eqref{platypus2} is equal to 
\begin{align*}
\P(x \in \mc{T}_n,\mc{T}_{\ngtr x}\in A),
\end{align*}
and the result follows.
\end{proof}

\medskip

\noindent{\bf Condition \ref{cond:self-repel}:}  By \eqref{LTFn}, to verify Condition \ref{cond:self-repel} for lattice trees, it is sufficient to show that there exists $c_{\ref{cond:self-repel}}>0$ such that for all $n\in \Z_+,k\in \N$, any $T'\in\T_{\le n}$ such that $\P(\mc{T}_{\le n}=T')>0$ and any $x \in T'_n$,
\begin{equation}\P(\exists z : (n,x) \ara (n+k,z) \big | \mc{T}_{\le n}=T')\le \frac{c_{\ref{cond:self-repel}}}{k}.\label{banana0}\end{equation}

Let $B$ denote the set of lattice trees containing $o$ that survive until at least generation $k$, so $B_x$ is the set of lattice trees rooted at $x$ for which there is at least one vertex in the tree of tree distance $k$ from $x$.  Then 
\begin{align}
\nn\P(\exists z : (n,x) \ara (n+k,z) \big | \mc{T}_{\le n}=T')&=\frac{\P(\exists z : (n,x) \ara (n+k,z) \, , \,  \mc{T}_{\le n}=T')}{\P(\mc{T}_{\le n}=T')}\\
&=\frac{\P(R_x(\mc{T})\in B_x \, , \,  \mc{T}_{\le n}=T')}{\P(\mc{T}_{\le n}=T')}.\label{banana1}
\end{align}
Note that for any $T\in\T_L$, if $x \in T_n$ then $T_{\le n}=(T_{\ngtr x})_{\le n}$.  Therefore  the numerator in \eqref{banana1} can be written as
\begin{align*}
\P\big(x \in \mc{T}_n,  (\mc{T}_{\ngtr x})_{\le n}=T', R_x(\mc{T})\in B_x\big)\le \rho \P\big(x \in \mc{T}_n,  (\mc{T}_{\ngtr x})_{\le n}=T'\big)\P(\mc{T}\in B\big),
\end{align*}
where we have used Lemma \ref{lem:basictreebounds}.  But (since $x \in T'_n$), 
\[\P\big(x \in \mc{T}_n,  (\mc{T}_{\ngtr x})_{\le n}=T'\big)=\P(\mc{T}_{\le n}=T'),\]
so  for all $k\in \N$, \eqref{banana1} is bounded above by $\rho \P(\mc{T}\in B\big)=\rho\theta(k)\le c\rho/k$, by \eqref{survivbnds} and Condition \ref{cond:surv}. By \eqref{mLTs} we have proved \eqref{banana0}, as needed. \qed

\medskip

\noindent{\bf Condition \ref{cond:self-avoiding}:}
 Let $(R_x(\mc{T}))_m$ denote the set of vertices in the tree $R_x(\mc{T})$ of tree distance $m$ from $x$ (e.g.~$(R_x(\mc{T}))_0=\{x\}$).
By Lemma~\ref{disccondsa} we need to show that there exists $c_{\ref{cond:self-avoiding}}>0$ such that for any $\ell\in\Z_+,m\in\N^{\ge 4}$, and $M>0$,
\begin{align}
\nn&\E\left[\sum_{x \in \mc{T}_\ell}\1\Big(\exists y \in (R_x(\mc{T}))_m\, , \, \sum_{j=m+2}^{2m-1} \#(R_x(\mc{T}))_j\le M\Big)\right]\\
 &\le c_{\ref{cond:self-avoiding}}  \P\Big(S^{\sss(1)}\ge m\, ,\, \sum_{j=m+2}^{2m-1} \#\mc{T}_j\le M\Big). \label{mmm0}
 \end{align}
The left hand side can be written as 
\begin{align}
&\sum_{x\in \Z^d}\P\Big(x \in \mc{T}_\ell\, , \, \exists y \in (R_x(\mc{T}))_m\, , \,\sum_{j=m+2}^{2m-1} \#(R_x(\mc{T}))_j\le M\Big).\label{mmm1}
\end{align}

Let $B_x$ denote the set of lattice trees $T$ rooted at $x$ (i.e.~the unique particle of generation 0 is $x$) that survive until time $m$ such that the total number of particles of generation between $m+2$ and $2m-1$ is at most $M$, and let $B=B_o$.  Then \eqref{mmm1} is 
\begin{equation*}
\sum_{x\in \Z^d}\P\Big(x \in \mc{T}_\ell, R_x(\mc{T})\in B_x\Big).
\end{equation*}
Applying Lemma \ref{lem:basictreebounds} this is at most
\begin{equation*}
\rho \sum_{x\in \Z^d}\P(x \in \mc{T}_\ell)\P( \mc{T}\in B)=\rho \E\left[\sum_{x \in \mc{T}_\ell}1\right]\P\Big(\exists y \in \mc{T}_m,\sum_{j=m+2}^{2m-1} \#\mc{T}_j\le M\Big),
\end{equation*}
which, by Condition~\ref{cond:L1bound}, verifies \eqref{mmm0} with $c_{\ref{cond:self-avoiding}}=\rho c_{\ref{cond:L1bound}}$.\qed

\medskip

For $T\in \T_L(o)$, $0\le m< n$, and $x \in T_{n}$, let $x_m(T)=w_{m}(n,x)$ denote the unique ancestor of $x$ in $T$ of generation $m$.  
In preparation for verifying Condition \ref{cond:6moment} subject to Lemma \ref{lem:LT6} we prove the following Lemma.
\begin{LEM}
\label{lem:reduction} If $c_{\ref{cond:L1bound}}$ is the constant in Condition~\ref{cond:L1bound} for lattice trees, then for any $f:\Z^d \ra \R_+$ such that $f(-x)=f(x)$ and any $m<n\in \N$,
\begin{equation}
\nn
\E\left[\sum_{x \in \mc{T}_{n}}f(x-x_m(\mc{T}))\right]\le c_{\ref{cond:L1bound}} \sum_{y\in \Z^d}f(y)\P(y \in \mc{T}_{n-m}).\end{equation}
\end{LEM}
\proof The left hand side is equal to 
\begin{align}
&\rho^{-1}\sum_{x,y \in \Z^d}\sum_{T \in\T_L}W(T) \indic{x \in T_{n}}\indic{x_m(T)=y}f(x-y).\label{hippo1}
\end{align}
Now every 

{\em tree $T$ rooted at $o$ and containing $x$ at tree distance $n$ from $o$, such that the unique path in the tree from $o$ to $x$ passes through $y$ at tree distance $m$ from $o$}

is also 

{\em a tree (with the same weight) rooted at $x$ containing $o$ at tree distance $n$ from $o$ such that the unique path in the tree from $x$ to $o$ passes through $y$ at tree distance $n-m$ from $x$},

and vice versa.  The above are actually the same tree, but since we are also specifying the root, we will refer to the latter as $T_x$.

Translating this tree by $-x$, we obtain a tree $T'=T_x-x$ (with the same weight as $T$) rooted at $o=x-x$, containing $x':=-x$ at tree distance $n$ from $o$ and such that the unique path in $T'$ from $o$ to $x'$ passes through $y':=y-x$ at tree distance $n-m$ from $o$.  Since $x-y=-y'$, and $f(-y')=f(y')$, \eqref{hippo1} is equal to 
\begin{align*}
\rho^{-1}\sum_{x',y' \in \Z^d}\sum_{T' \in\T_L}W(T') \indic{x' \in T'_{n}}\indic{x'_{n-m}(T')=y'}f(y').
\end{align*}
Now we can simply drop the $'$  to get that \eqref{hippo1} is equal to 
\begin{align*}
\rho^{-1}\sum_{x,y \in \Z^d}\sum_{T \in \T_L}W(T) \indic{x \in T_{n}}\indic{x_{n-m}(T)=y}f(y).
\end{align*}
Now as in \eqref{ribs} we can write this as
\begin{align}
\rho^{-1}\sum_{x,y \in \Z^d}\sum_{\vec{\omega}_n:o \overset{n-m}\ra y \overset{m}\ra x}W(\vec{\omega}_n)\sum_{\vec{R}_{n}\ni \vec{\omega}_n}W(\vec{R}_n)\indic{(R_i)_{0\le i\le n} \text{ avoid each other}}f(y),\label{hippo2}
\end{align}
where the sum over $\vec{\omega}_n$ is a sum over random walk paths of length $n$ from $o$ to $x$ that are at $y$ at time $n-m$.

Now since $R_{n-m}\ni y$, we have that 
\begin{align*}
&\indic{(R_i)_{0\le i\le n} \text{ avoid each other}}\\
&\le \indic{(R_i)_{0\le i\le n-m} \text{ avoid each other}}\indic{(R_i)_{n-m+1\le  i\le n} \text{ avoid each other and }y}.
\end{align*}
Thus, using the fact that also the weight of a tree containing a single vertex $y$ is 1, we see that \eqref{hippo2} is at most
\begin{align*}
\rho^{-1}\sum_{x,y \in \Z^d}\sum_{\vec{\omega}^{\sss(1)}_{n-m}:o \to y }&\sum_{\vec{\omega}^{\sss(2)}_m:y  \to x }W(\vec{\omega}^{\sss(1)}_{n-m})W(\vec{\omega}^{\sss(2)}_m)\\
&\sum_{\vec{R}^{\sss(1)}_{n-m}\ni \vec{\omega}_{n-m}^{\sss(1)}}W(\vec{R}^{\sss(1)}_{n-m})\indic{(R^{\sss(2)}_i)_{0\le i\le n-m} \text{ avoid each other}}f(y)\\
&\sum_{\vec{R}^{\sss(2)}_{m}\ni \vec{\omega}_m^{\sss(2)}}W(\vec{R}^{\sss(2)}_{m})\indic{(R^{\sss(2)}_{i'})_{0\le i'\le m} \text{ avoid each other}}.
\end{align*}
Collecting terms, this is equal to 
\begin{align*}
&\rho^{-1}\sum_{y \in \Z^d}\sum_{\vec{\omega}^{\sss(1)}_{n-m}:o \to y }W(\vec{\omega}^{\sss(1)}_{n-m})\sum_{\vec{R}^{\sss(1)}_{n-m}\ni \vec{\omega}^{\sss(1)}_{n-m}}W(\vec{R}^{\sss(1)}_{n-m})\indic{(R^{\sss(1)}_i)_{0\le i\le n-m} \text{ avoid each other}}f(y)\\
&\times \Bigg(\sum_{x \in \Z^d}\sum_{\vec{\omega}^{\sss(2)}_m:y  \to  x }W(\vec{\omega}^{\sss(2)}_m)\sum_{\vec{R}^{\sss(2)}_{m}\ni \vec{\omega}^{\sss(2)}_m}W(\vec{R}^{\sss(2)}_{m})\indic{(R^{\sss(2)}_{i'})_{0\le i\le m} \text{ avoid each other}}\Bigg).
\end{align*}
Changing variables from $x$ to $u=x-y$ in the last summation, we see the above is equal to 
\[\sum_{y\in \Z^d}f(y)\P(y \in \mc{T}_{n-m})\sum_{u \in\Z^d}\P(u\in \mT_m)\le c_{\ref{cond:L1bound}}\sum_{y\in \Z^d}f(y)\P(y \in \mc{T}_{n-m}),\]
by Condition~\ref{cond:L1bound} for lattice trees, as required.\qed
\medskip

{\noindent \bf Condition \ref{cond:6moment}:} Apply Lemma \ref{lem:reduction} with $f(x)=|x|^6$, $n=t$ and $m=t-s$ to see that in order to verify Condition \ref{cond:6moment}, it is sufficient to prove that $\sum_{y\in \Z^d}|y|^6\P(y \in \mc{T}_{s})\le c(s\vee 1)^{6/2}$.  Together with Lemma \ref{lem:LT6} this verifies Condition \ref{cond:6moment}.\qed

Assuming Lemma \ref{lem:LT6}, we can now prove Theorem \ref{thm:LT}.
\begin{proof}[Proof of Theorem \ref{thm:LT}]
Conditions \ref{cond:surv}-\ref{cond:self-avoiding} all hold, with $s_D=2A$ in Condition~\ref{cond:surv}, $p=6$ in Condition \ref{cond:6moment}, any $\kappa>4$ in Condition~\ref{cond:smallinc} and $(\gamma,\sigma_0^2)=(1,v)$ in Condition~\ref{cond:finite_int}.  Hence, with a bit of elementary arithmetic, Theorem  \ref{thm:mod_con_disc} implies (t1), Theorem \ref{thm:range} implies (t2) and Theorem \ref{thm:one-arm} implies (t3).
\end{proof}

\section{Proof of Lemma \ref{lem:LT6}}
\label{sec:tree6}

We will prove the required bound for $\rho t_n(x)$, which is of course equivalent.  
Before we prove Lemma \ref{lem:LT6}, note that
\begin{align}\label{tnformula}
\rho t_n(x)&=\sum_{\vec{\omega}_n:o \ra x}W(\vec{\omega}_n)\sum_{\vec{R}_n \ni \vec{\omega}_n}W(\vec{R}_{n})K_{[0,n]}(\vec{R}_n)
\end{align}
where for $0\le a<b\le n$ we have 
\begin{align*}
K_{[a,b]}(\vec{R}_n)=\prod_{a\le s<t\le b}\indic{R_s\cap R_t=\varnothing}=\prod_{a\le s<t\le b}\left[1+U_{s,t}(\vec{R}_n)\right],
\end{align*}
where $U_{s,t}(\vec{R}_n)=-\indic{R_s \cap R_t \ne \varnothing}$.

The first step of the lace expansion analysis (see \cite{Sla06} for an introduction to the lace expansion in various settings) is to rewrite $K_{[0,n]}$ as 
\begin{equation*}
K_{[0,n]}=\sum_{\Gamma \in \mc{G}[0,n]}\prod_{(s,t)\in \Gamma}U_{s,t},
\end{equation*}
where $\mc{G}[0,n]$ is the set of graphs (ordered pairs of vertices) on $[0,n]$.  Every such graph can be decomposed into its connected components (here a graph $\Gamma$ on $[a,b]$ is connected if $\cup_{(s,t)\in \Gamma}[s,t]=[a,b]$.  If $\Gamma$ is a graph on $[0,n]$ and $a\in [0,n]$ satisfies $a \notin \cup_{(s,t)\in \Gamma}[s,t]$ then the connected component of $a$ is simply $[a,a]=\{a\}$.   Let $\mc{G}^{\text{conn}}[a,b]$ denote the set of connected graphs on $[a,b]$ (here the empty graph is considered to be a connected graph on $[a,a]$), and 
\[J_{[a,b]}:=\sum_{\Gamma \in \mc{G}^{\text{conn}}[a,b]}\prod_{(s,t)\in \Gamma}U_{s,t}.\]
Let $\pi_n(x)=\sum_{\vec{\omega}_n:o \ra x}W(\omega)\sum_{\vec{R}_n \ni \vec{\omega}_n}W(\vec{R}_{n})J_{[0,n]}$, and note that 
\[\pi_0(x)=\delta_{o,x}\sum_{R_0\in\T_L}W(R_0)=\delta_{o,x}\rho.\]

Write $n-\vec{M}$ for $n-\vec{M}\cdot \vec{1}$ and $\vec{M}_r$ for $(M_1,\dots ,M_r)$.  Writing $\vec{m}_{N}\ge 0$ for the set of $(m_1,\dots, m_N)\in (\Z_+)^N$, note that (with $\ms_0=0$ and $\ms_i=\sum_{\ell=1}^i m_\ell$), 
\begin{align}
K_{[0,n]}=\sum_{N=1}^{n+1}   \sum_{\substack{\vec{m}_{N}\ge 0:\\  \vec{m}_N+N-1=n}}   \prod_{j=1}^N J_{[\ms_{j-1}+{j-1},\ms_j+{j-1}]},\label{KJ1}
\end{align}
where $N$ denotes the number of connected components. (See Figure~\ref{fig:expand}.)  The case that $N=n+1$ corresponds to the empty graph where every vertex $0,1,\dots, n$ is its own connected component.
Thus,
\begin{align}
\rho t_n(x)&=\sum_{N=1}^{n+1}   \sum_{\substack{\vec{m}_{N}\ge 0:\\  \vec{m}_N+N-1=n}} \sum_{\vec{\omega}_n:o \ra x}W(\vec{\omega}_n)\sum_{\vec{R}_n \ni \vec{\omega}_n}W(\vec{R}_{n}) \prod_{j=1}^N J_{[\ms_{j-1}+{j-1},\ms_j+{j-1}]}.\label{tnexpanded}
\end{align}
\begin{figure}
\begin{center}
\begin{tikzpicture}
\node[circle,fill=black,scale=.5,label=below:{$0$}] (A) at (0,0) {};
\node[circle,fill=black,scale=.5,label=below:{$$}] (B) at (1,0) {};
\node[circle,fill=black,scale=.5,label=below:{$$}] (C) at (2,0) {};
\node[circle,fill=black,scale=.5,label=below:{$$}] (D) at (3,0) {};
\node[circle,fill=black,scale=.5,label=below:{$$}] (E) at (4,0) {};
\node[circle,fill=black,scale=.5,label=below:{$$}] (F) at (5,0) {};
\node[circle,fill=black,scale=.5,label=below:{$$}] (G) at (6,0) {};
\node[circle,fill=black,scale=.5,label=below:{$$}] (H) at (7,0) {};
\node[circle,fill=black,scale=.5,label=below:{$$}] (I) at (8,0) {};
\node[circle,fill=black,scale=.5,label=below:{$n$}] (J) at (9,0) {};
\draw (A)--(B)--(C)--(D)--(E)--(F)--(G)--(H)--(I)--(J);
\path (B) edge  [bend left=40] node[above] {} (D);
\path (F) edge  [bend left=40] node[above] {} (I);
\end{tikzpicture}
\end{center}
\caption{Here $n=9$, $N=5$, the lengths of the 5 ``connected graphs'' are $m_1=0$, $m_2=2$, $m_3=0$, $m_4=3$ and $m_5=0$.  Note that $0+1+2+1+0+1+3+1+0=9$. }
\label{fig:expand}

\end{figure}
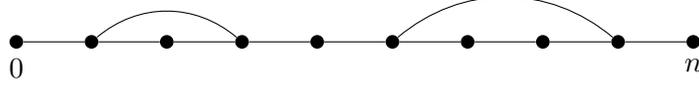

Note that the case $N=1$ (for which $\ms_1=m_1=n$) gives the contribution
\begin{align}\label{pindefn}
\sum_{\vec{\omega}_n:o \ra x}W(\vec{\omega}_n)\sum_{\vec{R}_n \ni \vec{\omega}_n}W(\vec{R}_{n})   J_{[0,n]}:=\pi_n(x).
\end{align}
For $N>1$, letting $y_0=o$ and $y_N=x$ and $y_i=\omega_{\ms_i+i}$ for $i=1,\dots, N-1$ we write $\sum_{\vec{\omega}_n:o \overset{\vec{\ms}_N,\vec{y}_N}{\ra} x}$ to denote the sum over random walk paths from the origin to $x$ in $n$ steps that pass through $y_i$ at $\ms_i+i$ steps for each $i=1,\dots, N-1$.  Then we can write \eqref{tnexpanded}  as 
\begin{align}
\nn\rho t_n(x)&=\pi_n(x)+\sum_{N=2}^{n+1}   \sum_{\substack{\vec{m}_{N}\ge 0:\\  \vec{m}_N+N-1=n}} \sum_{\vec{y}_{N-1}} \sum_{\vec{\omega}_n:o \overset{\vec{\ms}_N,\vec{y}_N}{\ra} x}W(\omega)\\&\phantom{=\pi_n(x)+\sum_{N=2}^{n+1}    }\times\sum_{\vec{R}_n \ni \vec{\omega}_n}W(\vec{R}_{n})   \prod_{\ell=1}^N J_{[\ms_{\ell-1}+{\ell-1},\ms_\ell+{\ell-1}]}.\label{tn2}
\end{align}

Let $\pi_{m,+}(y)=(\pi_m*z_DD)(y)$. 
We will be interested in the terms~($r\in\Z_+$)
\begin{align*}
\ttt{r}{n}&:=\sum_y \rho|y|^{2r} t_n(y), \qquad n\ge 0\\
\tplus{r}{n}&:=\sum_y |y|^{2r}(\rho t_{n-1}*z_{\sss D}D)(y), \quad 
\text{ for }n\in\N, \text{ and }\tplus{r}{0}:=1,\\
\ppp{r}{n}&:=\sum_y |y|^{2r}\pi_n(y), \qquad n\ge 0\\
\pplus{r}{n}&:=\sum_y |y|^{2r}(\pi_{n}*z_{\sss D}D)(y), \qquad n\ge 0.
\end{align*}
Note that $\rho t_0(y)=\rho \delta_{o,y}=\pi_0(y)$.

We know from {Condition \ref{cond:L1bound}} that there is a $K$ so that for all $n\in\Z_+$,
\begin{align}
\ttt{0}{n}\le K, \qquad \tplus{0}{n}\le K.\label{b1}
\end{align}
Moreover we know e.g. from \cite[Proposition 5.1]{H08} that 
\begin{align}
\sum_y |y|^{2r}|\pi_n(y)|\le \frac{CL^{2r}}{(n\vee 1)^{\frac{d-4}{2}-r}}, \qquad \textrm{for }r=0,1,2. \label{b2}
\end{align}
This implies that also 
\begin{align}
\sum_y |y|^{2r}(|\pi_n|*z_{\sss D}D)(y)\le \frac{CL^{2r}}{(n\vee 1)^{\frac{d-4}{2}-r}}, \qquad \textrm{for }r=0,1,2.\label{b3}
\end{align}
Note that we allow the constant $C$ to vary from line to line.  
Since $\pi_n(y)=0$ if $|y|>nL$  we get also that 
\begin{align}
&\sum_y |y|^{6}|\pi_n(y)|\le n^2L^2\frac{CL^{4}}{(n\vee 1)^{\frac{d-8}{2}}}\le CL^6(n\vee 1)^{(12-d)/2}\label{b4}\\
&\sum_y |y|^{6}(|\pi_n|*z_{\sss D}D)(y)\le n^2L^2\frac{CL^{4}}{(n\vee 1)^{\frac{d-8}{2}}}\le CL^6(n\vee 1)^{(12-d)/2}.\label{b5}
\end{align}

For $r\in\Z_+$, let $\mc{A}_r$ denote the set of square upper triangular matrices $A=(a_{k_1,k_2})_{k_1,k_2\in [r]}\in (\Z_+)^{r \times r}$ such that $a_{k_1,k_2}=0$ if $k_2<k_1$.  Here if $r=0$, $\mc{A}_0$ consists only of the empty matrix $\emptyset$.
\begin{DEF}
\label{def:null}
A matrix $A\in \mc{A}_r$ is said to be {\em null} if $B:=A+A^T$ has at least one row $i$ such that $\sum_{j\ne i}b_{i,j}$ is odd.
\end{DEF}
For $r$ as above and $A\in \mc{A}_r$ and $n \in \N$ define $\Xi^{[0]}_n[\varnothing]=\ttt{0}{n}$, and 
\begin{align}
\nn&\Xi^{[r]}_n[A]\\
\nn&:=\sum_{\vec{M}_r\ge 0}\sum_{\vec{L}_r\ge 0}\left(\prod_{s=1}^r \tplus{0}{M_s}\right)\sum_{\vec{v}_r} \left(\prod_{k_1=1}^r \prod_{k_2=1}^r (v_{k_1}\cdot v_{k_2})^{a_{k_1,k_2}}\right)\\
&\phantom{:=\sum_{\vec{M}_r\ge 0}\sum_{\vec{L}_r\ge 0}(}\times \left(\prod_{\ell=1}^r \pi_{L_\ell,+}(v_\ell)\right)\ttt{0}{n-\vec{M}_r -\vec{L}_r-r }\indic{n\ge \vec{M}_r +\vec{L}_r+r}\label{Xi1}\\
\nn&\quad +\sum_{\vec{M}_r\ge 0}\sum_{\vec{L}_r\ge 0}\left(\prod_{s=1}^r \tplus{0}{M_s}\right)\sum_{\vec{v}_r} \left(\prod_{k_1=1}^r \prod_{k_2=1}^r (v_{k_1}\cdot v_{k_2})^{a_{k_1,k_2}}\right)\\
&\phantom{:=\sum_{\vec{M}_r\ge 0}\sum_{\vec{L}_r\ge 0}(}\times \left(\prod_{\ell=2}^{r} \pi_{L_\ell,+}(v_\ell)\right)  \pi_{L_1}(v_1)\indic{n=\vec{M}_r +\vec{L}_r+(r-1)}.\label{Xi2}
\end{align}
Note that the $r=0$, $A=\varnothing$ case can be interpreted as \eqref{Xi1} with no sums or products, leaving only $\ttt{0}{n-0-0-0}$.

\begin{LEM}
\label{lem:null1}
Fix $r$, $A\in \mc{A}_r$, and $n \in \N$.  If $A$ is null then $\Xi^{[r]}_n[A]=0$.
\end{LEM}
\proof Since $A$ is null, there exists some $i\le r$ such that $\sum_{j\ne i}b_{i,j}$ is odd.  This means that the term $v_i$ appears an odd number of times in $\Xi^{[r]}_n[A]$.  Since $\pi_{L_i}$ and $\pi_{L_i,+}$ are symmetric ($\pi_{L_i}(-v)=\pi_{L_i}(v)$ and  $\pi_{L_i,+}(-v)=\pi_{L_i,+}(v)$),  for each fixed $(v_j)_{j \ne i}$, the sum over $v_i$ gives 0, hence $\Xi^{[r]}_n[A]=0$.   \qed

\begin{LEM}
\label{lem:swap} Let $j_1,j_2$ be distinct indices in $\{2,\dots,r\}$ and 
 $A,A'\in \mc{A}_r$ be such that $A'$ is obtained by swapping columns $j_1,j_2$ of $A$ and then swapping rows $j_1,j_2$ of the resulting matrix, where $j_1\ne 1, j_2\ne 1$.  That is, $a'_{i,j}=a_{\pi_i,\pi_j}$, where $\pi:\{1,\dots,r\}\to\{1,\dots, r\}$ swaps $j_1$ and $j_2$.  Then $\Xi^{[r]}_n[A]=\Xi^{[r]}_n[A']$ for all $n\in\N$.
\end{LEM}
\proof We must show that the above sum defining $\Xi_n^{[r]}[A]$ is invariant under the permutation of the indices of $A$ by $\pi$. To  see this we may sum over $\vec{v}^\pi_r=(v_{\pi_1},\dots,v_{\pi_r})$ and $\vec{L}^\pi_r$ instead of $\vec{v}_r$ and $\vec{L}_r$. The invariance is then easy to recognize.\qed

\medskip

Note that the above lemma does not hold if either $j_1=1$ or $j_2=1$, due to the asymmetry in \eqref{Xi2}.

\medskip

The majority of the work remaining in the paper is to prove the following. 
\begin{LEM}
\label{lem:third}
There exist positive integers $k_2,\dots, k_9$ such that for all $n\in\N$,
\begin{align}
\nn\ttt{3}{n}&=\Xi^{[1]}_n[(3)]+
k_2\Xi^{[2]}_n\left[\begin{pmatrix}0& 2\\0 & 1\end{pmatrix}\right]+
k_3\Xi^{[2]}_n\left[\begin{pmatrix}1& 2\\0 & 0\end{pmatrix}\right] +k_4\Xi^{[2]}_n\left[\begin{pmatrix}1& 0\\0 & 2\end{pmatrix}\right]\\
\nn&\qquad  +
k_5\Xi^{[2]}_n\left[\begin{pmatrix}2& 0\\0 & 1\end{pmatrix}\right]+k_6\Xi^{[3]}_n\left[\begin{pmatrix}0& 2 & 0\\0 & 0 & 0\\ 0&0&1\end{pmatrix}\right]+ 
k_7 \Xi^{[3]}_n\left[\begin{pmatrix}1& 0 & 0\\0 & 0 & 2\\ 0&0&0\end{pmatrix}\right]\\
\nn&\qquad +
k_8 \Xi^{[3]}_n\left[\begin{pmatrix}1& 0 & 0\\0 & 1 & 0\\ 0&0&1\end{pmatrix}\right]+
k_9 \Xi^{[3]}_n\left[\begin{pmatrix}0& 1 & 1\\0 & 0 & 1\\ 0&0&0\end{pmatrix}\right].
\end{align}
\end{LEM}
Lemma \ref{lem:third} requires some explanation.  Recall \eqref{tn2} and note that for fixed $N$ and with $y_N=x$ and $y_0=o$ and letting $z_i=y_i-y_{i-1}$ for $i=1,\dots, N$ we can write $x=\sum_{i=1}^Nz_i$.  Expanding each of the $p\in \N$ factors of $|x|^2$  enables us to write
\begin{align}
|x|^{2p}=\prod_{j=1}^p\left(\sum_{i_j=1}^N|z_{i_j}|^2+\sum_{i_j=1}^N\sum_{i_{j+p}\ne i_j}z_{i_j}\cdot z_{i_{j+p}}\right),\nn
\end{align}
where the choice of summation index $i_{j+p}$ may not be the most natural at this point.
Expanding the above in the case $p=3$, and relabelling indices of summation gives
\begin{align}
|x|^6=&\sum_{i_1=1}^N\sum_{i_2=1}^N\sum_{i_3=1}^N |z_{i_1}|^2|z_{i_2}|^2|z_{i_3}|^2\label{6_1}\\
&+3\sum_{i_1=1}^N\sum_{i_2=1}^N\sum_{i_3=1}^N \sum_{i_4=1}^N \indic{i_3\ne i_4}|z_{i_1}|^2|z_{i_2}|^2(z_{i_3}\cdot z_{i_4})\label{6_2}\\
&+3\sum_{i_1=1}^N	\sum_{i_2=1}^N\sum_{i_3=1}^N \sum_{i_4=1}^N\sum_{i_5=1}^N \indic{i_2\ne i_3}\indic{i_4\ne i_5}|z_{i_1}|^2(z_{i_2}\cdot z_{i_3})(z_{i_4}\cdot z_{i_5})\label{6_3}\\
\nn &+\sum_{i_1=1}^N	\sum_{i_2=1}^N\sum_{i_3=1}^N \sum_{i_4=1}^N\sum_{i_5=1}^N \sum_{i_6=1}^N \indic{i_1\ne i_2}\indic{i_3\ne i_4}\indic{i_5\ne i_6}(z_{i_1}\cdot z_{i_2})\\
&\phantom{+\sum_{i_1=1}^N	\sum_{i_2=1}^N\sum_{i_3=1}^N \sum_{i_4=1}^N\sum_{i_5=1}^N \sum_{i_6=1}^N \indic{i_1\ne i_2}}\times(z_{i_3}\cdot z_{i_4})(z_{i_5}\cdot z_{i_6}).\label{6_4}
\end{align}
Now one can split these sums according to the number $r$ of {\em distinct indices} (without loss of generality they are $i_1,\dots, i_r$).  The number of distinct indices determines the dimensions $r$ of a matrix $A$.  The $j,k$-th entry of the matrix is precisely the number of times that $z_{i_j}\cdot z_{i_k}$ appears in a given sum.  For example \eqref{6_1} includes contributions where $i_1,i_2,i_3$ are all distinct (so $r=3$), and in this case $|z_{i_j}|^2=z_{i_j}\cdot z_{i_j}$ appears exactly once for each $j=1,2,3$, which corresponds to the matrix $\begin{pmatrix}1& 0 & 0\\0 & 1 & 0\\ 0&0&1\end{pmatrix}$ in the lemma.  But  \eqref{6_1} also includes (contributions where exactly two indices are distinct and) a contribution where $i_1=i_2=i_3$ (so $r=1$), in which case $|z_{i_1}|^2=z_{i_1}\cdot z_{i_1}$ appears 3 times, corresponding to the matrix $(3)$.  

The proof will proceed by expressing $\ttt{3}{n}$ as a large sum involving integral multiples of terms $\Xi^{[r]}_n[A]$, arising in this way, of which many of the $A$ for $r=2,\dots, 6$ are null.  These latter terms are then dropped thanks to Lemma \ref{lem:null1}.  Lemma \ref{lem:third} shows all the non-null matrices that arise.

A similar but simpler argument applies to the case $p=2$, for which we obtain the following.
\begin{LEM}
\label{lem:second}
\begin{align*}
\ttt{2}{n}&=
\Xi^{[1]}_n[(2)]+4\Xi^{[2]}_n\left[\begin{pmatrix}0& 2\\0 & 0\end{pmatrix}\right]+2\Xi^{[2]}_n\left[\begin{pmatrix}1& 0\\0 & 1\end{pmatrix}\right].
\end{align*}
\end{LEM}
We will be using Lemma \ref{lem:second} in the extension to $d=9$ given in Section~\ref{sec:upgrade} below.  A careful proof of Lemma~\ref{lem:second} would be similar to, but simpler than, the proof of Lemma~\ref{lem:third} given below, and so we will omit it.

\begin{proof}[Proof of Lemma \ref{lem:LT6} for $d\ge 10$, assuming Lemma \ref{lem:third}.]
We may assume $n\in\N$ because $t_0(x)=\1(x=0)$.
It is sufficient to show that each of the terms $ \Xi^{[r]}_n[A]$ (for non-null $A$) appearing in Lemma \ref{lem:third} are bounded by $Cn^3$.

Recall our definition of $\Xi^{[r]}_n[A]$ from \eqref{Xi1}-\eqref{Xi2} gives 
\begin{align}
\nn&\Xi^{[r]}_n[A]\\
\nn&\le \sum_{\vec{M}_r\ge 0}\sum_{\vec{L}_r\ge 0}\left(\prod_{s=1}^r \tplus{0}{M_s}\right)\sum_{\vec{v}_r} \left(\prod_{k_1=1}^r \prod_{k_2=1}^r |v_{k_1}\cdot v_{k_2}|^{a_{k_1,k_2}}\right)\\
&\phantom{\le \sum_{\vec{M}_r\ge 0}\sum_{\vec{L}_r\ge 0}\left(\prod_{s=1}^r \tplus{0}{M_s}\right)} \times\left(\prod_{\ell=1}^r |\pi_{L_\ell,+}(v_\ell)|\right)\ttt{0}{n-\vec{M}_r -\vec{L}_r-r }\indic{n\ge \vec{M}_r +\vec{L}_r+r}
\nn
\\
\nn&\quad +\sum_{\vec{M}_r\ge 0}\sum_{\vec{L}_r\ge 0}\left(\prod_{s=1}^r \tplus{0}{M_s}\right)\sum_{\vec{v}_r} \left(\prod_{k_1=1}^r \prod_{k_2=1}^r (v_{k_1}\cdot v_{k_2})^{a_{k_1,k_2}}\right)\\
&\phantom{\le \sum_{\vec{M}_r\ge 0}\sum_{\vec{L}_r\ge 0}\left(\prod_{s=1}^r \tplus{0}{M_s}\right)}\times \left(\prod_{\ell=2}^{r} |\pi_{L_\ell,+}(v_\ell)|\right)  |\pi_{L_1}(v_1)|\indic{n=\vec{M}_r +\vec{L}_r+(r-1)}.\nn
\end{align}
From \eqref{b1} and using Cauchy-Schwarz ($|u\cdot v|\le |u||v|$) this is bounded above by a constant times 
\begin{align}
&\sum_{\vec{M}_r\ge 0}\sum_{\vec{L}_r\ge 0}\sum_{\vec{v}_r} \left(\prod_{k_1=1}^r \prod_{k_2=1}^r (|v_{k_1}||v_{k_2}|)^{a_{k_1,k_2}}\right) \left(\prod_{\ell=1}^r |\pi_{L_\ell,+}(v_\ell)|\right)\indic{n\ge \vec{M}_r +\vec{L}_r+r}
\nn
\\
&\quad +\sum_{\vec{M}_r\ge 0}\sum_{\vec{L}_r\ge 0}\sum_{\vec{v}_r} \left(\prod_{k_1=1}^r \prod_{k_2=1}^r (|v_{k_1}| |v_{k_2}|)^{a_{k_1,k_2}}\right) \left(\prod_{\ell=2}^{r} |\pi_{L_\ell,+}(v_\ell)|\right)  |\pi_{L_1}(v_1)|\indic{n=\vec{M}_r +\vec{L}_r+(r-1)}.\nn%
\end{align}

It follows that 
\begin{align}
\nn\Xi^{[1]}_n[(3)]&\le C\sum_{M_1\ge 0}\sum_{L_1\ge 0}\indic{n\ge M_1 +L_1+1}\sum_{v_1} (|v_{1}|^2)^3|\pi_{L_1,+}(v_1)|\\
\label{Xi3andpi}&\quad +C\sum_{M_1\ge 0}\sum_{L_1\ge 0}\indic{n=M_1 +L_1}\sum_{v_1}(|v_{1}|^2)^3|\pi_{L_1}(v_1)|.
\end{align}
By \eqref{b4} and \eqref{b5} this is at most 
\begin{align}\label{Xi3andpib}
 CL^6\sum_{M_1\ge 0}\sum_{L_1\ge 0}\indic{n\ge M_1 +L_1+1} (L_1\vee 1)^{(12-d)/2},
\end{align}
which is at most $CL^6 n^3$ when $d\ge 10$ (recall that $n\ge 1$ now).

Next,
\begin{align*}
&\Xi^{[2]}_n\left[\begin{pmatrix}0& 2\\0 & 1\end{pmatrix}\right]\\
&\le C\sum_{M_1,M_2\ge 0}\sum_{L_1,L_2\ge 0}\indic{n\ge \vec{M}_2 +\vec{L}_2+2}\sum_{v_1,v_2}(|v_1||v_2|)^2|v_2|^2
|\pi_{L_1,+}(v_1)||\pi_{L_2,+}(v_2)|\\
&\quad + C\sum_{M_1,M_2\ge 0}\sum_{L_1,L_2\ge 0}\indic{n=\vec{M}_2 +\vec{L}_2+1}\sum_{v_1,v_2}(|v_1||v_2|)^2|v_2|^2|\pi_{L_1}(v_1)||\pi_{L_2,+}(v_2)|.
\end{align*}
Note that the powers here are $|v_1|^2|v_2|^4$.  Using the bounds  \eqref{b2},\eqref{b3} (with $r=2$ for the 4th power and with $r=1$ for the 2nd power) this is at most 
\begin{align*}
CL^2 CL^4\sum_{M_1,M_2\ge 0}\sum_{L_1,L_2\ge 0}\indic{n\ge \vec{M}_2 +\vec{L}_2+1}(L_1\vee 1)^{-(d-6)/2}(L_2\vee 1)^{-(d-8)/2},
\end{align*}
which is less than $CL^6 n^3$ when $d>8$.
The terms $\Xi^{[2]}_n\left[\begin{pmatrix}1& 2\\0 & 0\end{pmatrix}\right]$ and $\Xi^{[2]}_n\left[\begin{pmatrix}1& 0\\0 & 2\end{pmatrix}\right]$ and 
$\Xi^{[2]}_n\left[\begin{pmatrix}2& 0\\0 & 1\end{pmatrix}\right]$ are bounded in the same way.

Finally, the non-null terms $\Xi^{[3]}_n[\bullet]$ in the lemma each give
\[\left(\prod_{k_1=1}^3 \prod_{k_2=1}^3 (|v_{k_1}||v_{k_2}|)^{a_{k_1,k_2}}\right)=|v_1|^2|v_2|^2|v_3|^2.\]
So all of these terms each contribute at most a constant times 
\begin{align*}
&\sum_{M_1,M_2,M_3\ge 0}\sum_{L_1,L_2,L_3\ge 0}\indic{n\ge \vec{M}_3 +\vec{L}_3+3}\sum_{v_1,v_2,v_3}(|v_1||v_2|)^2|v_3|^2|\pi_{L_1,+}(v_1)||\pi_{L_2,+}(v_2)||\pi_{L_3,+}(v_3)|\\
&+\sum_{M_1,M_2,M_3\ge 0}\sum_{L_1,L_2,L_3\ge 0}\indic{n=\vec{M}_3 +\vec{L}_3+2}\sum_{v_1,v_2,v_3}(|v_1||v_2|)^2|v_3|^2|\pi_{L_1}(v_1)||\pi_{L_2,+}(v_2)||\pi_{L_3,+}(v_3)|.
\end{align*}
Using \eqref{b3} with $r=1$ three times we get that this is at most 
\begin{align*}
&C(L^2)^3  \sum_{M_1,M_2,M_3\ge 0}\sum_{L_1,L_2,L_3\ge 0}\indic{n\ge \vec{M}_3 +\vec{L}_3+2}(L_1\vee 1)^{-(d-6)/2}\\
&\phantom{C(L^2)^3  \sum_{M_1,M_2,M_3\ge 0}\sum_{L_1,L_2,L_3\ge 0}\indic{n\ge \vec{M}_3 }}\times(L_2\vee 1)^{-(d-6)/2}(L_3\vee 1)^{-(d-6)/2},
 \end{align*}
 which is at most $CL^6 n^3$ when $d>8$.
\end{proof}

\subsection{Proof of Lemma \ref{lem:third}}
In this section we prove Lemma \ref{lem:third} by using \eqref{tn2}.

Note first that the case $N=1$, corresponding to the $\pi_n(x)$ term in \eqref{tn2},  contributes $\sum_x |x|^6 \pi_n(x)$, which is the part of $\Xi_n^{[1]}[(3)]$ (in \eqref {Xi2}) corresponding to the case $M_1=0$, $L_1=n$.

For fixed $N>1$ and with $y_N=x$ and $y_0=o$, and $z_i=y_i-y_{i-1}$  our starting point is
\eqref{6_1}-\eqref{6_4}.  As noted after \eqref{6_4} we split the various sums over indices into sums over {\em distinct} indices.
Doing this split into sums over distinct indices we can write \eqref{6_1} as
\begin{align}
\nn&\sum_{i_1=1}^N |z_{i_1}|^6+3 \sum_{i_1=1}^N\sum_{i_2=1}^N \indic{i_2\ne i_1}|z_{i_1}|^4|z_{i_2}|^2\\
&\phantom{\sum_{i_1=1}^N |z_{i_1}|^6}+\sum_{i_1=1}^N\sum_{i_2=1}^N\sum_{i_3=1}^N \indic{i_1,i_2,i_3 \text{ distinct}}|z_{i_1}|^2|z_{i_2}|^2|z_{i_3}|^2.\label{611}
\end{align}

Until further notice we consider the contribution from the first term of \eqref{611} with $i=i_1$, i.e. the contribution from $\sum_{i=1}^N |y_i-y_{i-1}|^6$, where $N>1$.  This contribution is equal to
\begin{align*}
&\sum_{N=2}^{n+1}\sum_{y_N}\sum_{\substack{\vec{m}_{N}\ge 0:\\  \vec{m}_N+N-1=n}} 
\sum_{i=1}^N\sum_{\vec{y}_{N-1}}|y_i-y_{i-1}|^6\\
&\qquad \sum_{\vec{\omega}_n:o \overset{\vec{\ms}_N,\vec{y}_N}{\ra} y_N}W(\vec{\omega}_n)\sum_{\vec{R}_n \ni \vec{\omega}_n}W(\vec{R}_{n})  \prod_{\ell=1}^N J_{[\ms_{\ell-1}+{\ell-1},\ms_\ell+{\ell-1}]}.
\end{align*}

Consider the contribution from $i=N>1$.  The sum over $\vec{\omega}_n$ can be expressed as a sum over 
$\vec{\omega}^{\sss\{1\}}_{\ms_{N-1}+N-1}:o \overset{\vec{\ms}_{N-1},\vec{y}_{N-1}}{\ra} y_{N-1}$ and a sum over $\vec{\omega}^{\sss\{1'\}}_{n-\vec{\ms}_{N-1}-(N-1)}: y_{N-1} \ra y_N$,
and note that the weight $W(\vec{\omega})$ factors as $W(\vec{\omega}^{\sss\{1\}})W(\vec{\omega}^{\sss\{1'\}})$.  Now let $L_1=m_N$ and note that $n-\ms_{N-1}-(N-1)=L_1$ so the sum over $\vec{\omega}^{\sss\{1'\}}$ can be written as
\[ \sum_{\vec{\omega}^{\sss\{1'\}}_{L_1}: y_{N-1} \ra y_N}W(\vec{\omega}^{\sss\{1'\}}_{L_1})\sum_{\vec{R}^{\sss\{1'\}}_{L_1} \ni \vec{\omega}^{\sss\{1'\}}_{L_1}}W(\vec{R}^{\sss\{1'\}}_{L_1})J_{[0,L_1]}(\vec{R}^{\sss\{1'\}}_{L_1})=\pi_{L_1}(y_N-y_{N-1}),\]
(where we have used \eqref{pindefn} and translation invariance of the weights and $J_{[0,L_1]}$) so that
\begin{align}\label{omegamatch}
&\sum_{y_N}|y_N-y_{N-1}|^6  \sum_{\vec{\omega}^{\sss\{1'\}}_{L_1}: y_{N-1} \ra y_N}W(\vec{\omega}_{L_1}^{\sss\{1'\}})\sum_{\vec{R}_{L_1}^{\sss\{1'\}}\ni \vec{\omega}^{\sss\{1'\}}}W(\vec{R}^{\sss\{1'\}}_{L_1}) J_{[0,L_1]}(\vec{R}^{\sss\{1'\}}_{L_1})\nn\\
&=\sum_y |y|^6\pi_{L_1}(y),
\end{align}
where in the last line we have changed the sum over $y_N$ into a sum over $y=y_N-y_{N-1}$.
Let $N_1=N-1$. Then using the fact (see e.g.~\eqref{KJ1}) that 
\[\sum_{N_1=1}^{n-L_1}\sum_{\substack{\vec{m}_{N_1}\ge 0:\\  \vec{m}_{N_1}+N_1-1=n-(L_1+1)}}\prod_{\ell=1}^{N_1} J_{[\ms_{\ell-1}+{\ell-1},\ms_\ell+{\ell-1}]} =K_{[0,n-(L_1+1)]},\]
the contribution from $i=N$ can be expressed as 
\begin{align*}
\sum_{L_1=0}^{n-1}\sum_{y_{N_1}}  \sum_{\vec{\omega}^{\sss\{1\}}_{n-L_1}:o \ra  y_{N_1}}   W(\vec{\omega}_{n-L_1}^{\{1\}}) &\sum_{\vec{R}^{\sss\{1\}}_{n-(L_1+1)} \ni \vec{\omega}^{\sss\{1\}}_{n-(L_1+1)}}W(\vec{R}^{\sss\{1\}}_{n-(L_1+1)})\\
\nn&\phantom{n-(L_1-1)}\times  K_{[0,n-(L_1+1)]}  \sum_y |y|^6\pi_{L_1}(y).
\end{align*}
Now the sum over $\vec{R}_{n-(L_1+1)}^{\sss\{1\}}$ contains $n-L_1$ trees $R_0\ni o,\dots R_{n-(L_1+1)}\ni \omega_{n-(L_1+1)}$ whose interaction is encoded by $K_{[0,n-(L_1+1)]}$, however the sum over $\vec{\omega}^{\sss\{1\}}$ is over $n-L_1$ step walks (which have $n-L_1+1$ vertices), with the last (or ``extra'') step being from say  $u$ to $y_{N_1}$.  Of course $W(\vec{\omega}^{\sss\{1\}})$ is the product of the weights of the first $n-(L_1+1)$ steps and the last one (which has weight $z_{\sss D}D(y_{N'}-u)$).  Letting $M_1=n-L_1$, it follows that the contribution with $i=N$ can be rewritten as 
\begin{align}
\nn&\sum_{M_1\ge 1}\sum_{L_1=0}^{n-1}\indic{M_1+L_1=n}\sum_u \sum_{\vec{\omega}'_{M_1-1}:o \ra  u}W(\vec{\omega}'_{M_1-1})    \sum_{\vec{R}_{M_1-1}' \ni \vec{\omega}'}W(\vec{R}_{M_1-1}')  K_{[0,M_1-1]}\\
\nn&\qquad \qquad \qquad \st  \sum_{y_{N'}}z_{\sss D}D(y_{N'}-u) \sum_y |y|^6\pi_{L_1}(y)\\
\nn&=\sum_{M_1\ge 1}\sum_{L_1\ge 0}\indic{M_1+L_1=n}\sum_u \rho t_{M_1-1}(u)\st \sum_{y'}z_{\sss D}D(y'-u) \st \sum_y |y|^6\pi_{L_1}(y)\\
\nn&\phantom{=\sum_{M_1\ge 1}\sum_{L_1\ge 0}\indic{M_1+L_1=n}\sum_u \rho t_{M_1-1}(u)}\text{(by \eqref{tnformula})}\\
\label{extrastep}&=\sum_{M_1\ge 1}\sum_{L_1\ge 0}\indic{M_1+L_1=n}\tplus{0}{M_1}\st \ppp{3}{L_1},
\end{align}
which is equal to the $M_1\ge 1$ contribution on the right hand side of \eqref{Xi2} with $r=1$ and $A=(3)$.

Next we show that the contribution from $i<N$ (where $N>1$) gives the right hand side of \eqref{Xi1} with $r=1$ and $A=(3)$.  Consider first the contribution from $i=1$.  With a similar derivation as above, this contribution can be written as
\begin{align*}
&\sum_{L_1=0}^{n-1}\sum_y |y|^6 (\pi_{L_1}*z_{\sss D}D)(y)\st \sum_x \rho t_{n-(L_1+1)}(x)\\
&=\sum_{L_1=0}^{n-1}\pplus{3}{L_1}\st \ttt{0}{n-(L_1+1)}.
\end{align*}
This is the $M_1=0$ case of the right hand side of \eqref{Xi1} with $r=1$ and $(A)=3$.

Finally, consider the contribution where $i\ne 1,N$ (so in particular $N>2$, so this term only appears if $n\ge 2$).  In this case we wish to find a simple expression for
\begin{align*}
&\sum_{N=3}^{n+1}\sum_{\substack{\vec{m}_{N}\ge 0:\\  \vec{m}_N+N-1=n}} \sum_{i=2}^{N-1}\sum_{\vec{y}_{N}}|y_i-y_{i-1}|^6\\
&\qquad \sum_{\vec{\omega}_n:o \overset{\vec{\ms}_N,\vec{y}_N}{\ra} y_N}W(\vec{\omega}_n)\sum_{\vec{R}_n \ni \vec{\omega}_n}W(\vec{R}_{n})  \prod_{\ell=1}^N J_{[\ms_{\ell-1}+{\ell-1},\ms_\ell+{\ell-1}]}.
\end{align*}
Now let $M_1=\ms_{i-1}+{i-1}$ and $L_1=m_i$, and $M_0=n-(M_1+L_1+1)$, and 
$N_1=i-1$ and $N_0=N-i$.  Then $N=N_1+1+N_0$ and the above expression can be written as 
\begin{align}
\nn&\sum_{M_1\ge 1}\sum_{L_1\ge 0}\indic{n\ge M_1+L_1+1}\sum_{N_1=1}^{M_1}\sum_{N_0=1}^{M_0+1}
\sum_{\substack{\vec{m}^{\sss\{1\}}_{N_1}\ge 0:\\  \vec{m}^{\sss\{1\}}_{N_1}+N_1-1=M_1}}
\sum_{\substack{\vec{m}^{\sss\{0\}}_{N_0}\ge 0:\\  \vec{m}^{\sss\{0\}}_{N_0}+N_0=M_0}}\sum_{\vec{y}^{\sss\{1\}}_{N_1}}\sum_{y}\sum_{\vec{y}^{\sss\{0\}}_{N_0}}\\
&\qquad \sum_{\vec{\omega}^{\sss\{1\}}_{M_1}:o\overset{\vec{\ms}^{\sss\{1\}}_{N_1},\vec{y}^{\sss\{1\}}_{N_1}}{\ra} y^{\sss\{1\}}_{N_1}}W(\vec{\omega}^{\sss\{1\}}_{M_1})              \sum_{\vec{R}^{\sss\{1\}}_{M_1-1} \ni \vec{\omega}^{\sss\{1\}}_{M_1-1}}W(\vec{R}^{\sss\{1\}}_{M_1-1})  \prod_{\ell_1=1}^{N_1} J_{[\ms^{\sss\{1\}}_{\ell_1-1}+{\ell_1-1},\ms^{\sss\{1\}}_{\ell_1}+{\ell_1-1}]}\label{2part1}\\
&\qquad |y-y^{\sss\{1\}}_{N_1}|^6\sum_{\vec{\omega}^{\sss\{1'\}}_{L_1+1}:y^{\sss\{1\}}_{N_1} \ra y}W(\vec{\omega}^{\sss\{1'\}}_{L_1+1} )        \sum_{\vec{R}^{\sss\{1'\}}_{L_1} \ni \vec{\omega}^{\sss\{1'\}}_{L_1}}W(\vec{R}^{\sss\{1'\}}_{L_1} )  J_{[0,L_1]}(\vec{R}^{\sss\{1'\}}_{L_1} ) \label{2part2}\\
&\qquad \sum_{\vec{\omega}^{\sss\{0\}} _{M_0}:y\overset{\vec{\ms}^{\sss\{0\}} _{N_0},\vec{y}^{\sss\{0\}} _{N_0}}{\ra} y^{\sss\{0\}} _{N_0}}W(\vec{\omega}^{\sss\{0\}}_{M_0} )            \sum_{\vec{R}^{\sss\{0\}}_{M_0}  \ni \vec{\omega}^{\sss\{0\}}_{M_0}}W(\vec{R}^{\sss\{0\}}_{M_0} )  \prod_{\ell_0=1}^{N_0} J_{[\ms^{\sss\{0\}}_{\ell_0-1}+{\ell_0-1},\ms^{\sss\{0\}}_{\ell_0}+{\ell_0-1}]}\label{2part3}.
\end{align}

Now note that from \eqref{KJ1} we have 
\begin{align*}
 K_{[0,M_0]}=\sum_{N_0=1}^{M_0+1}\sum_{\substack{\vec{m}^{\sss\{0\}}_{N_0}\ge 0:\\  \vec{m}^{\sss\{0\}}_{N_0}+N_0=M_0}}\prod_{\ell_0=1}^{N_0} J_{[\ms^{\sss\{0\}}_{\ell_0-1}+{\ell_0-1},\ms^{\sss\{0\}}_{\ell_0}+{\ell_0-1}]}
\end{align*}
to see that the contribution of \eqref{2part3}, including the sums over $\vec{y}^{\sss\{0\}} _{N_0}$, $\vec m_{N_0}^{\{0\}}$, and $N_0$,  is equal to (use \eqref{tnformula} in the first equality)
\begin{align}
\nn\sum_u \sum_{\vec{\omega}^{\sss\{0\}}_{M_0}:y\ra y+u}&W(\vec{\omega}^{\sss\{0\}}_{M_0})            \sum_{\vec{R}^{\sss\{0\}}_{M_0} \ni \vec{\omega}^{\sss\{0\}}_{M_0}}W(\vec{R}^{\sss\{0\}}_{M_0}) K_{[0,M_0]}\\
\label{Te1}&=\sum_u \rho t_{M_0}(u)=\tilde t^{\{0\}}_{M_0}=\tilde t^{\{0\}}_{n-(M_1+L_1+1)}.
\end{align}
The contribution of \eqref{2part2} summed over $y$ is equal to 
\begin{align*}
\sum_{y}
|y-y^{\sss\{1\}}_{N_1}|^6\sum_{\vec{\omega}^{\sss\{1'\}}_{L_1+1}:y^{\sss\{1\}}_{N_1} \ra y}W(\vec{\omega}^{\sss\{1'\}}_{L_1+1} )        \sum_{\vec{R}^{\sss\{1'\}}_{L_1} \ni \vec{\omega}^{\sss\{1'\}}_{L_1}}W(\vec{R}^{\sss\{1'\}}_{L_1})  J_{[0,L_1]}(\vec{R}^{\sss\{1'\}}_{L_1} ) 
\end{align*}
which looks similar to \eqref{omegamatch}, but now the $\vec{\omega}^{\{1'\}}$ has an extra step leading
to an additional convolution as in the discussion preceding \eqref{extrastep}, and so the above equals
\begin{align}\sum_z |z|^6 (\pi_{L_1}*z_{\sss D}D)(z)=\tilde \pi^{(3)}_{L_1,+}.\label{Te2}
\end{align}
Finally, again using \eqref{KJ1} we have 
\begin{align*}
K_{[0,M_1-1]}=\sum_{N_1=1}^{M_1}\sum_{\substack{\vec{m}^{\sss\{1\}}_{N_1}\ge 0:\\  \vec{m}^{\sss\{1\}}_{N_1}+N_1-1=M_1}}\prod_{\ell_1=1}^{N_1} J_{[\ms^{\sss\{1\}}_{\ell_1-1}+{\ell_1-1},\ms^{\sss\{1\}}_{\ell_1}+{\ell_1-1}]},
\end{align*}
so the contribution from \eqref{2part1}, summed over $\vec{y}_{N_1}^{\{1\}}$, $\vec{m}_{N_1}^{\{1\}}$ and $N_1$, is equal to ($\vec{\omega}$ again has an extra step)
\begin{align}
\nn\sum_{v}\sum_{\vec{\omega}^{\sss\{1\}}_{M_1}:o\ra v}&W(\vec{\omega}^{\sss\{1\}}_{M_1})\sum_{\vec{R}^{\sss\{1\}}_{M_1-1} \ni \vec{\omega}^{\sss\{1\}}_{M_1-1}}W(\vec{R}^{\sss\{1\}}_{M_1-1})  
K_{[0,M_1-1]}\\
&=\sum_v (t_{M_1-1}*z_{\sss D}D)(v)=\tilde t^{\{0\}}_{M_1,+}.\label{Te3}
\end{align}
We conclude from \eqref{Te1}, \eqref{Te2} and \eqref{Te3} that the total contribution from $i\ne 1,N$ gives 
\begin{align*}
\sum_{M_1\ge 1}\sum_{L_1\ge 0}\indic{n\ge M_1+L_1+1}\tplus{0}{M_1}\st \pplus{3}{L_1}\st \ttt{0}{n-(M_1+L_1+1)},
\end{align*}
which is the contribution to the right hand side of \eqref{Xi1} coming from the sum over $M_1\ge 1$ with $r=1$ and $A=(3)$. 

The sum of all terms considered thus far gives $\Xi^{[1]}_n[(3)]$, which is the first term appearing in Lemma  \ref{lem:third}.  

Let us now consider the second term appearing in \eqref{611}.  Ignoring constants this is equal to
\begin{align}
\sum_{i_1=1}^N\sum_{i_2<i_1}|y_{i_1}-y_{i_1-1}|^4|y_{i_2}-y_{i_2-1}|^2+\sum_{i_1=1}^N\sum_{i_2<i_1}|y_{i_2}-y_{i_2-1}|^4|y_{i_1}-y_{i_1-1}|^2.\label{61111}
\end{align}
Until further notice we consider the first term of \eqref{61111}.   If $N>4$ then consider the contribution from $1<i_2<i_1-1$ and $i_1<N$, i.e from $\sum_{i_1=4}^{N-1}\sum_{i_2=2}^{i_1-2}$.  Recalling \eqref{tn2}, we seek a ``simple'' expression for
\begin{align*}
&\sum_{N=5}^{n+1}\sum_{\substack{\vec{m}_{N}\ge 0:\\  \vec{m}_N+N-1=n}} \sum_{i_1=4}^{N-1}\sum_{i_2=2}^{i_1-2}\sum_{\vec{y}_{N}}|y_{i_1}-y_{i_1-1}|^4|y_{i_2}-y_{i_2-1}|^2\\
&\qquad \sum_{\vec{\omega}_n:o \overset{\vec{\ms}_N,\vec{y}_N}{\ra} y_N}W(\vec{\omega}_n)\sum_{\vec{R}_n \ni \vec{\omega}_n}W(\vec{R}_n)  \prod_{\ell=1}^N J_{[\ms_{\ell-1}+{\ell-1},\ms_\ell+{\ell-1}]}.
\end{align*}
Letting $L_1=m_{i_1}$, $L_2=m_{i_2}$, $M_2=\ms_{i_2-1}$,  and $M_1=\ms_{i_1-1}-(M_2+L_2+1)$ and $M_0=n-(\vec{M}_2+\vec{L}_2+2)$, and $N_0=N-i_1$ and $N_1=i_1-i_2-1$ and $N_2=i_2-1$ we can write the above as
\begin{align*}
&\sum_{N=5}^{n+1}\sum_{i_1=4}^{N-1}\sum_{i_2=2}^{i_1-2}\sum_{M_2,M_1\ge 1}\sum_{L_2,L_1\ge 0}\indic{\vec{M}_2+\vec{L}_2+2\le n}\\
&\qquad \sum_{\substack{\vec{m}^{\sss \{2\}}_{N_2}\ge 0:\\  \vec{m}^{\sss \{2\}}_{N_2}+N_2=M_2}}\qquad \sum_{\substack{\vec{m}^{\sss \{1\}}_{N_1}\ge 0:\\  \vec{m}^{\sss \{1\}}_{N_1}+N_1=M_1}}\qquad \sum_{\substack{\vec{m}^{\sss \{0\}}_{N_0}\ge 0:\\  \vec{m}^{\sss \{0\}}_{N_0}+N_0=M_0}}\\
&\sum_{\vec{y}^{\sss\{2\}}_{N_2}}\sum_{y_{i_2}}\sum_{\vec{y}^{\sss\{1\}}_{N_1}}\sum_{y_{i_1}}\sum_{\vec{y}^{\sss\{0\}}_{N_0}}|y_{i_1}-y^{\sss\{1\}}_{N_1}|^4|y_{i_2}-y^{\sss\{2\}}_{N_2}|^2\\
&\prod_{s=1}^2 \Bigg(\sum_{\vec{\omega}^{\sss \{s\}}_{M_s}:y^{\sss(s)}_0\overset{\vec{\ms}^{\sss \{s\}}_{N_s},\vec{y}^{\sss(s)}_{N_s}}{\ra} y^{\sss(s)}_{N_s}}W(\vec{\omega}^{\sss \{s\}}_{M_s})              \sum_{\vec{R}_{M_s-1}^{\sss \{s\}} \ni \vec{\omega}^{\sss \{s\}}_{M_s-1}}W(\vec{R}_{M_s-1}^{\sss \{s\}})  \prod_{\ell_s=1}^{N_s} J_{[\ms^{\sss \{s\}}_{\ell_s-1}+{\ell_s-1},\ms^{\sss \{s\}}_{\ell_s}+{\ell_s-1}]}\\
&\qquad \sum_{\vec{\omega}^{\sss \{s'\}}_{L_s+1}:y^{\sss(s)}_{N_s} \ra y_{i_s}}W(\vec{\omega}^{\sss \{s'\}}_{L_s+1})           \sum_{\vec{R}_{L_s}^{\sss \{s'\}} \ni \vec{\omega}^{\sss \{s'\}}_{L_s}}W(\vec{R}_{L_s}^{\sss \{s'\}})  J_{[0,L_s]}(\vec{R}_{L_s}^{\sss \{s'\}})\Bigg)\\
&\sum_{\vec{\omega}^{\sss \{0\}}_{M_0}:y^{\sss(0)}_0\overset{\vec{\ms}^{\sss \{0\}}_{N_0},\vec{y}^{\sss(0)}_{N_0}}{\ra} y^{\sss(0)}_{N_0}}W(\vec{\omega}^{\sss \{0\}}_{M_0})              \sum_{\vec{R}_{M_0}^{\sss \{0\}} \ni \vec{\omega}^{\sss \{0\}}_{M_0}}W(\vec{R}_{M_0}^{\sss \{0\}})  \prod_{\ell_0=1}^{N_0} J_{[\ms^{\sss \{0\}}_{\ell_0-1}+{\ell_0-1},\ms^{\sss \{0\}}_{\ell_0}+{\ell_0-1}]}.
\end{align*}
Instead of summing over $N$ and $i_2<i_1$ we may now sum over $N_0,N_1,N_2$ with $1\le N_s\le M_s$, for each $s=0,1,2$.  Then arguing as above using \eqref{KJ1} (and \eqref{tnexpanded}) and the definition of $\pi_n$, we see that with $u_3\equiv 0$ in the formula below, this is equal to 
\begin{align*}
&\sum_{M_2,M_1\ge 1}\sum_{L_2,L_1\ge 0}\indic{\vec{M}_2+\vec{L}_2+2\le n}\\
&\quad \sum_{u_2,u_1,z_{2},z_{1},z_0}|u_{1}-z_1|^4|u_2-z_2|^2\\
&\phantom{\quad \sum_{u_1,z_{1}}}\times \left(\prod_{s=1}^2(t_{M_s-1}*z_{\sss D}D)(z_s-u_{s+1})(\pi_{L_s}*z_{\sss D}D)(u_s-z_s)\right)t_{M_0}(z_0-u_1).
\end{align*}
After a change of variables this is equal to
\begin{align}
\sum_{M_2,M_1\ge 1}\sum_{L_2,L_1\ge 0}\indic{\vec{M}_2+\vec{L}_2+2\le n}\left(\prod_{s=1}^2\tplus{0}{M_s}\right)\sum_{v_1,v_2}|v_1|^4 |v_2|^2 \left(\prod_{\ell=1}^2 \pi_{L_\ell,+}(v_\ell)\right)\ttt{0}{M_0}.\label{phew1}
\end{align}
Recalling that $M_0=n-(\vec{M}_2+\vec{L}_2+2)$, we see that \eqref{phew1} is equal to the contribution to the right hand side of \eqref{Xi1} with $r=2$ and $A= \begin{pmatrix}2& 0\\0 & 1\end{pmatrix}$ from $M_1,M_2\ge 1$.  This completes our analysis of the cases where $1<i_2<i_1-1$ and $i_1<N$.  We can proceed similarly for the remaining cases.  The cases with $i_1=N$ are those that contribute to the right hand side of \eqref{Xi2} (with this $r$ and $A$).  The cases where $i_2=1$ provide the contribution to the right hand side of \eqref{Xi1} and \eqref{Xi2} from $M_2=0$ with this $r$ and $A$.  The cases with $i_2=i_1-1$ provide the contribution to the right hand side of \eqref{Xi1} and \eqref{Xi2} from $M_1=0$ with this $r$ and $A$.  Thus, the first term in \eqref{61111} gives the quantity 
\[\Xi^{[2]}_n\left[\begin{pmatrix}2& 0\\0 & 1\end{pmatrix}\right].\]

The second term of \eqref{61111} can be handled in exactly the same way, except for a switch of the powers on the $|y_{i_k}-y_{i_k-1}|$ terms, giving 
\[\Xi^{[2]}_n\left[\begin{pmatrix}1& 0\\0 & 2\end{pmatrix}\right].\]

More generally the derivation above applies whenever we have a term with just two distinct indices of summation.

All subsequent terms can be handled similarly:  if $i_r<i_{r-1}<\dots <i_1$ are the distinct indices of summation then the expansion yields $r$ $\pi$ terms (one for each distinct $i_k$), and we use \eqref{KJ1} to recombine the connected graphs not corresponding to any $i_k$, giving $\tilde{t}$ terms.  Terms with some  $\tilde{t}$ ``missing", or having 0 length, arise when $i_{k}=i_{k-1}-1$ or $i_1=N$, or $i_r=1$.

Consider now the third term in \eqref{611}.  This can be written as
\begin{align*}
6\sum_{i_1=1}^N\sum_{i_2<i_1}\sum_{i_3<i_2}|y_{i_1}-y_{i_1-1}|^2|y_{i_2}-y_{i_2-1}|^2|y_{i_3}-y_{i_3-1}|^2.
\end{align*}
This has $r=3$ distinct indices of summation.  Applying the expansion and resummation to this gives the contribution 
\begin{equation*}
c \Xi^{[3]}_n\left[\begin{pmatrix}1& 0 & 0\\0 & 1& 0\\0&0&1\end{pmatrix}\right].
\end{equation*}
This gives all of the contributions from \eqref{6_1}.

Consider now the contribution from \eqref{6_2}.  When we split the sums into distinct indices of summation we always get at least one of the $z_{i_k}$ appearing an odd number of times, due to the restriction that $i_3\ne i_4$.  Therefore the expansion and resummation applied to these terms gives a sum of terms of the form $\Xi^{[r]}_n[A]$ for $2\le r\le 4$ and null $A$.  

Consider now the contribution from \eqref{6_3}.  When we split the sums into distinct indices of summation we will always get at least one of the $z_{i_k}$ appearing an odd number of times except in the following cases (including relabelling of indices of summation when necessary)
\begin{align}
&\sum_{i_1=1}^N	\sum_{i_2=1}^N \indic{i_2\ne i_1}|z_{i_1}|^2(z_{i_1}\cdot z_{i_2})^2\label{6111}\\
&\sum_{i_1=1}^N	\sum_{i_2=1}^N\sum_{i_3=1}^N \indic{i_1,i_2,i_3 \text{ distinct}}|z_{i_1}|^2(z_{i_2}\cdot z_{i_3})^2.\label{6112}
\end{align}
The first setting above arises from \eqref{6_3} when exactly one of $i_2,i_3$ is equal to $i_1$, exactly one of $i_4,i_5$ is equal to $i_1$ and the remaining two indices are equal to each other but not $i_1$.  The second setting arises from \eqref{6_3} when none of the other indices are equal to $i_1$ but either $i_2=i_4$ and $i_3=i_5$ or $i_2=i_5$ and $i_3=i_4$.  
The term \eqref{6111} can be written as
\begin{align*}
\sum_{i_1=1}^N	\sum_{i_2<i_1}^N |z_{i_1}|^2(z_{i_1}\cdot z_{i_2})^2+\sum_{i_1=1}^N	\sum_{i_2<i_1}^N |z_{i_2}|^2(z_{i_1}\cdot z_{i_2})^2,
\end{align*}
and the expansion and resummation for these terms gives (respectively)
\begin{align*}
\Xi^{[2]}_n\left[\begin{pmatrix}1& 2\\0 & 0\end{pmatrix}\right], \qquad  \Xi^{[2]}_n\left[\begin{pmatrix}0& 2\\0 & 1\end{pmatrix}\right].
\end{align*}
Similarly the term \eqref{6112} can be written (up to combinatorial constants) as 
\begin{align*}
&\sum_{i_1=1}^N	\sum_{i_2<i_1}\sum_{i_3<i_2} |z_{i_1}|^2(z_{i_2}\cdot z_{i_3})^2+\sum_{i_1=1}^N	\sum_{i_2<i_1}\sum_{i_3<i_2} |z_{i_2}|^2(z_{i_1}\cdot z_{i_3})^2\\
&\quad+\sum_{i_1=1}^N\sum_{i_2<i_1}\sum_{i_3<i_2} |z_{i_3}|^2(z_{i_1}\cdot z_{i_2})^2,
\end{align*}
and the expansion and resummation for these terms gives (respectively)
\begin{align*}
\Xi^{[3]}_n\left[\begin{pmatrix}1& 0 & 0\\0 & 0& 2\\0&0&0\end{pmatrix}\right], \qquad  \Xi^{[3]}_n\left[\begin{pmatrix}0& 0 & 2\\0 & 1& 0\\0&0&0\end{pmatrix}\right], \qquad \Xi^{[3]}_n\left[\begin{pmatrix}0& 2 & 0\\0 & 0& 0\\0&0&1\end{pmatrix}\right].
\end{align*}
The latter two are identical by Lemma \ref{lem:swap}.

It remains to consider \eqref{6_4}.  The restrictions on the indices therein imply that each $z_j$ can appear either 1, 2, or 3 times depending on other relations between the indices.  Of these possibilities, the only situation for which each term appears an even number of times is that when each term appears exactly twice.  Thus, up to combinatorial constants, the term \eqref{6_4} can be written as 
\begin{align*}
\sum_{i_1=1}^N	\sum_{i_2=1}^N\sum_{i_3=1}^N \indic{i_1, i_2, i_3 \text{ distinct}}(z_{i_1}\cdot z_{i_2})(z_{i_3}\cdot z_{i_1})(z_{i_2}\cdot z_{i_3}).
\end{align*}
This can be written (up to combinatorial constants) as 
\begin{align*}
\sum_{i_1=1}^N	\sum_{i_2<i_1}\sum_{i_3<i_2}(z_{i_1}\cdot z_{i_2})(z_{i_3}\cdot z_{i_1})(z_{i_2}\cdot z_{i_3}),
\end{align*}
and the expansion and resummation for this term gives 
\begin{align*}
\Xi^{[3]}_n\left[\begin{pmatrix}0& 1 & 1\\0 & 0& 1\\0&0&0\end{pmatrix}\right].
\end{align*}
Now check that the list of all of the non-null matrices appearing above coincides with the list in the statement of the Lemma to complete the proof.

\ARXIV{\subsection{Proof of Lemma~\ref{lem:LT6} for $8<d<10$.}\label{sec:upgrade}}
\SUBMIT{\subsection{Outline of the proof of Lemma~\ref{lem:LT6} for $8<d<10$.}\label{sec:upgrade}}
\SUBMIT{As the proof requires only minor modifications to proofs in \cite[Section 5]{H08} and our proof for $d\ge 10$ above, we give only a brief sketch here.  Further details of this proof appear in \cite{Arxiv_version}.}
In the above proof of Lemma \ref{lem:LT6} for $d\ge 10$, the only place where we required $d\ge 10$ instead of $d>8$ was in bounding $\Xi^{[1]}_n[(3)]$.  There (recall \eqref{Xi3andpi} and \eqref{Xi3andpib}) we used 
\[\sum_{x}|v_1|^6|\pi_{L_1}(v_1)|\le CL^6 (L_1\vee 1)^{(12-d)/2},\]
and 
\[ \sum_v|v|^6|(\pi_m*z_DD)(v)|\le CL^6 (L_1\vee 1)^{(12-d)/2},\]
to get that
\[\Xi^{[1]}_n[(3)]\le CL^6 \sumtwo{M_1,L_1\ge 0:}{M_1+L_1\le n}(L_1\vee 1)^{(12-d)/2}\le CL^6 (n\vee 1)^{(12-d)/2+2}.\]
Therefore to prove that the bound $\sum_x|x|^6 t_n(x)\le C L^6 n^3$ holds for all $n$ (note that the left hand side is 0 if $n=0$), it is sufficient to show that in fact
\begin{align}
\label{LT6upg}
\sum_v |v|^6 |\pi_{L_1}(v)|&\le CL^6 (L_1\vee 1)^{(10-d)/2}, 
\end{align}
and
\begin{align}\label{LT6upgb}
\sum_v|v|^6|(\pi_{L_1}*z_DD)(v)|&\le CL^6 (L_1\vee 1)^{(10-d)/2}.
\end{align}
The second bound is an easy consequence of the first one, together with the bounds $\sum_x|x|^{2r}D(x)\le C_rL^{2r}$ (with $r=0,3$), \eqref{b2} (with $r=0$), and the fact that 
\[\sum_v |v|^6|(\pi_{L_1}*z_DD)(v)|\le \sum_u \sum_{v-u}C(|u|^6+|v-u|^6)|\pi_{L_1}(u)|z_DD(v-u).\]

To achieve \eqref{LT6upg}, we first use \eqref{b2}-\eqref{b3} (with $r=1,2$) together with Lemma \ref{lem:second}, as in the proof of Lemma~\ref{lem:LT6} for $d\ge 10$, to get (for $d>8$)
\begin{equation}
\sum_x |x|^4 t_n(x)\le CL^4 n^2,\qquad \text{ for all }n\ge 0.\label{t4bound}
\end{equation}
One can now proceed as in \cite[Section 5]{H08}, armed with the bound \eqref{t4bound}.

\SUBMIT{First, an application of the lace expansion allows us to extract from each connected graph on $[0,m]$ a minimally connected graph (called a lace), and as in \cite{H08} we can write 
\begin{align}
\pi_m(x)&=\sum_{N=1}^{\infty}(-1)^N\pi^{\sss(N)}_m(x),\nn
\end{align}
where $\pi^{\sss(N)}_m(x)\ge0$ is the contribution from laces containing exactly $N$ bonds, and only finitely many terms in the above sum are non-zero (in particular we get 0 if $N>m$).
Hence
\begin{equation}\label{piNbnd}
\sum_x|x|^6|\pi_m(x)|\le \sum_{N=1}^{\infty}\sum_x|x|^6\pi^{\sss(N)}_m(x).
\end{equation}
The following (in which $\vep_L$ denotes a positive quantity which approaches zero as $L$ becomes large) is a straightforward extension of Proposition~5.1 of \cite{H08} to include the $q=3$  case.A brief description of the proof of this extension, admittedly aimed at experts, is given below; further details are provided in \cite{Arxiv_version}.
\begin{LEM}
\label{lem:piNbound}
For $d>8$ and all $L$ sufficiently large the following holds:  For each $m,N \in \N$ and $q\in \{0,1,2,3\}$,   
\begin{align}
\sum_x |x|^{2q}\pi_m^{\sss(N)}(x)\le \frac{L^{2q}(C_{\ref{lem:piNbound}}\vep_L)^N}{m^{\frac{d-4}{2}-q}}.\label{ghij}
\end{align}
\end{LEM}
Now choose $L$ sufficiently large so that $C_{\ref{lem:piNbound}}\vep_L<1$ in \eqref{ghij}, and use the above bound  for $q=3$ in \eqref{piNbnd}  to get \eqref{LT6upg}, and hence complete the proof of Lemma~\ref{lem:LT6}.

Consider briefly the proof of Lemma~\ref{lem:piNbound} for $q=3$. One bounds $\sum_x|x|^6\pi^{\sss(N)}_m(x)$ in terms of lace expansion diagrams as in \cite[(5.42)]{H08}.  Figure \ref{fig:diagram} gives an example of such a diagram in the case $N=5$.
\begin{figure}
\begin{center}
\begin{tikzpicture}
\foreach \x in {0,2,4,6,8,10} {
            \foreach \y in {0,2} {
               \node at (\x,\y) [circle,inner sep=0pt,minimum size=2mm,fill=black] {};
                }} 
                
\node at (-.2,-.2) {$o$}   ; 
\draw (0,0)--(10,0);       
\draw (0,0)--(0,2);  
\draw (0,2)--(10,2);  
\draw (2,0)--(2,2);  
\draw (4,0)--(4,2);  
\draw (6,0)--(6,2);  
\draw (8,0)--(8,2);  
\draw (10,0)--(10,2);  
\draw[line width=.8mm] (0,0)--(4,0);
\draw[line width=.8mm] (4,0)--(4,2);
\draw[line width=.8mm] (4,2)--(6,2);
\draw[line width=.8mm] (6,2)--(6,0);
\draw[line width=.8mm] (6,0)--(8,0);
\draw[line width=.8mm] (8,0)--(8,2);
\draw[line width=.8mm] (8,2)--(10,2);
\node at (10.2,2.2) {$x$}   ; 
\node at (3,2) [circle,inner sep=0pt,minimum size=2mm,fill=black] {};
\node at (5,0) [circle,inner sep=0pt,minimum size=2mm,fill=black] {};
\node at (7,2) [circle,inner sep=0pt,minimum size=2mm,fill=black] {};
\node at (1,-.3) {$m_1$} ;
\node at (3,-.3) {$m_3$} ;
\node at (4-.3,1) {$m_4$} ;
\node at (5,2.3) {$m_5$} ;
\node at (6-.3,1) {$m_6$} ;
\node at (7,-.3) {$m_7$} ;
\node at (8-.3,1) {$m_8$} ;
\node at (9,2.3) {$m_9$} ;
\end{tikzpicture}
\caption[A typical diagram]{An example of a diagram arising from the lace expansion.  The bold path from $o$ to $x$ represents the backbone, which has components of fixed lengths.  The thin lines correspond to two point functions $\rho(\cdot)$ of unrestricted length.  The vertices (other than $o$ and $x$) indicate spatial locations, which, when summed over, produce $M^{(5)}_{\vec m}$.  Note that $m_2=0$ in this diagram.}
\label{fig:diagram}
\end{center}
\end{figure}
Now one assigns to each of the top path, bottom path, and backbone path in the diagram from $o$ to $x$ one factor of $|x|^2$ each.  For each of these paths we express $x$ as a sum of displacements of line segments contributing to that path, and distribute the factor $|x|^2$ to these segments accordingly.  This was also done in \cite[Proof of Proposition 5.1]{H08} for the case $q=2$ using only the top and bottom path.  When doing this one can get at most a 4th power on a segment that is part of the backbone (a segment of the backbone cannot be simultaneously in the top path and the bottom path of the diagram), and at most a 2nd power on a segment that is not part of the backbone.  This observation is key since it means that we only need bounds on 4th powers as in \eqref{t4bound} to use the proofs in \cite[Section 5]{H08} to prove Lemma \ref{lem:piNbound} for $q=3$.\qed
}
\ARXIV{
An application of the lace expansion allows us to extract from each connected graph on $[0,m]$ (where $m\in \N$) a minimally connected graph (called a lace), and as in \cite{H08} we can write 
\begin{align*}
\pi_m(x)&=\sum_{N=1}^{\infty}(-1)^N\pi^{\sss(N)}_m(x),
\end{align*}
where $\pi^{\sss(N)}_m(x)\ge0$ is the contribution from laces containing exactly $N$ bonds, and only finitely many terms in the above sum are non-zero (in particular we get 0 if $N>m$).
Hence
\begin{equation}\label{piNsum}
|\pi_m(x)|\le \sum_{N=1}^{\infty}\pi^{\sss(N)}_m(x).
\end{equation}
We use $\vep_L$ to denote a positive quantity which approaches zero as $L$ becomes large.  
The following result is an extension (to $q=3$) of \cite[Proposition 5.1]{H08}.

\begin{LEM}
\label{lem:piNbound}
For all $L$ sufficiently large the following holds:  For each $m,N \in \N$ and $q\in \{0,1,2,3\}$,   
\begin{align}
\sum_x |x|^{2q}\pi_m^{\sss(N)}(x)\le \frac{L^{2q}(C_{\ref{lem:piNbound}}\vep_L)^N}{m^{\frac{d-4}{2}-q}}.
\nn
\end{align}
\end{LEM}
The required result \eqref{LT6upg} (and hence Lemma \ref{lem:LT6}) is a trivial consequence of Lemma \ref{lem:piNbound} as we now show.
\begin{proof}[Proof of Lemma \ref{lem:LT6}]
Choose $L$ sufficiently large so that $C_{\ref{lem:piNbound}}\vep_L\le 1/2$.  Then by Lemma \ref{lem:piNbound} with $q=3$ and \eqref{piNsum} we have for $m\in \N$
\[\sum_x|x|^6\pi_m(x)\le \sum_{N=1}^\infty L^6(C_{\ref{lem:piNbound}}\vep_L)^Nm^{3-(d-4)/2}\le L^6 m^{(10-d)/2},\]
thus proving \eqref{LT6upg}, as required.
\end{proof}
In order to state a version of  \cite[Lemma 5.4]{H08} (that will be used frequently below) we
 introduce further notation as in \cite[Section 5]{H08}.   Define $h_{m}(u)$ by
\begin{align*}
h_{m}(u)=\begin{cases}
z_{\sss D}^2(D*t_{m-2}*D)(u), & \text{ if }m\ge 2\\
z_{\sss D} D(u), & \text{ if }m =1\\
\1_{\{u=o\}}, & \text{ if }m =0,
\end{cases}
\end{align*}
where $z_{\sss D}$ is the critical value of $z$ (note that this is written as $p_c$ in \cite{H08}).
Let $\rho(x)=\sum_{T\in\T_L(o):x\in T}W(T)$ (note that our constant $\rho=\rho(o)$) and for $k=2,3,4$, let $\rho^{(*k)}(\cdot)$ denote the $k$-fold convolution of $\rho(\cdot)$ with itself.

 For $q,q_i \in \{0,1,2\}$, $m,m_i \in \Z_+$ we define $s_{m,q}(x)=|x|^{2q}h_{m}(x)$.  For $l \ge 1$ we define $s^{(*l)}_{\vec{m}_l, \vec{q}_l}(x)$ to be the $l$-fold spatial convolution of the $s_{m_i,q_i}$.  Also for $r,r_i \in \{0,1\}$, let $\phi_{r}(x)=|x|^{2r}\rho(x)$.  For $l \in \{1,2,3,4\}$, let $\phi^{(*l)}_{\vec{r}_l}(x)$ denote the $l$-fold spatial convolution of the $\phi_{r_i}$ (whenever this exists for all $x$), and define $\phi^{(*0)}(x)=\1_{\{x=o\}}$. 
\begin{LEM}
\label{lem:generalpiece}
For $L$ sufficiently large the following holds:  For $l\ge 1$, $\vec{m}_l \in \Z_+^l$ (such that $m:=\sum_i m_i>0$),   $\vec{q}_l \in \{0,1,2\}^l$, and $r_i \in \{0,1\}$,  $k\in \Z_+$ such that $2(k+\sum_{i=1}^k r_i)\le 8$,
\begin{align}
\label{gpb1}
\|s_{\vec{m}_l,\vec{q}_l}^{(*l)}*\phi^{(*k)}_{\vec{r}_{k}}\|_{\infty}
&\le (mL^2)^{\sum_{i=1}^\ell q_i+\sum_{j=1}^k r_j}\sm{C_l\vep_L}{m}{\frac{d-2k}{2}},\\
\text{and }\  \|s_{\vec{m}_l,\vec{q}_l}^{(*l)}\|_1
&\le C_l (mL^2)^{\sum_{i=1}^\ell q_i}.\label{gpb2}
\end{align}
\end{LEM}
\begin{REM}
Lemma \ref{lem:generalpiece} is an upgrade of \cite[Lemma 5.4]{H08} to include $q_i=2$.  The latter  includes assumptions on $z$ and bounds holding for $m\le n$, but the result of the inductive (on $n$) approach to the lace expansion applied in \cite{H08} is that the relevant bounds and the conclusion of the Lemma hold at the critical point $z_c$ for every $m$ (when $L$ is sufficiently large).  This is the setting in which we are working.
\end{REM}

We will return to the proof of Lemma~\ref{lem:generalpiece} at the end of this section but let us now describe how Lemma \ref{lem:piNbound} is proved, given this result.  From \cite[(5.42)]{H08}, for $q \in \{0,1,2,3\}$ and $N\ge 1$,
\begin{align}
\label{piMrelation}
\sum_x |x|^{2q}\pi_m^{\sss(N)}(x)\le \sum_{\vec{m} \in \mc{H}_m^{\sss N}}\sum_x |x|^{2q}M^{\sss(N)}_{\vec{m}}(o,o,x,o),
\end{align}
where
\[\mc{H}_m^{\sss N}=\left\{\vec{m}\in \Z_+^{2N-1}:\sum_{i=1}^{2N-1} m_i=m, \; m_{2j}\ge 0, m_{2j-1}>0\right\},\]
and $M_{\vec{m}}^{\sss(N)}(a,b,x,y)$ is defined recursively by 
\begin{align}
\nn
&M^{\sss(1)}_m(a,b,x,y)\equiv h_m(x-a)\rho^{(*2)}(x+y-b),
\end{align}
and
\begin{align}
\label{MN1}
M_{\vec{m}}^{\sss(N)}(a,b,x,y)&\equiv\sum_{u,v}A_{m_1,m_2}(a,b,u,v)M_{(m_3, \dots, m_{2N-1})}^{\sss(N-1)}(u,v,x,y).
\end{align}
where
\begin{align*}
&A_{m_1,m_2}(a,b,u,v)\equiv \begin{cases}
h_{m_1}(v-a)h_{m_2}(u-v)\rho^{(*2)}(b-u), & m_2\ne 0,\\
h_{m_1}(u-a)\rho(v-u)\rho^{(*2)}(b-v), & m_2= 0.\\
\end{cases}
\end{align*}
The expression \eqref{piMrelation} is therefore a bound on $\pi$ in terms of diagrams which graphically represent the $M_{\vec m}^{\sss(N)}$ terms (see e.g.~Figures \ref{fig:diagram} and \ref{fig:diagram2}).  In fact up to constants which are of no concern, 
\eqref{piMrelation} holds if $h_m$ is replaced with $t_m$ everywhere, but as \cite{H08} works with the former, so shall we.

\begin{figure}
\begin{center}
\begin{tikzpicture}
\foreach \x in {0,2,4,6,8,10} {
            \foreach \y in {0,2} {
               \node at (\x,\y) [circle,inner sep=0pt,minimum size=2mm,fill=black] {};
                }} 
                
\node at (-.2,-.2) {$o$}   ; 
\draw (0,0)--(10,0);       
\draw (0,0)--(0,2);  
\draw (0,2)--(10,2);  
\draw (2,0)--(2,2);  
\draw (4,0)--(4,2);  
\draw (6,0)--(6,2);  
\draw (8,0)--(8,2);  
\draw (10,0)--(10,2);  
\draw[line width=.8mm] (0,0)--(4,0);
\draw[line width=.8mm] (4,0)--(4,2);
\draw[line width=.8mm] (4,2)--(6,2);
\draw[line width=.8mm] (6,2)--(6,0);
\draw[line width=.8mm] (6,0)--(8,0);
\draw[line width=.8mm] (8,0)--(8,2);
\draw[line width=.8mm] (8,2)--(10,2);
\node at (10.2,2.2) {$x$}   ; 
\node at (3,2) [circle,inner sep=0pt,minimum size=2mm,fill=black] {};
\node at (5,0) [circle,inner sep=0pt,minimum size=2mm,fill=black] {};
\node at (7,2) [circle,inner sep=0pt,minimum size=2mm,fill=black] {};
\node at (1,-.3) {$m_1$} ;
\node at (3,-.3) {$m_3$} ;
\node at (4-.3,1) {$m_4$} ;
\node at (5,2.3) {$m_5$} ;
\node at (6-.3,1) {$m_6$} ;
\node at (7,-.3) {$m_7$} ;
\node at (8-.3,1) {$m_8$} ;
\node at (9,2.3) {$m_9$} ;
\end{tikzpicture}
\caption[A typical diagram]{An example of a diagram, $M_{\vec{m}}^{\sss(5)}(o,o,x,o)$ arising from the lace expansion.  The bold path from $o$ to $x$ represents the backbone, which has components of fixed lengths.  The thin lines correspond to two point functions $\rho(\cdot)$ of unrestricted length.  The vertices (other than $o$ and $x$) indicate spatial locations which,  when summed over, produce $M^{\sss(5)}_{\vec m}$.  Note that $m_2=0$ in this diagram.}
\label{fig:diagram}
\end{center}
\end{figure}

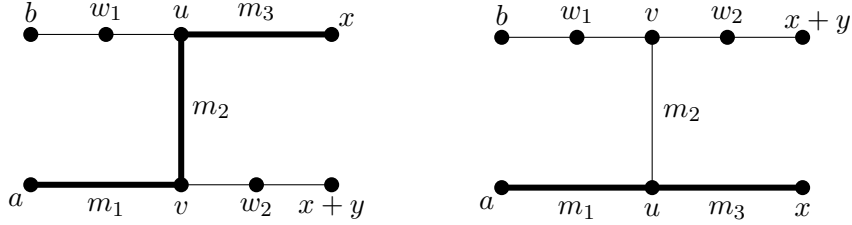
\begin{figure}
\begin{center}
\begin{tikzpicture}
\foreach \x in {0,2,4} {
            \foreach \y in {0,2} {
               \node at (\x,\y) [circle,inner sep=0pt,minimum size=2mm,fill=black] {};
                }} 
\node at (1,2)  [circle,inner sep=0pt,minimum size=2mm,fill=black] {};
\node at (3,0)  [circle,inner sep=0pt,minimum size=2mm,fill=black] {};      
\node at (-.2,-.2) {$a$}   ; 
\draw (2,0)--(2,2);  
\draw (2,0)--(4,0);  
\draw (0,2)--(2,2);  
\draw[line width=.8mm] (0,0)--(2,0);  
\draw[line width=.8mm] (2,0)--(2,2);  
\draw[line width=.8mm] (2,2)--(4,2);  
\node at (4.2,2.2) {$x$}   ; 
\node at (2,-.3) {$v$} ;
\node at (2,2.3) {$u$}   ;
\node at (1,2.3) {$w_1$}   ;
\node at (3,-.3) {$w_2$}   ;
\node at (4,-.3) {$x+y$}   ;
\node at (0,2.3) {$b$}   ;
\node at (1,-.3) {$m_1$} ;
\node at (2.4,1) {$m_2$};
\node at (3,2.3) {$m_3$} ;
\end{tikzpicture}
\hspace{1cm}
\begin{tikzpicture}
\foreach \x in {0,2,4} {
            \foreach \y in {0,2} {
               \node at (\x,\y) [circle,inner sep=0pt,minimum size=2mm,fill=black] {};
                }} 
\node at (1,2)  [circle,inner sep=0pt,minimum size=2mm,fill=black] {};
\node at (3,2)  [circle,inner sep=0pt,minimum size=2mm,fill=black] {};      
\node at (-.2,-.2) {$a$}   ; 
\draw (2,0)--(2,2);  
\draw (2,0)--(4,0);  
\draw (0,2)--(2,2);  
\draw[line width=.8mm] (0,0)--(4,0);  
\draw (2,2)--(4,2);  
\node at (4.2,2.2) {$x+y$}   ; 
\node at (2,-.3) {$u$} ;
\node at (2,2.3) {$v$}   ;
\node at (1,2.3) {$w_1$}   ;
\node at (3,2.3) {$w_2$}   ;
\node at (4,-.3) {$x$}   ;
\node at (0,2.3) {$b$}   ;
\node at (1,-.3) {$m_1$} ;
\node at (2.4,1) {$m_2$};
\node at (3,-0.3) {$m_3$} ;
\end{tikzpicture}
\caption{The diagrams $M_{\vec{m}}^{\sss(2)}(a,b,x,y)$ when $m_2>0$ (on the left) and $m_2=0$ (on the right).  In both diagrams $u,v,w_1,w_2$ are summed over.  When $m_2>0$ we have $k(\vec{m})=1$  (i.e.~there is one vertical thick line) which is odd, and the ``top path'' goes from $b$ to $x$.  When $m_2=0$ we have $k(\vec{m})=0$ (even) and the ``top path" goes from $b$ to $x+y$.}
\label{fig:diagram2}
\end{center}
\end{figure}

It is also shown in \cite[Lemma 5.6]{H08} that
\begin{align}
M_{\vec{m}}^{\sss(N)}(a,b,x,y)&=\sum_{u,v}M_{(m_1, \dots, m_{2N-3})}^{\sss(N-1)}(a,b,u,v)A_{m_{2N-1},m_{2N-2}}(x,y,u,v).\label{MN2}
\end{align}
Let $k(\vec{m})=\#\{i<N:m_{2i}>0\}$.  Given $a,b,y,x\in \Z^d$, define
\[\overline{x}=\begin{cases}
x-b, &\text{ if $k(\vec{m})$ is odd}\\
x+y-b, &\text{ if $k(\vec{m})$ is even,}
\end{cases} 
\quad \text{ and }
\quad 
\underline{x}=\begin{cases}
x-a, &\text{ if $k(\vec{m})$ is even}\\
x+y-a, &\text{ if $k(\vec{m})$ is odd.}
\end{cases}
\]
The quantities $\overline{x}$ and $\underline{x}$ represent the displacements along the top path and bottom path of a diagram respectively (see e.g.~Figure \ref{fig:diagram2}).
Then 
\begin{align}
&\sum_x |x|^6 M_{\vec{m}}^{\sss(N)}(o,o,x,o)\nn\\
&\le \sup_{a,b,y\in \Z^d}\sum_x |\overline{x}|^2 |x-a|^2|\underline{x}|^2 M_{\vec{m}}^{\sss(N)}(a,b,x,y).\label{sigh0}
\end{align}
Note that one of the supremums above is redundant by translation invariance, but it will be convenient to keep if for the purposes of discussing the inductive argument.

Let $m=\sum_{i=1}^{2N-1}m_i$.  We prove by induction on $N$ that for $\overline{j},j,\underline{j}\in \{0,1\}$
\begin{align}
\sum_{\vec{m}\in \mc{H}^N_{m}}\sup_{a,b,y\in \Z^d}\sum_x |\overline{x}|^{2\overline{j}} |x-a|^{2j}|\underline{x}|^{2\underline{j}} M_{\vec{m}}^{\sss(N)}(a,b,x,y)\le \frac{(mL^2)^{\overline{j}+j+\underline{j}}}{m^{(d-4)/2}}(C\vep_L)^N.\nn
\end{align}
This, \eqref{piMrelation}, and \eqref{sigh0} imply Lemma \ref{lem:piNbound}.  
The situation of interest in the present setting (recall \eqref{sigh0}) is where $\overline{j}=j=\underline{j}=1$, so we will focus on this one.

First let us consider the setting where $N>2$.  First break the sum over $\vec{m}$ into two sums,  depending on whether $m_1+m_2<m/2$ (so $\sum_{i=3}^{2N-1}m_i>m/2$) or not.  For the first sum  we can use 
\eqref{MN1} to decompose the diagram, and for the second we must have $m_{2N-1}+m_{2N-2}<m/2$ and we can use the same arguments but with \eqref{MN2} instead.  So let us consider only the former, i.e. that $m_{1}+m_{2}<m/2$.  Next, we split the sum based on whether $k(\vec{m})$ is even or odd (as this determines the form of $|\overline{x}|$ etc.).  Both contributions are similar, so we will consider the contribution where $k(\vec{m})$ is odd.  Now there are two cases depending on whether or not $m_2>0$.  We will consider only the contribution from $m_2>0$ as this is slightly more difficult. 

For this setting, we have that 
\[|\overline{x}|^2 |x-a|^2|\underline{x}|^2 =|x-b|^2|x-a|^2|x+y-a|^2.\]
Now using the bounds $|x-b|^2\le 2(|x-u|^2+|b-u|^2)$,  $|x-a|^2\le 2(|x-u|^2+|u-a|^2)$ and $|x+y-a|^2\le 2(|x+y-v|^2+|v-a|^2)$ (let us henceforth refer to such bounds as {\em squares inequalities}) we have that
\begin{align}
&|x-b|^2|x-a|^2|x+y-a|^2\nn\\
&\le 8(|x-u|^2+|b-u|^2)(|x-u|^2+|u-a|^2)(|x+y-v|^2+|v-a|^2).\label{8terms}
\end{align}
Expanding gives 8 terms, one of which is $8|x-u|^2|x-u|^2|x+y-v|^2$.  Note that for such $\vec{m}$ we have 
\begin{align}
&\sup_{a,b,y\in \Z^d}\sum_{u,v}\sum_x |x-u|^2|x-u|^2|x+y-v|^2\nn\\
&\phantom{\sup_{a,b,y\in \Z^d}\sum_{u,v}\sum_x} A_{m_1,m_2}(a,b,u,v)M_{(m_3, \dots, m_{2N-1})}^{\sss(N-1)}(u,v,x,y)\nn\\
&\le \sup_{a,b}\sum_{u,v}   A_{m_1,m_2}(a,b,u,v)\nn\\
& \phantom{\le \sup_{a,b}\sum_{u,v} } \sup_{y,u',v'} \sum_x|x-u'|^2|x-u'|^2|x+y-v'|^2M_{(m_3, \dots, m_{2N-1})}^{\sss(N-1)}(u',v',x,y).\label{cry1}
\end{align}
Summing \eqref{cry1}  over $\vec{m}'=(m_3,\dots, m_{2N-1})\in \mc{H}^{N-1}_{m-m_1-m_2}$ for which $k(\vec{m}')$ is even (because $m_2>0$) and using the induction hypothesis  and the fact that $m-m_1-m_2>m/2$ we get at most $c\frac{(mL^2)^3}{m^{(d-4)/2}}(C\vep_L)^{N-1}$.  What remains (summed over $m_1,m_2$) is at most 
\begin{align}
\sum_{m_1,m_2>0}\sup_{a,b}\sum_{u,v}   A_{m_1,m_2}(a,b,u,v)=\sum_{m_1,m_2>0}\sup_{a,b}(h_{m_1}*h_{m_2}*\rho^{(*2)})(b-a).\nn
\end{align}
By Lemma \ref{lem:generalpiece} this last quantity is at most
\begin{align}
\sum_{m_1,m_2>0}\frac{\vep_L}{(m_1+m_2)^{(d-4)/2}}\le c \vep_L.\nn
\end{align}
Thus the contribution from $|x-u|^2|x-u|^2|x+y-v|^2$ that we are looking at (sum over $m_1+m_2<m/2$ with $m_2>0$ and with $k(\vec{m})$ odd) satisfies the required bound.  

Consider instead the contribution from $8|b-u|^2|u-a|^2|v-a|^2$ (and $\vec{m}$ as above) in  \eqref{8terms}.  We have 
\begin{align}
&\sup_{a,b,y\in \Z^d}\sum_{u,v}\sum_x |b-u|^2|u-a|^2|v-a|^2\nn\\
&\phantom{\sup_{a,b,y\in \Z^d}\sum_{u,v}\sum_x} A_{m_1,m_2}(a,b,u,v)M_{(m_3, \dots, m_{2N-1})}^{\sss(N-1)}(u,v,x,y)\nn\\
&\le \sup_{a,b}\sum_{u,v}   |b-u|^2|u-a|^2|v-a|^2A_{m_1,m_2}(a,b,u,v)\nn\\
& \phantom{\le \sup_{a,b}\sum_{u,v} } \sup_{y,u',v'} \sum_x M_{(m_3, \dots, m_{2N-1})}^{\sss(N-1)}(u',v',x,y).\label{cry2}
\end{align}
Again, we can apply the induction hypothesis, now with $j=\underline j=\overline j=0$, to see that \eqref{cry2} summed over $\vec{m}'\in \mc{H}^{N-1}_{m-m_1-m_2}$ as above is at most $c\frac{(mL^2)^0}{m^{(d-4)/2}}(C\vep_L)^{N-1}$.  What remains is at most 
\begin{align}
\sum_{m_1,m_2>0}\sup_{a,b}\sum_{u,v}   |b-u|^2|u-a|^2|v-a|^2A_{m_1,m_2}(a,b,u,v).\nn
\end{align}
Note from the form of that the quantity $A_{m_1,m_2}(a,b,u,v)$ that the backbone (i.e.~the convolution of $h_m$ terms) goes from $a$ to $v$ to $u$.  Thus the quantity $|b-u|^2$ does not correspond to a backbone displacement, so we only get a 4th power ($4=2+2$) on the backbone.  Lemma \ref{lem:generalpiece} (together with further applications of squares inequalities) then says that this quantity is at most
\begin{align}
\sum_{\substack{m_1,m_2>0:\\
m_1+m_2<m/2}}\frac{c((m_1+m_2)L^2)^3\vep_L}{(m_1+m_2)^{(d-4)/2}}\le c m^3 L^6 \vep_L.\nn
\end{align}
Thus the contribution from $|b-u|^2|u-a|^2|v-a|^2$ that we are looking at (sum over $m_1+m_2<m/2$ with $m_2>0$ and with $k(\vec{m})$ odd) satisfies the required bound.  

The 6 other terms arising from \eqref{8terms} can be handled similarly - some of the factors are attached to $M_{(m_3, \dots, m_{2N-1})}^{\sss(N-1)}(u,v,x,y)$ and some are attached to $A_{m_1,m_2}(a,b,u,v)$ and we use the induction hypothesis on the former and Lemma \ref{lem:generalpiece} on the latter.

As noted earlier we can handle all other cases with fixed $N>2$ similarly.  The constant $C$ must be chosen sufficiently large to incorporate the constants accumulated on the way (factors of 8, bounded number of distinct cases etc.).  

The above argument also applies in the case $N=2$ if either $m_1+m_2<2m/3$ or $m_2+m_3<2m/3$.  This however can fail if $m_2$ is larger than $m_1$ and $m_3$.  In this special case $k(\vec{m})=1$ is odd since $m_2>0$, and we use \eqref{MN2} if $m_3\le m_1$ (and $m_1<m_2$) and otherwise we use \eqref{MN2} to decompose the diagram, and then we use Lemma \ref{lem:generalpiece} on each piece.  The case $N=1$ is a straightforward application of Lemma \ref{lem:generalpiece} (no decomposition is required).  Thus the proof of Lemma \ref{lem:piNbound} has been reduced to proving Lemma \ref{lem:generalpiece}, which we now address.

The bounds \eqref{gpb1}-\eqref{gpb2} are shown in \cite{H08} (see Lemmas 5.4, 5.8, 5.10 therein) to hold assuming the following two Lemmas. We introduce a parameter $\nu>0$ appearing in \cite{H08} that can be taken to be $\min\{d-8,2\}/6$ here.
\begin{LEM}
\label{lem:phiboundsused}
For all $L$ sufficiently large the following holds:  
Let $k \in \{1,2,3,4\}$ and $\vec{r}_k\in \{0,1\}^k$ satisfy $k+\sum_{i=1}^k r_i\le 4$.  Then there exists $C>0$, 
such that for every $m\ge 1$,
\begin{align*}
&\sum_{0\le |x|\le \sqrt{m}L}\phi^{(k)}_{\vec{r}^{(k)}}(x) \le Cm^{k+\sum r_j} L^{k\nu+2\sum_{j=1}^k r_j},\, \text{ and } \\
& \sup_{|x|>\sqrt{m}L}\phi^{(k)}_{\vec{r}^{(k)}}(x) \le \frac{CL^{2\sum_{j=1}^k r_j}\vep_L}{m^{\frac{d-2k-2\sum_{j=1}^k r_j}{2}}}.
\end{align*}
\end{LEM}

Note that the above is exactly \cite[Lemma 5.8]{H08}.
\begin{LEM}
\label{lem:sconvbounds}
For all $L$ sufficiently large the following holds:  
For all $l\ge 1$, $\vec{q}_l\in \{0,1,2\}^l$ and $\vec{m}_l\in \mathbb{Z}_+^l$ (such that $m:=\sum_i m_i>0$), there exists $C_l>0$ 
such that 
\begin{align*}
&\|s^{(l)}_{\vec{m}_l, \vec{q}_l}\|_\infty \le \sm{C_lL^{2\sum_{i=1}^l q_i}L^{-d} m^{\sum_{i=1}^l q_i}}{m}{\frac{d}{2}}, \quad\text{ and }\\
& \|s^{(l)}_{\vec{m}_l, \vec{q}_l}\|_1 \le C_lL^{2\sum_{i=1}^l q_i} m^{\sum_{i=1}^l q_i}.
\end{align*}
\end{LEM}
Note that here we have upgraded Lemma \ref{lem:sconvbounds} to $\vec{q}_l\in \{0,1,2\}^l$, rather than $\vec{q}_l\in \{0,1\}^l$ as appears in \cite[Lemma 5.10]{H08}.  The derivation of \eqref{gpb1}-\eqref{gpb2} (for $q_i \in \{0,1,2\}$) from these  two lemmas is the same as that for $q_i\in\{0,1\}$ in \cite[Proof of Lemma 5.4]{H08}, despite the extra power allowed here.
It therefore remains to prove Lemma \ref{lem:sconvbounds} with one or more $q_i$ equal to 2.
\begin{proof}[Proof of Lemma \ref{lem:sconvbounds}]
The proof of Lemma \ref{lem:sconvbounds} here is the same as in \cite[Proof of Lemma 5.10]{H08}, except that the extra power is handled by our bound \eqref{t4bound} which we now restate
\begin{equation}\label{4tnbound}\sum_x |x|^4 t_n(x)\le C L^4 n^2, \text{ for all }n\in\Z_+.
\end{equation}
This implies that $\sum_x |x|^4 h_n(x)\le C L^4 n^2$ holds for all $n\in\Z_+$, just as for \eqref{LT6upgb},
  which gives the $l=1$ case of the required $\|\cdot \|_1$ bound. 

Turning to the $\|\cdot \|_\infty$ bound (with $l=1$), recall from \eqref{2point0} that for $n\in\N$,
\begin{equation}
\nn
t_{2n}(x)=\rho^{-1}\sum_{\vec{\omega}_{2n}: o \ra x}W(\vec{\omega}_{2n})\sum_{\vec{R}_{2n}\ni \vec{\omega}_{2n}}W(\vec{R}_{2n})\indic{R_0,\dots, R_{2n} \text{ avoid each other}}.
\end{equation}
Let $u=\omega_n$.  Recall that a single vertex $u$ is a tree with weight $W(\{u\})=1$, and neglect the interaction between some of the $R_i$ to see that this is bounded above by
\begin{align*}
&\rho^{-1}\sum_u\sum_{\vec{\omega}_{n}: o \ra u}W(\vec{\omega}_{n})\sum_{\vec{R}_n\ni \vec{\omega}_{n}}W(\vec{R}_n)\indic{R_0,\dots, R_{n} \text{ avoid each other}}\\
&\phantom{\rho^{-1}\sum_u}
\sum_{\vec{\omega}'_{n}: u \ra x}W(\vec{\omega}'_{n})\sum_{\vec{R}'_n\ni \vec{\omega}'_{n}}\indic{R'_0=\{u\}}W(\vec{R}'_n)\indic{R'_0,\dots, R'_{n} \text{ avoid each other}}\\
&\le \rho \sum_u t_n(u)t_n(x-u),
\end{align*}
where in the last we have again used \eqref{2point0} and translation invariance.  
Similarly we can obtain for $n\in\Z_+$
\begin{equation}\label{t2n+1bound}t_{2n+1}(x)\le \rho \sum_{u,v} t_n(u)D(v-u)t_n(x-v).
\end{equation}

Thus we have for $n\in\N$
\begin{align*}
 |x|^4 t_{2n}(x)&\le \rho|x|^4 \sum_u t_n(u)t_n(x-u)\\
 &\le C\sum_u |u|^4 t_n(u)t_n(x-u)+C\sum_u  t_n(u)|x-u|^4t_n(x-u)\\
 &\le \frac{C L^{-d}}{n^{d/2}}\sum_u |u|^4 t_n(u)\\
 &\le \frac{C L^4 L^{-d}n^2}{n^{d/2}},
 \end{align*}
where we have used \eqref{4tnbound} and $\sup_x t_n(x)\le CL^{-d}n^{-d/2}$ (which holds for all $n$ by \cite[(5.68) and subsequent discussion]{H08}).  Similarly, using \eqref{t2n+1bound} we have for $n\in\Z_+$,
\begin{align*}
 |x|^4 t_{2n+1}(x)\le \rho|x|^4 \sum_{u,v} t_n(u)D(v-u)t_n(x-v)\le \frac{CL^4 L^{-d}n^2}{n^{d/2}}.
 \end{align*}
An elementary argument shows that the same bounds (up to constants) hold for $h_{m}$ ($m\in\N$) as well.  This verifies the required bounds for $l=1$.

For general $l$ we use the (induction on $l$) argument in \cite[Proof of Lemma 5.10]{H08} (but including $q_i=2$) to prove Lemma \ref{lem:sconvbounds}, and therefore complete the proof.   This approach first uses the $l=1$ case and
\begin{align*}
\|s^{(l)}_{\vec{m}_l, \vec{q}_l}\|_1\le \|s^{(1)}_{m_l, q_l}\|_1 \|s^{(l-1)}_{\vec{m}_{l-1},\vec{q}_{l-1}}\|_1,
\end{align*}
to obtain the second bound of Lemma \ref{lem:sconvbounds} (if $m_l=0$ then the above $l=1$ case doesn't apply, but we can then  use the trivial bound $\|s^{(1)}_{0, q}\|_1=\indic{q=0}$ in the above).   
To obtain the first bound, note that if $m_l>m/2$ then we use 
\begin{align*}
\|s^{(l)}_{\vec{m}_l, \vec{q}_l}\|_\infty\le \|s^{(1)}_{m_l, q_l}\|_\infty \|s^{(l-1)}_{\vec{m}_{l-1},\vec{q}_{l-1}}\|_1,
\end{align*}
and, if not, then we use
\begin{align*}
\|s^{(l)}_{\vec{m}_l, \vec{q}_l}\|_\infty\le \|s^{(1)}_{m_l, q_l}\|_1 \|s^{(l-1)}_{\vec{m}_{l-1},\vec{q}_{l-1}}\|_\infty.
\end{align*}
The case $m_\ell=0$ is again easily handled as above.  The induction hypothesis now completes the proof.
\end{proof}
The fact that we can only ever get a 4th (or lower) power on a backbone displacement is what allows the argument to work using only \eqref{t4bound} and lower powers (this corresponds to $q_i\in \{0,1,2\}$ in Lemma \ref{lem:generalpiece}). 
}

\ARXIV{
\section{Oriented Percolation}
\label{sec:OP}
Recall the definition of the model in Section \ref{sec:discussion}.  It is easy to see that $\ara$ is an ancestral relation (i.e.~(AR) holds).

In this section we verify all of the conditions for this model except Condition \ref{cond:6moment}, which is verified subject to a conjectured bound on the 6th moment of the two-point function:
\begin{CON}
\label{lem:OP6}
For $d>4$ and $L$ sufficiently large there exists $C_L$ such that for all $n\in \Z_+$,
\begin{align}
\nn
\sum_x |x|^6 \P(x \in \mc{T}_n)\le C_L n^3.
\end{align}
\end{CON}

\medskip

\noindent{\bf Condition \ref{cond:surv}:} This is immediate from \cite[Theorem 1.4]{HH13} with\\
 $m(t)=A^2V(t \vee 1)$ and $s_D=2A$ (this result was first proved in \cite{HofHolSla07a}-\cite{HofHolSla07b}).\qed

\medskip

\noindent{\bf Condition \ref{cond:L1bound}:}  This is immediate from \cite[Theorem 1.11(a)]{HofSla03b} with $k=0$.\qed

\medskip

\noindent{\bf Condition \ref{cond:self-repel}:} This is a trivial consequence of Condition \ref{cond:surv} for oriented percolation, since the event that there is an occupied path from $(s,y)$ to $(s+t,z)$ is independent of $\mc{F}_s$, and has probability $\theta(t)=\P(\mT_t \ne \varnothing)$, by the translation invariance of the model.

\medskip

We will verify Condition \ref{cond:6moment} assuming Lemma \ref{lem:OP6}, using the following Lemma (in which we again use independence of bond occupation status).
\begin{LEM}
\label{lem:reductionOP}
Let $f:\Z^d \ra \R_+$.  Then for oriented percolation (critical, spread out, in dimensions $d>4$), and $m<n\in \N$
\begin{equation}
\E\Bigl[\sum_{x\in\mT_n}\sum_{y\in \mT_{m}}\1((m,y)\to(n,x))f(x-y)\Bigr]\le \E[\mc{T}_{m}] \sum_{z\in \Z^d}f(z)\P(z \in \mc{T}_{n-m}).\label{OP1}
\end{equation}
\end{LEM}
\proof 
Let $\mc{C}(n,(m,z))=\{x:(m,z)\to (n+m,x)\}$.  Then the left hand side is equal to
\begin{align*}
&\sum_{x,y\in \Z^d}f(x-y)\P(y \in \mT_m,x \in \mc{C}(n-m,(m,y))) \\
&=\sum_{x,y\in \Z^d}f(x-y)\P\big(y\in  \mT_m)\P\big(x \in \mc{C}(n-m,(m,y))\big)\\
&=\sum_{x,y\in \Z^d}f(x-y)\P\big(y\in \mT_m)\P\big(x-y \in \mc{C}(n-m,(0,o))\big)\\
&=\sum_y \P\big(y\in  \mT_m\big) \sum_z f(z) \P\big(z \in \mc{C}(n-m,(0,o))\big)\\
&=\E[\mc{T}_{m}]\sum_z f(z) \P\big(z \in \mc{T}_{n-m}\big),
\end{align*}
as claimed.\qed

\medskip

\noindent{\bf Condition \ref{cond:6moment}:} Let $f(x)=|x|^6$, $p=6$, $n=t$, and $m=s$.  Then the left hand sides of \eqref{cond6momentv1} and \eqref{OP1} are identical.  By Lemma \ref{lem:reductionOP} the left hand side of \eqref{cond6momentv1} is at most 
\[\E[\mc{T}_{m}] \sum_{z\in \Z^d}|z|^6\P(z \in \mc{T}_{n-m}).\]
By Condition \ref{cond:L1bound} and Lemma \ref{lem:OP6} this is at most $cL^6 n^{6/2}$  and so Condition~\ref{cond:6moment} is verified for $p=6$.\qed

\medskip

\noindent{\bf Condition  \ref{cond:smallinc}:} This is immediate for any $\kappa>4$ when our steps are within a box of size $L$ (see Remark \ref{smallincdisc}). Note that, more generally, the left hand side of \eqref{smallincbnd} is (by independence and translation invariance),
\begin{align*}
\P(\exists (t,x)\text{ s.t. }t\in [0,2], x \in \mT_{t},|x|\ge N)
&\le \sum_{|x|\ge N}\P(x \in \mT_1)+\sum_{|x|\ge N}\P(x' \in \mT_2)\\
&\le p_c  \sum_{|x|\ge N} D(x)+ p_c^2  \sum_{|x|\ge N} D^{(*2)}(x)\\
&\le C  \sum_{|x|\ge N/2}D(x)\\
&\le C \frac{\sum_x |x|^{4+\vep}D(x)}{N^{4+\vep}},
\end{align*}
which satisfies the required bound with $\kappa=4+\vep$ provided that $D$ has $4+\vep$ finite moments for some $\vep>0$.\qed

\medskip

\noindent{\bf Condition \ref{cond:finite_int}:}  We verify the conditions of Lemma \ref{lem:int_fdd} with $(\gamma,\sigma_0^2)=(1,v)$.  The first condition holds by \cite[Theorem 1.2]{HofSla03b} together with the survival asymptotics \cite[Theorem 1.5]{HH13} and \cite[Proposition 2.4]{HolPer07}.  The second condition holds since
\[\E_n^s[X_t^{\sss(n)}(1)^p]=C_s\frac{n}{n^p}\sum_{x_1,\dots, x_p}\P\left(\cap_{i=1}^p\{x _i \in \mT_{\floor{nt}}\}\right),\]
and by \cite[Theorem 1.2]{HofSla03b} (with $\vec{k}=\vec{0}$) the sum is at most $C_{s,t^*,p}n^{p-1}$.\qed

\medskip

\noindent{\bf  Condition \ref{cond:self-avoiding}:} We use Lemma \ref{disccondsa}.  By independence of bond occupation status before time $\ell$ and after time $\ell$, and translation invariance, the left hand side of \eqref{dtsa} is equal to
\begin{align*}
&\sum_x \P\big(x \in \mT_\ell, \exists x'\in \Z^d\text{ s.t. }(\ell,x)\to (\ell+m,x'),\\
&\phantom{\sum_x \P\big(} |\{(i,y):(\ell,x)\to(i,y), m+2\le i-\ell\le 2m-1\}|\le M\big)\\
&=\sum_x \P\big(x \in \mT_\ell\big)\\
&\phantom{=\sum_x }\times\P(\exists z\in \mT_m,|\{(i,y):(0,o)\to (i,y), m+2\le i\le 2m-1\}|\le M\big)\\
&=\left[\sum_x \P\big(x \in \mT_\ell\big)\right]\P\left(S^{(1)}>m,\ \sum_{i=m+2}^{2m-1}|\mT_i|\le M\right).
\end{align*}
From Condition \ref{cond:L1bound} we see that the first term is bounded by a constant as required.\qed

Having verified Conditions~\ref{cond:surv}-\ref{cond:self-avoiding}, Theorem~\ref{thm:OP} now follows from Theorems~\ref{thm:mod_con_disc}, \ref{thm:range} and \ref{thm:one-arm}, and the same arithmetic used to verify Theorem~\ref{thm:LT}.
}

\paragraph{Acknowledgements.}
The work of MH was supported by Future Fellowship FT160100166, from the Australian Research Council.  
The work of EP was supported by an NSERC Discovery grant. 
MH thanks Mathieu Merle for a helpful conversation. EP thanks Ted Cox for helpful comments on the voter model.
Both authors thank Remco van der Hofstad for suggesting one of the main ingredients for the proof of the crucial Lemma \ref{lem:LT6}.
  
\bibliographystyle{plain}
\def\cprime{$'$}

\end{document}